\documentclass[a4paper]{amsart}

\usepackage{lineno,hyperref,amsmath, amssymb,amsthm, amsfonts}
\usepackage[dvipdfmx]{graphicx}
\usepackage{color}
\usepackage{bm}
\usepackage[all]{xy} 
\usepackage{booktabs}
\usepackage{tabularx}
\newtheorem{theorem}{Theorem}[section] 
\newtheorem{lemma}[theorem]{Lemma}
\newtheorem{proposition}[theorem]{Proposition}
\newtheorem{corollary}[theorem]{Corollary}
\newtheorem{example}[theorem]{Example}

\newtheorem{claim}{Claim}

\theoremstyle{definition}
\newtheorem{remark}[theorem]{Remark}

\newcommand{\C}{\mathbb{C}}
\newcommand{\R}{\mathbb{R}}

\newcommand{\p}{\partial}

\newcommand{\onab}{\overline{\nabla}}

\newcommand{\fg}{\mathfrak{g}}

\newcommand{\abs}[1]{\lvert#1\rvert}

\begin{document}
\pagestyle{plain}

\title{Convergence of Cohomogeneity-One Lagrangian Mean Curvature Flow in Positive K\"ahler-Einstein Manifolds}

\author{Naotoshi Fujihara, Toru Kajigaya and Albert Wood}%\fnref{myfootnote}}

\address[N. Fujihara]{Department of Mathematics, Faculty of Science, Tokyo University of Science
1-3 Kagurazaka Shinjuku-ku, Tokyo 162-8601 Japan\\ 
Research Institute for Science and Technology at Tokyo University of Science,
          Division of Joint Research of Geometry and Natural Science
}
\email{fujihara@rs.tus.ac.jp}

\address[T. Kajigaya]{Department of Mathematical Sciences, Shibaura Institute of Technology
307 Fukasaku, Minuma-ku, Saitama-shi, Saitama 337-8570, Japan\\
Research Institute for Science and Technology at Tokyo University of Science,
          Division of Joint Research of Geometry and Natural Science. }
\email{kajigaya@sic.shibaura-it.ac.jp}

\address[A. Wood]{Department of Mathematics, King's College London, Strand, London WC2R 2LS}
\email{albert.1.wood@kcl.ac.uk}

%\fntext[myfootnote]{}




\date{\today}

\begin{abstract}
We prove that Lagrangian mean curvature flow starting from a closed, embedded, cohomogeneity-one Lagrangian in a closed, positive Kähler-Einstein manifold exists for all time, remains embedded, and converges smoothly and graphically to a minimal Lagrangian, under natural exactness and regularity assumptions.

We also study the generalised Lagrangian mean curvature flow of Behrndt in K\"ahler manifolds which are almost-Einstein in the sense that the Ricci form satisfies $\rho = C\omega + ndd^cf$, and which satisfy $C>0$. With analogous assumptions on the flow, we obtain subconvergence to an $f$-minimal Lagrangian submanifold, with an upgrade to smooth graphical convergence in the case that $(M,g,f)$ is analytic. 
%To the authors' knowledge, this is the first long-time existence and convergence result for the flow that does not make a smallness assumption on the initial condition.

The proof proceeds by first reducing the flow to a weighted curve shortening flow on a compact two-dimensional orbifold. We then establish both a Grayson-type description of finite-time singularities and a long-time subconvergence theorem for weighted curve shortening flow in the orbifold setting. Finally, we use a \L{}ojasiewicz--Simon inequality argument to upgrade to smooth convergence of the full flow.
\end{abstract}

\maketitle

\section{Introduction}

Given a K\"ahler manifold $(M^{2n}, J, \omega, g)$, a \textit{Lagrangian submanifold} is an $n$-dimensional submanifold $L^n \subset M^{2n}$ on which the symplectic form $\omega$ vanishes. Given such a submanifold, a natural question is whether there exists a canonical representative for that Lagrangian in the same isotopy class. 

One well-studied sub-class of K\"ahler manifolds is the \textit{Calabi--Yau manifolds}, which may be characterised as those admitting a nonzero parallel holomorphic volume form $\Omega \in \Gamma(\Lambda^{n,0}M)$; this condition implies that the Ricci form $\rho$ vanishes. The natural candidate for a canonical Lagrangian representative in this setting would be a \textit{special Lagrangian}, which may be defined either as a minimal Lagrangian (i.e.\ with vanishing mean curvature), or equivalently as a submanifold calibrated by $\text{Re}(e^{-i\overline\theta}\Omega)$ for some $\overline\theta \in \mathbb{R}$. Since special Lagrangians are calibrated, it follows from work of Harvey--Lawson \cite{harvey_calibrated_1982} that they are volume-minimising in their homology class. 

It was shown by Smoczyk that the mean curvature flow (the negative gradient flow for the volume functional on the space of submanifolds) preserves the Lagrangian condition in this setting, and it is therefore natural to hope that it may flow any Lagrangian to a special Lagrangian representative. Refined versions of this conjecture are now referred to as the Thomas-Yau conjecture, see \cite{thomas_special_2002, joyce_conjectures_2015, li_thomasyau_2025} for the development of the conjectures. \textit{Lagrangian mean curvature flow} (henceforth \textit{LMCF}) has since become an important and active field of study.

More generally, the mean curvature flow also preserves the Lagrangian condition in \textit{K\"ahler-Einstein manifolds}, i.e.\ manifolds for which $\rho = C \omega$ for a constant $C$. In this case, the natural representatives for Lagrangian homology classes are minimal Lagrangians. 
Note that not all Lagrangians can be expected to converge to a minimal representative; for example among closed embedded curves in the round sphere, precisely those which divide the total area of the sphere in half will converge to a geodesic. In general, defining the \textit{mean curvature one-form} by $\alpha_H := \omega(H, \cdot)$, it can be shown that under mean curvature flow, the mean curvature one-form at time $t$ belongs to the rescaled cohomology class $e^{Ct}[\alpha_H(0)]$. Since $\alpha_H$ vanishes for minimal Lagrangians, in the case where $C\geq 0$ only Lagrangian embeddings for which $\alpha_H$ is \textit{exact}  can possibly converge to a minimal representative (In the case of Calabi-Yau, this is known as the {\it zero-Maslov condition}). 

Even under this necessary condition, however, long-time existence and convergence results for LMCF  remain rare;
historically, such results have typically assumed a graphicality or convexity condition (as in work of Smoczyk--Wang \cite{smoczyk_mean_2002}, Chau--Chen--He \cite{chau_lagrangian_2012} and Tsai--Tsui--Wang \cite{tsai_mean_2024}, \cite{tsai_entire_2025}) or closeness to the minimal example (as in work of Li \cite{li_convergence_2012}, Tsai--Wang \cite{tsai_mean_2018} and Lee--Tsai \cite{lee_dynamical_2024}).
Furthermore, it is known that in general, LMCF may develop finite-time singularities, which are of the more pathological `Type II' in many situations (see Chen-Li \cite{CL}, Neves \cite{Neves} and Wang \cite{Wang}). In recent years, there has been substantial progress in the analysis of singularities of LMCF especially in a Calabi-Yau manifold; see, for example, Joyce \cite{joyce_conjectures_2015}, Lotay--Schulze--Sz\'ekelyhidi \cite{LSS}, and Li--Sz\'ekelyhidi \cite{LiSz}. Nevertheless, the formation of finite-time singularities makes establishing long-time existence for LMCF a non-trivial problem in general.

%\subsection{Symmetric Lagrangian Submanifolds}
\subsection{A Main Result}

In this work, our main goal is to prove a long-time existence and convergence result for Lagrangian mean curvature flow on a compact K\"ahler--Einstein manifold $(M^{2n}, J, \omega, g)$ with positive Einstein constant $C>0$, so that $\rho = C\omega$. To make the problem tractable, we will assume that our Lagrangians are symmetric with respect to a group action of a compact Lie group acting by holomorphic isometries. Symmetry assumptions in the context of LMCF have been considered before, for example in the works of Pacini on the flow of Lagrangian orbits \cite{Pacini}, Anciaux on $O(n)$-invariant solitons \cite{anciaux_construction_2006}, Madnick--W. on cohomogeneity-one flows in $\mathbb{C}^n$ \cite{wood_singularities_2024, MW}, Evans on symmetric tori in $\mathbb{CP}^2$ \cite{evans_lagrangian_2022} and Lotay--Oliveira on $S^1$-invariant flows in the Gibbons--Hawking ansatz \cite{lotay_special_2024,lotay_neck_2025}.

The key advantage of working in the symmetric setting is the following dimension reduction. 
Firstly, in our setting we assume that a compact Lie group $G$ acts on $M$ as a \textit{Hamiltonian group action}, with a moment map $\mu:M \to \mathfrak{g}^*$. It is known that any $G$-invariant Lagrangian $L$ must lie in a level set of $\mu$, i.e.\ there exists $\xi \in \mathfrak{g}^*$ for which $L \subset \mu^{-1}(\xi)$. Moreover, if $\xi$ is a regular value of $\mu$, then $G$ acts on $\mu^{-1}(\xi)$ locally freely, and by taking the quotient of this level set by the group action, we obtain a {\it K\"ahler quotient} $\mu^{-1}(\xi)/G$. If furthermore $G$ acts on $L$ freely, then the Lagrangian submanifold $L$ also reduces to a Lagrangian submanifold $L/G$ in $\mu^{-1}(\xi)/G$.

Secondly, when $M$ is a positive K\"ahler-Einstein manifold and $G\subset {\rm Aut}(M,J,\omega)$,  there exists a canonical choice of the moment map $\mu_{\rm can}: M\to \fg^*$ (see Futaki \cite{Futaki} or  eq.\eqref{eq-mucan} for the explicit formula established by  Podest\`a \cite{Podesta} and the second author \cite{K}). Moreover,  the second author proved in \cite{K} that any $G$-invariant Lagrangian submanifold $L\subset M$ with exact mean curvature form must lie  in the $0$-level set of $\mu_{\rm can}$ whenever $G$ acts on $L$ freely. Since the LMCF preserves the exactness of the mean curvature form,  the LMCF starting from such an $L$ remains contained in $\mu_{\rm can}^{-1}(0)$. Therefore, by taking the quotient of the $0$-level set, the LMCF corresponds to a flow of the quotient Lagrangian $L/G \subset \mu_{\mathrm{can}}^{-1}(0)/G$ in place of $L$. 
A notable feature of the setting considered here is that the symmetry reduction of the LMCF described above naturally fits into the framework of K\"ahler reduction. See Section \ref{sec:glmcf} for details.

If we further assume that the action of $G$ on $L$ is \textit{cohomogeneity-one}, i.e. that the principal orbits are $(n-1)$-dimensional, then the lower dimensional problem becomes a flow of curves in $\mu_{\mathrm{can}}^{-1}(0)/G$. 
However, we remark that, in general, the quotient $\mu_{\mathrm{can}}^{-1}(0)/G$ need not be a smooth manifold and may have singular points, which complicates the analysis of the reduced flow.

Our first main result, concerning long-time existence and convergence of cohomogeneity-one LMCF in K\"ahler-Einstein manifolds, may now be stated as follows:

\begin{theorem}\label{thm:mainKE}
Let  $(M^{2n},J,\omega,g)$ be a closed K\"ahler-Einstein manifold with Einstein constant $C>0$. Let $G$ be a connected closed Lie subgroup of ${\rm Aut}(M,\omega,J)$ with canonical moment map $\mu_{\rm can}:M \to \mathfrak{g}^*$, and suppose $0$ is a regular value of $\mu_{{\rm can}}$.

Let $\varphi: L^{n}\to M$ be a $G$-equivariant Lagrangian embedding of a closed manifold $L^{n}$ such that:
\begin{itemize}
    \item[a)] The mean curvature form $\alpha_{H}$ is exact,
    \item[b)] The $G$-action on $L^{n}$ is free and cohomogeneity-one. 
\end{itemize}

Then there exists an LMCF $\{F_{t}\}_{t=0}^\infty$ starting from $F_{0} = \varphi$, such that $F_{t}: L\to M$ is a Lagrangian embedding for every $t\in [0,\infty)$. Moreover, there exists a smooth $G$-invariant minimal Lagrangian submanifold $L_\infty \subset M$ such that the image $L_t := F_t(L)$ converges smoothly and graphically to $L_\infty$.
\end{theorem}

    Note in particular that no smallness, graphicality or stability assumptions on the initial Lagrangian are necessary for Theorem \ref{thm:mainKE} to hold. To the authors' knowledge, this is the first long-time existence and convergence result for  LMCF of arbitrary dimension that does not make such assumptions on the initial condition.

Rather than directly prove this theorem, we will prove a more general result (Theorem \ref{thm:main}) regarding a generalisation of Lagrangian mean curvature flow which is valid for a much larger class of K\"ahler manifolds. In fact, it is more natural to work in this generalised setting, as it provides a unified framework for the arguments used in the proof. Theorem \ref{thm:mainKE} will then follow as a special case in the K\"ahler-Einstein setting. We will comment on the necessity of the assumptions after the statement of Theorem \ref{thm:main}.

\subsection{Generalised Lagrangian Mean Curvature Flow} In the general K\"ahler setting, the mean curvature flow does not preserve the Lagrangian condition. However, it was shown by Behrndt \cite{Behrndt} that if one considers the broader class of \textit{almost-Einstein} K\"ahler manifolds, i.e. those satisfying the condition
\[ \rho = C \omega + ndd^c f\]
for a constant $C$ and $f \in C^\infty(M)$, the class of Lagrangian submanifolds is preserved instead by \textit{$f$-mean curvature flow}:
\[ \frac{\partial F}{\partial t} = K := H - n \overline \nabla f^\perp. \]
Here, $K$ is referred to as the \textit{generalised mean curvature vector}. In the case where $M$ is compact, $C$ is determined cohomologically and $f$ is unique up to an additive constant. In this case, $K$ is therefore independent of the choice of $f$. A flow of Lagrangian submanifolds satisfying the $f$-MCF equation is referred to as a \textit{generalised Lagrangian mean curvature flow}, or GLMCF for short.

The analogues of special Lagrangians in the almost-Einstein K\"ahler setting are \textit{$f$-minimal Lagrangians}, which may be characterised either as the critical points of the weighted volume functional $\text{Vol}_f(L) = \int_L e^{nf}\,dv_L$, or Lagrangians for which the generalised mean curvature vector $K$ vanishes. 
Since the generalised Lagrangian mean curvature flow decreases the weighted volume, one may conjecture that in favourable cases GLMCF converges to an $f$-minimal Lagrangian representative. Analogously to the K\"ahler-Einstein setting, it may be shown that exactness of the \textit{generalised mean curvature one-form} $\alpha_K := \omega(K, \cdot) = \alpha_H - nd^c f$ is preserved by the flow, and so only Lagrangians with exact generalised mean curvature one-form could exhibit such a convergence. Results on long-time behaviour of GLMCF are uncommon, the most notable being the result of the second author and Kunikawa \cite{KK} which proves exponential convergence under smallness and stability assumptions.

Almost-Einstein K\"ahler manifolds arise naturally in many settings. Any K\"ahler form $\omega$ belonging to the first Chern class $2\pi c_1(M)$ on a Fano manifold is almost-Einstein with $C>0$. Gradient K\"ahler-Ricci solitons satisfying $\rho = C \omega + dd^cu$ are almost-Einstein with $f = \frac{u}{n}$. Another example is given by flat $\mathbb{C}^n$ equipped with the potential function $f(x) := \frac{1}{n}\langle T, x \rangle$ for a constant vector $T$. Since $dd^c f = 0$, this defines an almost-Einstein structure with $C = 0$, for which $f$-minimal Lagrangians are precisely translating solitons for Lagrangian mean curvature flow with velocity $T$ (see \cite{su_f_2021} for more details).

A particular case of GLMCF is the one-dimensional setting, which is also known as \textit{weighted curve shortening flow} (or $\psi$-CSF for short). Explicitly, given a Riemann surface $(M,g)$ with a weight function $\psi:M \to \mathbb{R}$, a family of immersions $\gamma:S^1 \times [0,T) \to M$ is a weighted curve shortening flow if it satisfies the equation
\[ \frac{\partial \gamma}{\partial t} \, = \, \kappa_\psi N := \kappa N - \langle \overline\nabla \psi, N \rangle N,\]
where $N$ is a unit normal to the curve and $\kappa$ is the geodesic curvature of $\gamma_t$. The weighted curve shortening flow decreases the $\psi$-length functional of the curve, defined as $L_\psi(\gamma) := \int_{S^1} e^\psi ds$. One may therefore expect solutions which exist for all time to converge to \textit{$\psi$-geodesics}: critical points of $L_\psi$ on which the $\psi$-curvature $\kappa_\psi$ vanishes. The first systematic study of weighted curve shortening flow was undertaken by Miquel and Vi\~nado-Lereu \cite{MV, miquel_type_2018}; their first work in particular includes a long-time existence and subconvergence result. The one-dimensional case is of particular importance in our work. In fact, the reduction procedure using a Hamiltonian group action works similarly in the almost-Einstein setting, and under our symmetry assumptions, the GLMCF will reduce to a weighted curve shortening flow in a K\"ahler quotient.

Our second main result, concerning long-time existence and convergence of cohomogeneity-one GLMCF in almost-Einstein K\"ahler manifolds, may be stated as follows.

\begin{theorem}\label{thm:main}
Let  $(M^{2n},J,\omega,g)$ be a closed K\"ahler manifold which is positive almost-Einstein, i.e.\ such that
$\rho=C\omega+ndd^{c}f$ for some $C>0$ and $f\in C^{\infty}(M)$. Let $G$ be a connected closed Lie subgroup of ${\rm Aut}(M,\omega,J)$ with canonical moment map $\mu_{\rm can}:M \to \mathfrak{g}^*$, and suppose $0$ is a regular value of $\mu_{{\rm can}}$.

Let $\varphi: L^{n}\to M$ be a $G$-equivariant Lagrangian embedding of a closed manifold $L^{n}$ such that:
\begin{itemize}
    \item[a)] The generalised mean curvature form $\alpha_{K}$ is exact,
    \item[b)] The $G$-action on $L^{n}$ is free and cohomogeneity-one. 
\end{itemize}

Then there exists a GLMCF $\{F_{t}\}_{t=0}^\infty$ starting from $F_{0} = \varphi$, such that $F_{t}: L\to M$ is a Lagrangian embedding for every $t\in [0,\infty)$. Moreover, there exists a subsequence $\{t_{i}\}_{i=1}^{\infty}$ with $t_{i}\to \infty$ such that the sequence of Lagrangian submanifolds $\{L_{i}=F_{t_{i}}(L)\}_{i=1}^{\infty}$ smoothly and graphically converges to a $G$-invariant $f$-minimal Lagrangian submanifold $L_{\infty}$ in $M$.

Furthermore, if $(M,g,f)$ is real analytic, then $L_t := F_t(L)$ converges smoothly and graphically to $L_\infty$ as $t \to \infty$.
\end{theorem}

The precise meaning of the smooth convergence $L_{i}\to L_{\infty}$ is given in Theorem \ref{thm:main1}. Note that Theorem \ref{thm:mainKE} corresponds to the case where $f\equiv 0$. 
In particular, the assumption of real analyticity is automatically satisfied when $M$ is K\"ahler-Einstein. Also, a closed K\"ahler manifold with positive almost-Einstein metric is precisely a Fano manifold equipped with a K\"ahler metric whose K\"ahler class is a positive multiple of $c_1(M)$. Many examples of  cohomogeneity-one Lagrangian submanifold satisfying the assumptions in Theorem \ref{thm:main1} are provided by torus invariant Lagrangian submanifolds in a toric Fano manifold. See subsection \ref{subsec:eg} for concrete examples.

We remark briefly upon the assumptions of Theorem \ref{thm:main} before sketching the proof. As already mentioned, the exactness assumption a) is necessary for cohomological reasons, else there is no possibility of convergence to an $f$-minimal Lagrangian. The freeness and cohomogeneity-one assumptions b) imply that the Lagrangian descends to a smooth curve $\gamma$ in the K\"ahler quotient. The assumption that $0$ is a regular value of the moment map $\mu_{\rm can}$ implies that the curve $\gamma$ naturally lies in an \textit{orbifold} $\mathcal{O}$, which allows for local analysis of the flow in smooth orbifold neighbourhoods. Finally, the assumption $C>0$ forces the K\"ahler quotient to be topologically $S^2$, so that the curve $\gamma$ encloses an area which can then be shown to be preserved under the flow; this will rule out the possibility of finite-time collapse.

\subsection{Overview of the Proof of Theorem \ref{thm:main}}

Theorem \ref{thm:main} is proven in Section \ref{sec:mainproof} as Theorems \ref{thm:main1} and \ref{thm:main2}. We give here a sketch of the proof for convenience.

Firstly, the assumptions that are made in Theorem \ref{thm:main} have the overall effect of reducing the study of the GLMCF $L_t \subset M$ to the study of a weighted curve shortening flow $\gamma_t$ in a 2-orbifold $\mathcal{O}$; the details are as follows. The assumption of exactness of $\alpha_K$ implies that $L \subset \mu_{\rm can}^{-1}(0)$ (Theorem \ref{thm:K}). The assumption that $0$ is a regular value implies that $\mu_{\rm can}^{-1}(0)$ is a smooth submanifold and that the $G$-action on $\mu_{\rm can}^{-1}(0)$ is locally free; this implies that the quotient $\mathcal{O} := \mu_{\rm can}^{-1}(0)/G$ is an orbifold. By K\"ahler reduction, $\mathcal{O}$ is K\"ahler; furthermore it is almost-Einstein (Proposition \ref{prop:Ricci}), and since $C>0$ it can be shown to be homeomorphic to a $2$-sphere (Proposition \ref{prop:topo}).

Since the exactness of $\alpha_K$ is preserved under GLMCF, the quotient of a GLMCF $\{F_t\}$ starting at $\varphi$ is a GLMCF $\{\gamma_t\} \subset \mathcal{O}$ (Proposition \ref{prop:redu}), and an $f$-minimal Lagrangian in $\mu_{\mathrm{can}}^{-1}(0)$ corresponds to a $\psi$-geodesic in $\mathcal{O}$, for a suitable induced weight $\psi$ on $\mathcal{O}$. The cohomogeneity-one assumption implies that $\gamma_t$ is 1-dimensional and $\mathcal{O}$ is 2-dimensional; $\gamma_t$ is therefore a weighted curve shortening flow.  
Note that even if $M$ is K\"ahler-Einstein (i.e.\ $f\equiv 0$), the reduced flow is in general a ``weighted" CSF. We also remark that, as mentioned above, the weighted CSF has been studied in \cite{MV, miquel_type_2018} when the ambient space is a smooth Riemann surface. In our setting, however, we need to deal with the flow on an orbifold, a space with singular points. More precisely, our quotient orbifold $\mathcal{O}$ may have finitely many singular points called {\it cone points} (see Appendix \ref{app-orbifold} and Section \ref{sec:2dorbifolds} for relevant notions on orbifolds).

We  require two key results relating to $\psi$-CSF in orbifolds. The first key result (proved as Theorem \ref{thm:beh}) is a Grayson-type theorem, stating that the $\psi$-length of a $\psi$-CSF $\gamma_t$ shrinks to zero at a finite time singularity.

\begin{theorem}\label{thm:introbeh}
    Let $(\mathcal{O},g, {\psi})$ be a closed, orientable $2$-dimensional weighted Riemannian orbifold and $\gamma: S^{1}\to \mathcal{O}^{{\rm reg}}$ be a smooth embedding into the regular part of $\mathcal{O}$. 
    
    If the $\psi$-CSF $\gamma_t$ starting from $\gamma$ has finite maximal existence time, then $\lim_{t\to T_{\max}}L_\psi(\gamma_t)=0$.
\end{theorem}

The main idea in the proof of Theorem \ref{thm:introbeh} is an avoidance-principle argument around a cone point. The structure of an orbifold chart around a cone point allows us to use a maximum-principle argument for the distance between two parts of the flow. More precisely, we first prove a Huisken-type distance comparison estimate. This estimate is used to show the existence of a suitable interval in $S^1$ on which the maximum-principle argument can be carried out. It then follows that the flow cannot approach a cone point of the orbifold if the $\psi$-length $L_\psi(\gamma_t)$ remains bounded below, since otherwise, in the orbifold chart, two parts of the flow corresponding to the interval would have to come together, contradicting the existence of such an interval. This shows that our flow belongs to the smooth part of the orbifold, and we may use a Grayson-type theorem to argue that the curve must collapse to a point.

The second key result (proved as Theorem \ref{thm:subcon}) essentially states that a $\psi$-CSF that exists for infinite time must subconverge to a $\psi$-geodesic:

\begin{theorem}\label{thm:introsubcon}
Let $(\mathcal{O},g, {\psi})$ be a closed, orientable $2$-dimensional weighted Riemannian orbifold, and $\gamma:[0,\infty)\times S^1 \to \mathcal{O}^{{\rm reg}}$ be a long-time solution of $\psi$-CSF contained in the regular part of $\mathcal{O}$.  Then $\kappa_{\psi}(t)$ uniformly converges to $0$, and there exists a subsequence $\{\gamma_{t_{i}}\}_{i=1}^{\infty}$ which converges to a continuous map $\gamma_{\infty}: S^{1}\to \mathcal{O}$ in $C^{0}(S^{1},\mathcal{O})$. Moreover, we have the following:
\begin{enumerate}
 \item  $\gamma_{\infty}$ does not pass through any cone point of order $\geq 3$. 
\item  If $\gamma_{\infty}$ passes through a cone point $x_0$ of order $2$, then there exists a cone point $y_0$ (where possibly $y_0=x_0$) of order $2$ such that $\gamma_\infty$ goes back and forth on a geodesic path between $x_0$ and $y_0$.
\item If $\gamma_{\infty}$ does not pass through any singular point of $\mathcal{O}$, then up to reparametrisation the convergence $\gamma_{t_i} \to \gamma_\infty$ is smooth and the image of $\gamma_\infty$ is a $\psi$-geodesic.
\end{enumerate}
\end{theorem}

To prove this theorem, we start by using the orbifold Gauss-Bonnet theorem to obtain a lower bound on the length of our flow, which is then utilised to produce a decay estimate on the curvature and its derivatives. Away from the cone points this then gives subsequential smooth convergence, which combined with a careful case analysis in the orbifold charts at cone points yields the full result.

\begin{remark}We remark that Theorems \ref{thm:introbeh} and \ref{thm:introsubcon} directly generalise existing results on curve shortening flow on surfaces, and should be of independent interest to those working with weighted curve shortening flow in orbifolds.
\end{remark}

Theorems \ref{thm:introbeh} and \ref{thm:introsubcon} together imply the main theorem as follows. Firstly, the fact that $\alpha_K$ is exact and $\mathcal{O}$ is homeomorphic to a 2-sphere can be shown to imply that the $\psi$-CSF is a Hamiltonian deformation, and therefore preserves the enclosed area. Together with Theorem \ref{thm:introbeh}, this implies that the flow can be extended for infinite time. Now, Theorem \ref{thm:introsubcon} implies that we have subsequential convergence to a $\psi$-geodesic (case (ii) can be ruled out as the $\psi$-CSF encloses a region of constant area). We can relate this result back to the original GLMCF to show subsequential convergence to an $f$-minimal Lagrangian submanifold. Finally, assuming real-analyticity of $g$ and $f$, we may employ a \L{}ojasiewicz-Simon inequality argument near the $f$-minimal Lagrangian in the limit to upgrade to smooth convergence in time, completing the theorem.

\subsection{Structure of the Paper}

The first few sections of the paper focus on $\psi$-CSF. Section \ref{sec:2dorbifolds} includes preliminaries on compact 2D orbifolds and one-parameter families of curves in such orbifolds (further background material on orbifolds is included as Appendix \ref{app-orbifold}). Section \ref{sec:csf} regards $\psi$-CSF and its long-time behaviour, in particular the proof of Theorem \ref{thm:subcon} (Theorem \ref{thm:introsubcon} above). Section \ref{sec:csf2} regards finite-time singularities of $\psi$-CSF, in particular the proof of Theorem \ref{thm:beh} (Theorem \ref{thm:introbeh} above). 

We then move on to the main area of study, generalised Lagrangian mean curvature flow. Section \ref{sec:glmcf} introduces almost-Einstein K\"ahler manifolds and GLMCF, outlines some important results regarding symmetric Lagrangians, and explores the relationship between cohomogeneity-one GLMCF and $\psi$-CSF. Section \ref{sec:mainproof} then contains the proof of the main theorems, Theorems \ref{thm:main1} and \ref{thm:main2} (Theorem \ref{thm:main} above). Some technical lemmas are relegated to Appendices \ref{app:distcomp} and \ref{app:norgra} for readability.

\subsection*{Acknowledgements}
T.K. is supported by JSPS KAKENHI Grant Number 23K03122. A.W. was partially supported by research grants from the Research Grants Council of
the Hong Kong Special Administrative Region, China [Project No.: CUHK 14304121, CUHK 14305122 and CUHK 14307123]. The authors also acknowledge the use of AI for assistance in refining some mathematical arguments and in improving the presentation of the manuscript.

\section{Preliminaries on compact 2-dimensional orbifolds}\label{sec:2dorbifolds}

Let $\mathcal{O}$ be a $2$-dimensional compact, connected, oriented orbifold without boundary (see Appendix \ref{app-orbifold} for definitions relating to orbifolds). It is known that such orbifolds arise as one of those listed in Table \ref{tb:orbi} (cf.\ \cite{KL}, \cite[Theorem 13.3.6]{Thurston}). 
By the classification of singular points of 2-dimensional orbifolds (see \cite[Proposition 13.3.1]{Thurston}), the singular set
$\mathcal{O}^{{\rm sing}}$ consists of finitely many {\it cone points}, where an element $x\in \mathcal{O}$ is called a {\it cone point} if there is an orbifold chart $(\widetilde U,G,\varphi)$ around $x$ such that $G$ is isomorphic to the cyclic group $\mathbb{Z}_{k}$ of order $k\geq 2$. Without loss of generality, we may assume
that $G=\mathbb{Z}_{k}$ acts on $\widetilde U$ by the rotation of order $k$ around the origin $0$ and $\varphi(0)=x$. We call the integer $k$ the {\it order} at the cone point $x$ (we define the order at a regular point $x$ to be $1$). 

\begin{table}[ht]
\centering
\begin{tabular}{l|l}
Type & Classification\\
\hline\hline
Bad & $S^{2}(k)$, $S^{2}(k_{1},k_{2})$ with $k_{1}<k_{2}$ \\ 
Spherical  &  $S^{2}$, $S^{2}(k,k)$, $S^{2}(2,2,k)$, $S^{2}(2,3,3)$, $S^{2}(2,3,4)$, $S^{2}(2,3,5)$\\ 
Euclidean  & $T^{2}$, $S^{2}(2,3,6)$, $S^{2}(2,4,4)$, $S^{2}(3,3,3)$, $S^{2}(2,2,2,2)$\\ 
Hyperbolic & otherwise ($\exists$ infinitely many)\\ \hline
\end{tabular}
\vspace{10pt}
\caption{List of connected, compact oriented 2-dimensional 
orbifolds without boundary. $S^2(k_{1},\ldots, k_{r})$ denotes an oriented orbifold with $|\mathcal{O}|=S^2$ and having cone points of order $k_{1},\ldots, k_{r}>1$. 
% A spherical (resp. Euclidean or hyperbolic) orbifold $\mathcal{O}$ is an orbifold obtained by $|\mathcal{O}|=S^{2}/\Gamma$ (resp. $|\mathcal{O}|=T^{2}/\Gamma$ or $\mathbb{H}^{2}/\Gamma$) for some discrete subgroup $\Gamma$.
Spherical, Euclidean and hyperbolic orbifolds are those whose orbifold universal covers are $S^2, \mathbb{R}^2$ and $\mathbb{H}^2$ respectively.
A bad orbifold is an orbifold {\it not} obtained by a quotient of smooth manifold divided by a discrete subgroup. }
\label{tb:orbi}
\end{table}

\subsection{Gauss-Bonnet for 2-orbifolds}

We call a closed subset $E\subset \mathcal{O}$ a {\it region} in $\mathcal{O}$ if $E$ is homeomorphic to a closed disk $D^{2}$; note that by compactness, a region may contain only finitely many cone points. We then define the {\it Euler characteristic} of a region $E$ as follows. We take a cell decomposition $\mathbf{C}(E)$ of $E$ so that any cone point becomes a vertex of $\mathbf{C}(E)$ and denote the set of vertices, edges and faces of $\mathbf{C}(E)$
by $\mathbf{V}, \mathbf{E}$ and $\mathbf{F}$, respectively. Then, we define
\begin{equation}\label{eq-eulercharacteristic}
\chi^{{\rm orb}}(E):=\sum_{v\in \mathbf{V}}\frac{1}{k_{v}}- \#\mathbf{E}+ \#\mathbf{F} \, = \,1- \sum_{v \in {E}^{{\rm sing}}}\left(1-\frac{1}{k_v}\right),
\end{equation}
where $k_{v}$ is the order at the vertex $v\in E\subset \mathcal{O}$ and $E^{{\rm sing}}$ is the set of cone points in $E$.
Note that we have  $\chi^{{\rm orb}}(E)\leq 1$ and the equality holds if and only if the region $E$ is contained in the regular part $\mathcal{O}^{{\rm reg}}$. We remark also that if $E$ contains at most one cone point, then we have $\chi^{{\rm orb}}(E)>0$ (otherwise, $\chi^{{\rm orb}}(E)$ possibly becomes non-positive). Indeed, if $E$ contains exactly one cone point of order $k$, we have that $\chi^{{\rm orb}}(E)=1/k>0$.

The Euler characteristic for the $2$-dimensional compact orbifold $\mathcal{O}$ is defined similarly, and we denote it by $\chi^{{\rm orb}}(\mathcal{O})$. 
By \eqref{eq-eulercharacteristic}, we have
\begin{equation}\label{eq-eulercharacteristicorbifold}
    \chi^{{\rm orb}}(\mathcal{O}) \, = \, \chi(|\mathcal{O}|) \, - \, \sum_{v \in \mathcal{O}^{{\rm sing}}}\left( 1 - \frac{1}{k_v}\right), 
\end{equation}
where $\chi(|\mathcal{O}|)$ is the Euler characteristic of the topological space $|\mathcal{O}|$. Note that in particular we have $\chi^{{\rm orb}}(\mathcal{O})\leq \chi(|\mathcal{O}|)$.
Moreover, if $\mathcal{O}$ is spherical (resp. Euclidean or hyperbolic),
then $\chi^{{\rm orb}}(\mathcal{O})>0$ (resp. $\chi^{{\rm orb}}(\mathcal{O})=0$ or $\chi^{{\rm orb}}(\mathcal{O})<0$). 

Now, assume that $\mathcal{O}$ is equipped with a Riemannian metric $g$. The {\it Gaussian curvature} $K$ of $\mathcal{O}$ is a smooth function $K:\mathcal{O}\to \R$ such that  the local lift $\widetilde{K}: \widetilde{U}\to \R$ defined on each chart $\widetilde{U}$ coincides with the Gaussian curvature with respect to the associated Riemannian metric $\widetilde{g}$ on $\widetilde{U}$.  The {\it area form} $\omega$ on $\mathcal{O}$ is a $2$-form defined 
by gluing the associated area forms $\{\widetilde{\omega}_{\alpha}\}_{\alpha\in A}$  with respect to the metric $g$.

For a $2$-form $\Omega$ on a compact $2$-dimensional orbifold $\mathcal{O}$, the integral over a chart $(\widetilde{U},G,\varphi)$ is defined by 
\[
\int_{U} \Omega:=\frac{1}{|G|}\int_{\widetilde{U}} \widetilde{\Omega},
\]
where $|G|$ is the order of the finite subgroup $G$ and $\widetilde{\Omega}$ is a $G$-invariant $2$-form defined on $\widetilde{U}$ associated with $\Omega$. 
The integral of a $2$-form over the entire orbifold $\mathcal{O}$ can then be defined by using the partition of unity for the orbifold (cf.\ \cite{KL}). The proof is similar to the non-orbifold case with minor modifications.

Using the above notation, the Gauss-Bonnet theorem on compact 2-dimensional orbifold may be stated as follows.

\begin{theorem}[Gauss-Bonnet theorem for compact $2$-orbifolds]\label{thm:GB}
Let $(\mathcal{O},g)$ be a compact, oriented 2-dimensional Riemannian orbifold without boundary.
\begin{enumerate}
\item[{\rm (1)}]
 Let $E$ be a region in $\mathcal{O}$. Suppose the boundary $\partial E$ of $E$ is given by a piecewise-smooth embedded closed curve $\gamma: [a,b]\to \mathcal{O}$ such that there is a partition $a=a_{0}<a_{1}<\ldots <a_{n-1}<a_{n}=b$ of the interval $[a,b]$ satisfying $\gamma(a_{n})=\gamma(a_{0})$ where 
 $\gamma|_{[a_{i-1},a_{i}]}$ is smooth for each $i=1,\ldots, n$. 
 Suppose that $\gamma$ is contained in the regular part $\mathcal{O}^{{\rm reg}}$ and $\gamma$ is parametrised anti-clockwise with respect to the orientation (i.e.\ $J\dot{\gamma}$ points inward of the region $E$, where $\dot{\gamma}$ is the velocity vector and $J$ is a compatible complex structure on $\mathcal{O}^{\rm reg}$ so that $\{X,JX\}$ is positively oriented for any tangent vector $X$). Then, we have
  \begin{align*}\label{eq:locGB}
\int_{E}K\, \omega+\int_{\gamma}\kappa\,ds=2\pi\chi^{{\rm orb}}(E)-\sum_{i=1}^{n}\epsilon_{i},
\end{align*}
where 
\begin{itemize}
\item $s$ is the  the arclength parameter.
\item $\kappa := g(\overline\nabla_{\partial_s}\partial_s, J\partial_s)$ is the geodesic curvature of the piecewise smooth curve $\gamma$.
\item $\epsilon_{i}$ is the exterior angle of $\gamma$ at the point $\gamma(a_{i})$ for $i=1,\ldots, n$.
\end{itemize}
\item[\rm (2)] We have
\[
\int_{\mathcal{O}}K\,\omega=2\pi \chi^{{\rm orb}}(\mathcal{O}).
\]
\end{enumerate}
\end{theorem}

\subsection{Curves on 2-orbifolds}\label{subsec:curvesonorbs}

Next, we prove some facts concerning closed curves and geodesics on a 2-dimensional Riemannian orbifold $\mathcal{O}$, which will be used in the proofs of our main results. Suppose that $\gamma: S^{1}\to \mathcal{O}^{{\rm reg}}$ is a smooth embedding into the regular part of $\mathcal{O}$. If $\gamma$ encloses a region $E_{\gamma}\subset \mathcal{O}$, namely,  the boundary $\partial E_{\gamma}$ of $E_{\gamma}$ coincides with the image of $\gamma$, then we define the {\it area} of $E_{\gamma}$ by 
\[
A(E_{\gamma}):=\int_{E_{\gamma}}\omega,
\]
where $\omega$ is the area form on $\mathcal{O}$ and the integral is defined on the orbifold $E_{\gamma}$ with boundary $\gamma$.

\begin{lemma}\label{lem:area}
Let $\{\gamma_{t}\}_{t\in [0,T)}$ be a smooth deformation of $\gamma=\gamma_{0}$. Suppose that for all $t\in [0,T)$, $\gamma_{t}: S^1\to \mathcal{O}$ is an embedded curve whose image is contained in $\mathcal{O}^{{\rm reg}}$ 
and which is the oriented boundary of a time-dependent region $E_t \subset \mathcal{O}$.
Then we have
\[
\frac{d}{dt}A(E_{t})=\int_{S^{1}}\gamma_{t}^{*} (i_{V_{t}}\omega),
\]
where $i$ denotes the interior product and $V_{t}=d\gamma_{t}/dt$.
\begin{proof}
Fix arbitrary $t_{0}\in (0,T)$. Because $E_{t}$ is homeomorphic to $D^{2}$, there exists a sufficiently small $\epsilon>0$ and an open subset $U\subset |\mathcal{O}|$ which is homeomorphic to $\R^{2}$ such that $E_{t}\subset U$ for any $t\in (t_{0}-\epsilon, t_{0}+\epsilon)$.  
Moreover, we may assume that the area form $\omega$ is exact on $U$, i.e.\ there exists a $1$-form $\eta$ defined on $U\subset|\mathcal{O}|$ such that $\omega=d\eta$. Indeed, since $U$ is an open suborbifold which is homeomorphic to $\R^{2}$,  de Rham's theorem for orbifolds (see \cite{Satake0}) tells us that the de Rham cohomology group $H^{2}(U)$ is isomorphic to the singular cohomology group $H^{2}(|U|;\R)=H^{2}(\R^{2};\R)=\{0\}$. Therefore, the closed $2$-form $\omega|_{U}$ is exact.

In particular, for any $t\in ( t_{0}-\epsilon, t_{0}+\epsilon)$, we have 
\[
A(E_{t})=\int_{E_{t}}d\eta=\int_{S^{1}}\gamma_{t}^{*}\eta
\]
by the Stokes theorem for orbifolds (see \cite{Satake0}).
Therefore, we obtain
\[
\frac{d}{dt}A(E_{t})=\int_{S^{1}}\frac{d}{dt}\gamma_{t}^{*}\eta=\int_{S^{1}}\gamma_{t}^{*}(i_{V_{t}}d\eta+d(i_{V_{t}}\eta))=\int_{S^{1}}\gamma_{t}^{*}(i_{V_{t}}\omega)
\]
as required.
\end{proof}
\end{lemma}

Since $\mathcal{O}$ is oriented, the area form $\omega$ is a non-degenerate closed $2$-form, and hence, it defines a symplectic structure on $\mathcal{O}$. Note that we have the relation $\omega(\cdot,\cdot)=g(J\cdot, \cdot)$ on  $\mathcal{O}^{\rm reg}$, where $J$ is the compatible complex structure on the oriented surface $\mathcal{O}^{\rm reg}$. A smooth deformation $\{\gamma_{t}\}_{t\in (-\epsilon,\epsilon)}$ of $\gamma=\gamma_{0}$ is called {\it Hamiltonian} if the variational vector field $V_{t}:=d\gamma_{t}/dt$ is a Hamiltonian vector field, i.e.\ if the associated $1$-form
\[
\alpha_{V_{t}}:=\gamma_{t}^{*}(i_{V_{t}}\omega)
\]
is an exact $1$-form. 
Hamiltonian deformations have the important property that they are area preserving:

\begin{proposition}\label{prop:areapre}
Let $\{\gamma_{t}\}_{t\in (-\epsilon,\epsilon)}$ be a Hamiltonian deformation of $\gamma=\gamma_{0}$.
Suppose that for all $t\in (-\epsilon,\epsilon)$, $\gamma_{t}: S^1\to \mathcal{O}$ is an embedding whose image is contained in $\mathcal{O}^{{\rm reg}}$ and which is the oriented boundary of a time-dependent region $E_t \subset \mathcal{O}$.
Then the area $A(E_{t})$ is preserved by the Hamiltonian deformation $\{\gamma_{t}\}_{t\in (-\epsilon,\epsilon)}$.
\end{proposition}

\begin{proof}
By assumption, there exists a smooth function $h_{t}\in C^{\infty}(S^{1})$ such that $\alpha_{V_{t}}=dh_{t}$ for each $t$. Then, by Lemma \ref{lem:area}, we see
\begin{align*}
\frac{d}{dt}A(E_{t})=\int_{S^{1}}\gamma_{t}^{*}(i_{V_{t}}\omega)=\int_{S^{1}}\alpha_{V_{t}}=\int_{S^{1}}dh_{t}=0.
\end{align*}
Therefore, $A(E_{t})$ is constant.
\end{proof}

For any smooth embedding $\gamma: S^{1}\to \mathcal{O}^{{\rm reg}}$ into the regular part of $\mathcal{O}$, we can construct a collar neighbourhood $U_{\epsilon}$ around $\gamma(S^{1})$ as follows. We define a smooth map $\Phi: S^{1}\times (-\epsilon,\epsilon)\to \mathcal{O}^{{\rm reg}}$ by
\[
\Phi(x,y):={\rm exp}_{\gamma(x)}(yN_{\gamma(x)}),
\]
where $N$ is a smooth unit normal vector field along $\gamma$. If $\mathcal{O}^{{\rm reg}}$ is oriented, then for sufficiently small $\epsilon>0$, $\Phi$ defines a diffeomorphism between the annulus $S^{1}\times (-\epsilon,\epsilon)$ and the collar neighbourhood $U_{\epsilon}:=\Phi(S^{1}\times (-\epsilon,\epsilon))$ of $\gamma(S^{1})$. 

The following technical lemma is a weaker version of \cite[Lemma A.2]{Gage}.

\begin{lemma}\label{lem:graph}
Suppose $\mathcal{O}^{{\rm reg}}$ is oriented and the embedding $\gamma: S^{1}\to \mathcal{O}^{{\rm reg}}$ is a geodesic. We take positive constants $a\geq 2$ and $\delta\leq \frac{\pi}{8aL(\gamma)}$, where $L(\gamma)$ is the length of $\gamma$. 

Then there exists an $\epsilon>0$ such that for any smooth embedding $c: S^{1}\to U_{\epsilon}\subset \mathcal{O}^{{\rm reg}}$ such that $|\kappa|\leq \delta$ and $L(c)\leq a\cdot L(\gamma)$,  the image of $c$ is described by a graph over the geodesic $\gamma(S^{1})$. Namely,
$
c(S^{1})=\{\Phi(x, h_{c}(x))\mid x\in S^{1}\}
$
for some smooth function $h_{c}: S^{1}\to \R$.
\end{lemma}

This lemma will be used in the proofs of Theorem \ref{thm:subcon} and the main result Theorem \ref{thm:main1}. Since the definition of $h_{c}$ will be important in the proof of Theorem \ref{thm:main1}, we briefly summarize the proof of this lemma according to \cite[Appendix A]{Gage}.

\begin{proof}[Proof of Lemma \ref{lem:graph}] 
Regarding $\Phi$ as a local coordinate (called the {\it Fermi coordinate}) of $\mathcal{O}^{{\rm reg}}$, the metric $g$ is written by $g = J^2 (dx)^2 + (dy)^2$,
where $J(x,y)$ is the Jacobi field coefficient satisfying $\p^{2}_{y}J+KJ=0$ with $J(x,0)=1$ and $\p_{y}J(x,0)=0$.
We denote the arclength parameter of the curve $c$ by $s$, and we define an angle function $\varphi(s)$ by $\cos \varphi(s)=\langle \frac{dc}{ds}(s), e_{1}\rangle$, where $e_{1}:=\frac{1}{J}\Phi_{*}\frac{\p}{\p x}$ is the  unit tangent vector to the curves $y=\text{const}.$
Then, a direct computation (see \cite[Appendix A]{Gage}) shows that the geodesic curvature $\kappa$ of $c$ is given by
\begin{align*}
\kappa=\frac{d\varphi}{ds}-\cos \varphi \cdot \frac{J_{y}}{J}.
\end{align*}
As shown in \cite[Lemma A.1]{Gage},  we have $|J_{y}/J|<\delta$ if we take sufficiently small $\epsilon>0$. Therefore, by assumption, we have
\[
\Big|\frac{d\varphi}{ds}\Big|\leq |\kappa|+\Big|\frac{J_{y}}{J}\Big|\leq 2\delta\leq \frac{\pi}{4aL(\gamma)}.
\]
Thus, we see
\[
|\varphi(s)-\varphi(0)|\leq \int_{0}^{s}\Big|\frac{d\varphi}{ds}\Big|\ ds\leq \frac{\pi}{4aL(\gamma)}\cdot L(c)\leq \frac{\pi}{4}
\]
since we assume $L(c)\leq a\cdot L(\gamma)$. Because $c$ is a closed curve, there is a point $c(s)$ such that $dc/ds(s)$ is parallel to $e_{1}$, and hence, we may assume that $\varphi(0)=0$ without loss of generality. Therefore, the angle function satisfies that $|\varphi(s)|\leq \pi/4$. 

We identify $U_{\epsilon}$ with $S^{1}\times (-\epsilon,\epsilon)$ by $\Phi$. Then $\gamma(S^{1})$ is identified with the slice $S^{1}\times \{0\}$. 
Define
$\widetilde{c}:=\Phi^{-1}\circ c$. Let $\pi_{1}: S^{1}\times (-\epsilon,\epsilon)\to S^{1}$ be the projection.  Since $d\pi_{1}(\p/\p x)={\p}/{\p x}|_{S^{1}\times \{0\}}$, $d\pi_{1}(\p/\p y)=0$, we see that
\[
(d\pi_{1})_{\widetilde{c}(s)}\Big(\frac{d\widetilde{c}}{ds}\Big)=\cos \varphi(s)\cdot \frac{1}{J}\frac{\p}{\p x}\Big|_{\pi_{1}\circ \widetilde{c}(s)}\neq 0
\]
since $|\varphi(s)|\leq \pi/4$. This implies that the restricted map $\pi_{1}|_{\widetilde{c}(S^{1})}: \widetilde{c}(S^{1})\to S^{1}$ defines a local diffeomorphism.  Moreover, by the same argument given in \cite[Lemma A.2]{Gage}, it turns out that $\pi_{1}|_{\widetilde{c}(S^{1})}$ is indeed a global diffeomorphism. Thus, we define a function $h_{c}:S^{1}\to \R$ by 
\begin{align}\label{def:h}
(\pi_{1}|_{\widetilde{c}(S^{1})})^{-1}(x)=(x, h_{c}(x)).
\end{align}
Then it is easy to see that $\widetilde{c}(S^{1})=\{(x,h_{c}(x))\mid x\in S^{1}\}$ and this implies the desired conclusion. 
\end{proof}

\section{Long-time behaviour of weighted CSF in  compact 2D-orbifolds}\label{sec:csf}

We now study the behaviour of weighted curve shortening flow in a compact, oriented 2-dimensional orbifold. In this section, we extend some fundamental results on CSF on a surface established by Gage \cite{Gage} and  Grayson \cite{Grayson} to our orbifold setting. The main result of this section is Theorem \ref{thm:subcon}, which concerns the infinite-time behaviour of the weighted CSF on the orbifold. The finite-time behaviour of the flow will be discussed in Section \ref{sec:csf2}.

% We will show in the next section that, in our setting, the analysis of the behaviour of the cohomogeneity-one GLMCF reduces to the study of the weighted curve shortening flow (CSF) on a 2-dimensional compact oriented orbifold. 
%In this section, we extend some fundamental results on CSF on a surface established by Gage \cite{Gage} and  Grayson \cite{Grayson} to our orbifold setting. Using these results, we prove two theorems concerning the finite-time and the infinite-time behaviour of the weighted CSF on a 2-dimensional compact orbifold.  The main results of this section are Theorems \ref{thm:beh} and \ref{thm:subcon}.

\subsection{Weighted CSF}
First, we summarize basic facts and formulas for the weighted curve shortening flow.  Let $(\mathcal{O}^2,g, \psi)$ be an oriented, compact 2-dimensional weighted Riemannian orbifold, where $\psi: \mathcal{O}\to \R$ is a smooth function on $\mathcal{O}$.

For our purpose, we suppose that {\it $\gamma: S^1\to \mathcal{O}^{{\rm reg}}$ is a smooth immersion from a circle $S^1$ into the regular part $\mathcal{O}^{{\rm reg}}$}.
We fix a manifold structure of $S^1$ by using a chart $\theta \mapsto (\cos \theta, \sin \theta)$, and $\theta$ will be called a \textit{parameter} of $S^1$. We denote the arclength parameter of $\gamma$ with respect to $g$ by $s$, and set $\p_s\gamma:=\partial \gamma/\partial s$.  We denote $\langle\cdot ,\cdot  \rangle:=g(\cdot,\cdot)$ and the Levi-Civita connection of $\mathcal{O}$ (resp. of $S^1$ with respect to the induced metric) by $\onab$ (resp. $\nabla$).

We define the {\it $\psi$-geodesic curvature} $\kappa_\psi$ of $\gamma$ by
\begin{align}\label{eq:kappa}
\kappa_{\psi}:=\kappa-\langle\overline{\nabla}\psi,N\rangle,
\end{align}
where $N=J\p_s$ is a unit normal vector field along $\gamma$ with respect to $g$ so that $\{\p_s, N\}$ is positively oriented, and $\kappa$ is the usual geodesic curvature which is defined by  $\kappa:=\langle \onab_{\p_s}\p_s, N\rangle$. An immersed curve $\gamma: S^{1}\to \mathcal{O}^{{\rm reg}}$ will be called {\it $\psi$-minimal} if $\kappa_{\psi}=0$ along $\gamma$. It is known that $\psi$-minimal immersion arises as a critical point of the {\it $\psi$-length functional}
\begin{equation}\label{eq:lengthfunctional}
L_{\psi}(\gamma):=\int_{S^1}ds_\psi,\quad \textup{where $ds_\psi:=e^{\psi}ds$.}
\end{equation}
Note that $L_{\psi}$ coincides with the length with respect to the weighted Riemannian metric $g_{\psi}:=e^{2\psi}g$.
Thus, the $\psi$-minimal curve is indeed a geodesic with respect to $g_{\psi}$ (we say {\it $\psi$-geodesic}) if we parametrise the curve by the arclength parameter w.r.t.\ $g_{\psi}$. 

\begin{remark}\label{rem:geocurv}
 It should be emphasised that the $\psi$-geodesic curvature $\kappa_{\psi}$ is different from  the geodesic curvature ${\kappa}'$ with respect to the weighted metric $g_\psi$. Indeed, we have the relation
$
\kappa'=e^{-\psi}\kappa_{\psi}.
$

\end{remark}

A one-parameter family of immersions $\gamma: S^1\times [0,T) \to \mathcal{O}^{{\rm reg}}$ is called a {\it weighted curve shortening flow} or {\it $\psi$-curve shortening flow} ({\it weighted CSF} or {\it $\psi$-CSF} for short) {\it contained in $\mathcal{O}^{{\rm reg}}$} if it satisfies 
\begin{align}\label{eq:CSF}
\frac{\p \gamma_{t}}{\p t}=\kappa_\psi(t) N(t),
\end{align}
where $\kappa_{\psi}(t)$ (resp. $N(t)$) is the $\psi$-geodesic curvature (resp. the unit normal vector field) of $\gamma_{t}=\gamma(t,\cdot)$.
 It follows from the first variational formula for $L_\psi$ (see also \eqref{eq:fv} below) that  $L_\psi(\gamma_t)$ is a decreasing function along $\psi$-CSF $\gamma_t$. In fact, the $\psi$-CSF is a negative gradient flow of $L_\psi$ w.r.t.\ the weighted $L^2$-norm $\langle V,W\rangle_{\psi}=\int_{S^1}\langle V,W\rangle ds_\psi$.

We also consider the following equation:
\begin{align}\label{eq:unpCSF}
\Big(\frac{\p \widehat{\gamma}_{t}}{\p t}\Big)^{\perp}=\widehat{\kappa}_{\psi}(t)\widehat{N}(t)
\end{align}
where $\perp$ means the orthogonal projection onto the normal space of $\widehat{\gamma}_{t}$, and $\widehat{\kappa}_{\psi}(t)$ (resp. $\widehat{N}(t)$) is the $\psi$-geodesic curvature (resp. the unit normal vector field)
of $\widehat{\gamma}_{t}$. We call a family of immersions $\{\widehat{\gamma}_{t}\}_{t\in [0,T)}$ satisfying \eqref{eq:unpCSF} an {\it unparametrised solution} of $\psi$-CSF. It is easy to see that if $\gamma: S^{1}\times [0,T)\to \mathcal{O}^{{\rm reg}}$ is a solution of $\psi$-CSF \eqref{eq:CSF} and  $\varphi_{t}: S^{1}\times [0,T)\to S^{1}$ is a 
smooth family of diffeomorphisms on $S^{1}$, then the family of immersions $\widehat{\gamma}_{t}:=\gamma_{t}\circ \varphi_{t}$ is an unparametrised solution. 
Note that if this is the case, the images of the flows coincide with each other.

The $\psi$-CSF satisfies the following evolution equations (see \cite[\S 2.2, eq. (10) and (14)]{MV}):
\begin{lemma}
Along the $\psi$-CSF, we have
 \begin{align}
 \label{eq:fv}
& \frac{\p}{\p t}ds_{\psi}(t)=-\kappa_\psi^2\,ds_{\psi}(t),\\
 \label{eq:kev}
&  \frac{\p \kappa_{\psi}}{\p t}=\Delta_{\psi}\kappa_{\psi}+\kappa_{\psi}(|\kappa|^{2}+{\rm Ric}_{\psi}(N,N)),
 \end{align}
 where $\Delta_\psi u=\Delta u+\langle \onab \psi, \nabla u\rangle$ for  $u\in C^\infty(S^1)$
 and ${\rm Ric}_{\psi}$ is the weighted Ricci curvature of $\mathcal{O}$ defined by using the Ricci curvature ${\rm Ric}$ of $g$ by ${\rm Ric}_{\psi}(N,N):={\rm Ric}(N,N)-\langle \onab_N\onab\psi, N\rangle$.
\end{lemma}

Moreover, similar to the usual CSF, we have the {\it avoidance principle} (see \cite[\S 2]{MV}):

\begin{proposition}\label{prop:ap}
Let $\gamma_{1},\gamma_{2}: S^{1}\to \mathcal{O}^{{\rm reg}}$ be two embedded curves
such that $\gamma_{1}(S^{1})\cap \gamma_{2}(S^{1})=\emptyset$,
and $\{\gamma_{1,t}\}_{t\in[0,T)},\{\gamma_{2,t}\}_{t \in [0,T)}$ be $\psi$-CSFs starting from $\gamma_1,\gamma_2$ respectively that remain in $\mathcal{O}^{{\rm reg}}$.
Then,  $\gamma_{1,t},\gamma_{2,t}$
do not intersect each other for all $t \in [0,T)$. 
Furthermore, the embeddedness is preserved under the $\psi$-CSF.
\end{proposition}

\subsection{Curvature estimates for long-time flows}

In this subsection, we prove curvature estimates which will be employed in our convergence result, Theorem \ref{thm:subcon}. Throughout this section, we will assume that $(\mathcal{O},g, {\psi})$ is a compact, orientable $2$-dimensional weighted Riemannian orbifold with finitely many cone points, and $\gamma:S^{1}\times [0,\infty) \to \mathcal{O}^{{\rm reg}}$ is a long-time embedded solution of $\psi$-CSF contained in the regular part of $\mathcal{O}$.

Note that, by the compactness of $\mathcal{O}$, we may assume that 
\begin{itemize}
\item There exist constants $C_{j}$, $P_{j}$, $E$, $D$ ($j=0,1,2,\ldots$) such that
\begin{align}\label{as:1}
\textup{$|\onab^{j}K|\leq C_{j}$, $|\onab^{j}\psi|\leq P_{j}$ and $0<E\leq e^{\psi}\leq D$, }
\end{align}
where $K$ is the Gaussian curvature of $(\mathcal{O},g)$. 
\item There exists a constant $A>0$ depending only on $(\mathcal{O},g)$ such that 
\begin{align}\label{as:2}
\textup{ $A(E)\leq A$ for any region $E\subset \mathcal{O}$,}
\end{align}
where $A(E)$ is the area of $E$ measured by the area form on $(\mathcal{O},g)$ (see Subsection \ref{subsec:curvesonorbs})
\end{itemize}

We first prove the following lower length bound for the flow.

\begin{lemma}[Length estimates for the flow]\label{lem:leng}
Under the assumptions of this subsection, there exists a constant 
$c>0$ such that  
\begin{align}\label{as:3}
\textup{$L(t)\geq c$ for any $t\in [0,\infty)$,}
\end{align}
 where $L(t)$ is the length of $\gamma_{t}$ with respect to the metric $g$.
\end{lemma}

\begin{proof}
We show that the argument given in \cite[Lemma 8]{MV} can be extended to our orbifold setting. We shall deduce a contradiction by assuming that ${\rm liminf}_{t\to \infty}L(t)=0$. Note that, since $L_{\psi}$ is decreasing and $e^{\psi}$ is bounded, ${\rm liminf}_{t\to \infty}L(t)=0$ implies that ${\rm lim}_{t\to \infty}L(t)=0$ (see also \cite[eq.(30)]{MV}). 
Then, our first claim is that for any sufficiently large $t$, $\gamma_{t}$ encloses a region $E_{t}\subset \mathcal{O}$  and we have $\lim_{t\to \infty}A(t)=0$, where $A(t)$ is the area of $E_{t}$.

To show this, we shall take a ``geodesic open ball'' $B(x,r_{x})$ around each cone point  $x$ as follows. Let $x\in \mathcal{O}$ be a cone point of order $k$, and $\varphi: \widetilde{U}\to \mathcal{O}$ be an orbifold chart around $x$. We may assume  that $0\in \widetilde{U}\subset \R^{2}$ and $\varphi(0)=x$. We induce a $\mathbb{Z}_{k}$-invariant metric $\widetilde{g}$ on $\widetilde{U}$ from $g$. We take a geodesic ball $\widetilde{B}(0,r_{x}):=\{{\rm exp}_{0}(tX)\mid X\in T_{0}\widetilde{U},\ 0\leq t<r_{x}\}$ so that $r_{x}>0$ is smaller than the convexity radius $\mathrm{conv}(0)$ at $0$. We then put $B(x,r_{x}):=\varphi(\widetilde{B}(0,r_{x}))$. Moreover, by taking $r_{x}$ sufficiently small for each cone point, we may assume that $B(x,r_{x})\cap B(y,r_{y})=\emptyset$ for any distinct  two cone points $x, y$.

Next, we define a closed subset $D\subset \mathcal{O}$ by the complement of the disjoint union $\sqcup_{i=1}^{N}B(x_{i},r_{x_{i}})$, where $x_{1},\ldots, x_{N}$ are all the cone points contained in $\mathcal{O}$. Then, by definition, $D\subset \mathcal{O}^{{\rm reg}}$ and $D$ is compact since so is $\mathcal{O}$. We put $r_{D}:={\rm min}_{x\in D}{\rm conv}(x)>0$, where ${\rm conv}(x)$ denotes the convexity radius at $x\in D\subset O^{{\rm reg}}$.

Now, we take an arbitrary $\epsilon>0$ such that $\epsilon<r_{D}$. Since $\lim_{t\to \infty}L(t)=0$, there exists a time $t_{\epsilon}>0$ such that $L(t)<\epsilon$ for any $t>t_{\epsilon}$. Then, for each $t>t_{\epsilon}$, there exists a point $x=x(t)\in \mathcal{O}$ and $r(t)>0$ such that $\gamma_{t}(S^{1})\subset B(x(t),r(t))$. Indeed, we have the following two possibilities:
\begin{itemize}
\item[(i)] If $\gamma_{t}(S^{1})\cap D\neq \emptyset$, we can take the point $x(t)$ as an arbitrary point in $\gamma_{t}(S^{1})\cap D$ because for any $y\in \gamma_{t}(S^{1})\subset \mathcal{O}^{{\rm reg}}$, we have $d(x(t), y)\leq L(t)<\epsilon<r_{D}$. Namely, we have $\gamma_{t}(S^{1})\subset B(x(t), r_{D})\subset O^{{\rm reg}}$. 
\item[(ii)] If $\gamma_{t}(S^{1})\cap D= \emptyset$, then by definition of $D$, we see $\gamma_{t}(S^{1})\subset B(x_{i},r_{x_{i}})$ for some $i=1,\ldots, N$. 
\end{itemize}

In case (i), $\gamma_{t} $ encloses a region $E_{t}$ contained in $B(x(t),r_{D})$ since $\gamma_{t}$ is an embedded curve. We denote the area enclosed by $\gamma_{t}$ by $A(t)$. Then, because $r_{D}\leq {\rm conv}(x(t))$, we can apply Croke's isoperimetric inequality \cite[Theorem 1.2]{Croke} to obtain that $L(t)^{2}\geq 2\pi A(t)$. 

In case (ii), we further divide the case into the following two possibilities:
\begin{itemize}
\item[(ii)-(a)] If the embedded curve $\gamma_{t}(S^{1})\subset B(x_{i},r_{i})$ is null-homotopic in $B(x_{i},r_{i})\setminus \{x_{i}\}$, then $\varphi^{-1}(\gamma_{t}(S^{1}))\subset \widetilde{B}(0,r_{x})$ consists of disjoint union of 
$k$ copies of closed curve $\gamma_{t}$.  Thus, in the same way as in (i), we see that the curve $\gamma_{t}$ satisfies $L(t)^{2}\geq 2\pi A(t)$ since a neighbourhood of each copy is locally isometric to a neighbourhood of $\gamma_{t}(S^{1})$.

 \item[(ii)-(b)] If the embedded curve $\gamma_{t}(S^{1})\subset B(x_{i},r_{i})$ is not null-homotopic in  $B(x_{i},r_{i})\setminus \{x_{i}\}$, then  $\varphi^{-1}(\gamma_{t}(S^{1}))$ becomes a connected, embedded closed curve $\widetilde{\gamma}_{t}: S^{1}\to \widetilde{B}(0,r_{x_i})$ homotopic to the origin $0\in \widetilde{B}(0,r_{x_i})$. Then we have $\widetilde{L}(t)^{2}\geq 2\pi \widetilde{A}(t)$, where $\widetilde{L}(t)$ (resp. $\widetilde{A}(t)$) is the length of $\widetilde{\gamma}_{t}$ (resp. the area enclosed by $\widetilde{\gamma}_{t}$). On the other hand, since the restriction $\varphi: \widetilde{\gamma}_{t}(S^{1})\to \gamma_{t}(S^{1})$ becomes a $k$-fold covering map, we have $\widetilde{L}(t)=kL(t)$. Moreover, $E_{t}=\widetilde{E}_{t}/\mathbb{Z}_k$ is homeomorphic to the closed disk and we have $\widetilde{A}(t)=kA(t)$ by definition of the integral over the region $E_{t}$ containing a cone point. Therefore, $\gamma_{t}$ encloses a region $E_{t}$ and we obtain
 \[
 L(t)^{2}\geq \frac{2\pi}{k}A(t).
 \]
\end{itemize}

We thus put $c_{1}:={\rm min}\{1, 1/k_{1},\ldots,1/k_{N}\}$, where $k_{i}$  is the order at the cone point $x_{i}$. Then for any sufficiently large $t$, we see that $\gamma_{t}$ encloses a region $E_{t}$ and we have $L(t)^{2}\geq 2\pi c_{1}
\cdot  A(t)$. Since $\lim_{t\to \infty}L(t)=0$, this implies that $\lim_{t\to \infty}A(t)=0$ as required. 

In particular,   we may assume that $L(t)\leq \epsilon$ and $A(t)\leq \epsilon$ for any sufficiently large $t$.
Recall that the time derivative of $A(t)$ is given by (Lemma \ref{lem:area})
\[
\frac{d}{dt}A(t)=\int_{S^{1}}\gamma_{t}^{*}(i_{\kappa_{\psi}N}\omega)=-\int_{S^{1}}(\kappa-\langle \overline{\nabla}\psi,N\rangle)\,ds_{t},
\]
Using the Gauss-Bonnet theorem for orbifold (Theorem \ref{thm:GB}), we obtain 
\begin{align*}
\frac{d}{dt}A(t)&=-2\pi\chi^{{\rm orb}}(E_{t})+\int_{E_{t}}K\,\omega+\int_{S^{1}} \langle \overline{\nabla}\psi,N\rangle\,ds_{t}\leq -2\pi\chi^{{\rm orb}}(E_{t})+(C_{0}+P_{1})\epsilon.
\end{align*}
Since each region $E_{t}$ contains at most one cone point, we have that 
\[
\chi^{{\rm orb}}(E_{t})=\begin{cases} 1 & \textup{if there is no cone point in $E_{t}$.}\\ 
1/k & \textup{if there is a cone point  of order $k$ in $E_{t}$}.
\end{cases}
\]
Hence, we see $\chi^{{\rm orb}}(E_{t})\geq c_{1}={\rm min}\{1, 1/k_{1},\ldots,1/k_{N}\}$. Therefore, by taking a sufficiently small $\epsilon$, it holds that 
\[
\frac{d}{dt}A(t)\leq -2\pi c_{1}+(C_{0}+P_{1})\epsilon<-\delta
\]
for any large $t$, where $\delta >0$
is some positive constant independent of $t$. In particular, we have $0<A(t)\leq A(t_{0})-\delta(t-t_{0})$ for some $t_{0}$, and this yields a contradiction by taking $t\to \infty$.  This proves the lemma.
\end{proof}

The following result is stated in \cite[Step 1,2 in \S 3]{MV} for the weighted CSF in a smooth surface. However, we provide an alternative proof of it for completeness.

\begin{lemma}[$L^2$-decay of the $\psi$-curvature]\label{lem:l2decay} 
Under the assumptions of this subsection,
    \begin{align}\label{cl:1}
 \lim_{t\to \infty}\int_{S^{1}}\kappa_{\psi}^{2}\,ds_{\psi}(t)=0.
 \end{align}
\end{lemma}

\begin{proof}
We generalise the argument of Gage \cite[\S 4]{Gage} and Grayson \cite[\S 7]{Grayson} for the usual CSF on surfaces.
By using \eqref{eq:fv}, we have
\begin{align}\label{eq:l1}
L_{\psi}(0)-L_{\psi}(t)=-\int_{0}^{t}\frac{dL_{\psi}}{dt}\,dt=\int_{0}^{t}\int_{S^{1}}\kappa_{\psi}^{2}\,ds_{\psi}(t)\,dt
\end{align}
and this shows that the integral of the $L^{2}$-norm over $[0,\infty)$ is bounded from above since $L_{\psi}(t)$ is a decreasing function, and it converges to a finite value.  In particular, for any $\epsilon>0$, there is a time $t_{1}=t_{1}(\epsilon)$ so that it holds that 
\begin{align}\label{eq:key}
\int_{S^{1}}\kappa_{\psi}^{2}\,ds_{\psi}(t_{1})<\frac{\epsilon}{2}\quad {\rm and}\quad \int_{t}^{\infty}\int_{S^{1}}\kappa_{\psi}^{2}\,ds_{\psi}(t)\,dt<\epsilon^{2}\ \textup{for any $t\in [t_{1},\infty)$}.
\end{align}

Next, we consider the time derivative of the $L^{2}$-norm.  By \eqref{eq:kappa}, \eqref{eq:fv} and \eqref{eq:kev}, we have
\begin{align*}
\frac{\p}{\p t}\int_{S^{1}}\kappa_{\psi}^{2}\,ds_{\psi}(t)&=\int_{S^{1}}2\kappa_{\psi}\Delta_{\psi}\kappa_{\psi}+2\kappa_{\psi}^{2}\{\kappa^{2}+{\rm Ric}_{\psi}(N,N)\}-\kappa_{\psi}^{4}\,ds_{\psi}(t)\\
&\leq \int_{S^{1}}-2|\p_{s}\kappa_{\psi}|^{2}+3\kappa_{\psi}^{4}+2\kappa_{\psi}^{2}\{2\langle \overline{\nabla}\psi, N\rangle^{2}+{\rm Ric}_{\psi}(N,N)\}\,ds_{\psi}(t),
\end{align*}
where we used the Stokes theorem for the weighted Laplacian
 \begin{align}\label{eq:stokes}
\int_{S^{1}} f_{1}\Delta_{\psi}f_{2}\,ds_{\psi}=\int_{S^{1}} -(\p_{s}f_{1})(\p_{s}f_{2})\,ds_{\psi}
 \end{align}
and the fact that $\kappa^{2}\leq 2\kappa_{\psi}^{2}+2\langle \overline{\nabla}\psi, N\rangle^{2}$.
Moreover, by using the assumption \eqref{as:1}, we obtain
\begin{align}\label{eq:l2}
\frac{\p}{\p t}\int_{S^{1}}\kappa_{\psi}^{2}\,ds_{\psi}(t)
&\leq  \int_{S^{1}}-2|\p_{s}\kappa_{\psi}|^{2}\,ds_{\psi}(t)+\Big(3\cdot {\rm sup}_{x\in S^{1}}\kappa_{\psi}^{2}(x,t)+a_{1}\Big)\int_{S^{1}}\kappa_{\psi}^{2}\,ds_{\psi}(t) 
\end{align}
for some positive constant $a_{1}$ independent of $t$. Here, we claim that there are positive constants $b_{1},b_{2}$ independent of $t$ such that 
\begin{align}\label{eq:supk0}
3\cdot {\rm sup}_{x\in S^{1}} \kappa_{\psi}^{2}(x,t)\leq b_{1}\int_{S^{1}}\kappa_{\psi}^{2}\,ds_{\psi}(t)+b_{2}\int_{S^{1}}|\p_{s}\kappa_{\psi}|^{2}\,ds_{\psi}(t).
\end{align}
In fact, for any fixed $t\in [0,\infty)$, we take an $s_{0}$ so that $\kappa_{\psi}^{2}(s_{0},t)={\rm min}_{s}\kappa_{\psi}^{2}(s,t)$, and then we see
\begin{align}\label{eq:supk1}
\kappa_{\psi}^{2}(s,t)&=\Big(\kappa_{\psi}(s_{0},t)+\int_{s_{0}}^{s} \frac{\p \kappa_{\psi}}{\p s}\,ds(t)\Big)^{2}
\leq\Big( |\kappa_{\psi}(s_{0},t)|+\Big|\int_{s_{0}}^{s} \frac{\p \kappa_{\psi}}{\p s}\,ds(t)\Big|\Big)^{2}\\
&\leq 2|\kappa_{\psi}(s_{0},t)|^{2}+2\Big|\int_{s_{0}}^{s} \frac{\p \kappa_{\psi}}{\p s}\,ds(t)\Big|^{2} \nonumber\\
&\leq 2|\kappa_{\psi}(s_{0},t)|^{2}+2\Big(\int_{s_{0}}^{s} \Big|\frac{\p \kappa_{\psi}}{\p s}\Big|^{2}\,ds(t)\Big)\Big|\int_{s_{0}}^{s} \,ds(t)\Big|\nonumber
\end{align}
for any $s$, where we parametrise $\gamma_{t}$ by the arc length parameter $s=s(t)$. 
Since $\kappa_{\psi}^{2}(s_{0},t)={\rm min}_{s}\kappa_{\psi}^{2}(s,t)$, we have
\begin{align}\label{eq:supk2}
|\kappa_{\psi}(s_{0},t)|^{2}\leq \frac{1}{L_{\psi}(t)}\int_{S^{1}}\kappa_{\psi}^{2}\,ds_{\psi}(t)\leq \frac{1}{Ec}\int_{S^{1}}\kappa_{\psi}^{2}\,ds_{\psi}(t),
\end{align}
where we used \eqref{as:1} and \eqref{as:3}. Also, by using \eqref{as:1}, we see
\begin{align}\label{eq:supk3}
\Big(\int_{s_{0}}^{s} \Big|\frac{\p \kappa_{\psi}}{\p s}\Big|^{2}\,ds(t)\Big)\Big|\int_{s_{0}}^{s} \,ds(t)\Big|
&\leq \frac{1}{E^{2}}\Big(\int_{S^{1}} \Big|\frac{\p \kappa_{\psi}}{\p s}\Big|^{2}\,ds_{\psi}(t)\Big)\Big(\int_{S^{1}} \,ds_{\psi}(t)\Big)\\
&\leq \frac{L_{\psi}(t)}{E^{2}}\int_{S^{1}}|\p_{s}\kappa_{\psi}|^{2}\,ds_{\psi}(t)\leq \frac{L_{\psi}(0)}{E^{2}}\int_{S^{1}}|\p_{s}\kappa_{\psi}|^{2}\,ds_{\psi}(t),\nonumber
\end{align}
where we used the fact that $L_{\psi}(t)\leq L_{\psi}(0)$. Substituting \eqref{eq:supk2} and \eqref{eq:supk3} to \eqref{eq:supk1}, we obtain \eqref{eq:supk0}. 

Thus, inserting \eqref{eq:supk0} to \eqref{eq:l2}, we obtain
\begin{align}\label{eq:l22}
\frac{\p}{\p t}\int_{S^{1}}\kappa_{\psi}^{2}\,ds_{\psi}(t)
&\leq  \int_{S^{1}}-2|\p_{s}\kappa_{\psi}|^{2}\,ds_{\psi}(t)+b_{2}\Big(\int_{S^{1}}\kappa_{\psi}^{2}\,ds_{\psi}(t) \Big)\Big(\int_{S^{1}}|\p_{s}\kappa_{\psi}|^{2}\,ds_{\psi}(t)\Big)\\
&\quad +b_{1}\Big(\int_{S^{1}}\kappa_{\psi}^{2}\,ds_{\psi}(t)\Big)^{2}+a_{1}\int_{S^{1}}\kappa_{\psi}^{2}\,ds_{\psi}(t). \nonumber
\end{align}

Now, we take an  $\epsilon>0$ such that $\epsilon<1/b_{2}$, and let $t_{1}=t_{1}(\epsilon)$ be the time satisfying \eqref{eq:key}. We claim that if $\epsilon$ is sufficiently small, then
\begin{align}\label{eq:cl1}
\int_{S^{1}}\kappa_{\psi}^{2}\,ds_{\psi}(t)<\epsilon\quad \forall t\in [t_{1},\infty),
\end{align}
which implies the desired conclusion \eqref{cl:1}. Suppose the contrary of \eqref{eq:cl1} were true. Namely, there is a time $t_{2}'>t_{1}$ so that 
\[
\int_{S^{1}}\kappa_{\psi}^{2}ds_{\psi}(t_{1})<\frac{\epsilon}{2}<\epsilon\leq \int_{S^{1}}\kappa_{\psi}^{2}ds_{\psi}(t_{2}').
\]
By the continuity of the $L^{2}$-norm, we may assume that there is a time $t_{2}\in (t_{1},t_{2}']$ so that 
\begin{align}\label{eq:key2}
\int_{S^{1}}\kappa_{\psi}^{2}\,ds_{\psi}(t_{2})=\epsilon\quad {\rm and}\quad \int_{S^{1}} \kappa_{\psi}^{2}\,ds_{\psi}(t)\leq \epsilon\quad \forall t\in [t_{1},t_{2}].
\end{align}
Notice that, since $\epsilon<1/b_{2}$, the estimate \eqref{eq:l22} implies that 
\begin{align*}
\frac{\p}{\p t}\int_{S^{1}}\kappa_{\psi}^{2}\,ds_{\psi}(t)
&\leq  b_{1}\Big(\int_{S^{1}}\kappa_{\psi}^{2}\,ds_{\psi}(t)\Big)^{2}+a_{1}\int_{S^{1}}\kappa_{\psi}^{2}\,ds_{\psi}(t)
\end{align*}
on the interval $[t_{1},t_{2}]$. Integrating this over $[t_{1},t_{2}]$, we see
\begin{align}\label{eq:key3}
\epsilon=\int_{S^{1}}\kappa_{\psi}^{2}\,ds_{\psi}(t_{2})&\leq \int_{S^{1}}\kappa_{\psi}^{2}\,ds_{\psi}(t_{1})+b_{1}\int_{t_{1}}^{t_{2}}\Big(\int_{S^{1}}\kappa_{\psi}^{2}\,ds_{\psi}(t)\Big)^{2}\,dt+a_{1}\int_{t_{1}}^{t_{2}}\Big(\int_{S^{1}}\kappa_{\psi}^{2}\,ds_{\psi}(t)\Big)\,dt\\
&<\frac{\epsilon}{2}+b_{1}\epsilon^{3}+a_{1}\epsilon^{2}, \nonumber
\end{align}
where we used \eqref{eq:key}, \eqref{eq:key2} and their consequence
\[
\int_{t_{1}}^{t_{2}}\Big(\int_{S^{1}}\kappa_{\psi}^{2}\,ds_{\psi}(t)\Big)^{2}\,dt\leq \int_{t_{1}}^{t_{2}}\epsilon\cdot \Big(\int_{S^{1}}\kappa_{\psi}^{2}\,ds_{\psi}(t)\Big)\,dt\leq \epsilon^{3}.
\]
The inequality \eqref{eq:key3} shows that $\frac{1}{2}<b_{1}\epsilon^{2}+a_{1}\epsilon$ which yields a contradiction if we take $\epsilon$ sufficiently small. This completes the proof.
\end{proof}

From Lemma \ref{lem:l2decay}, we can deduce the following decay of the higher derivatives of the $\psi$-curvature.

\begin{lemma}[$W^{m,2}$-decay of the $\psi$-curvature]\label{lem:w2ndecay} Under the assumptions of this subsection,
    \begin{align}\label{cl:2}
 \lim_{t\to \infty}\int_{S^{1}}|\partial_{s}^{m}\kappa_{\psi}|^{2}\,ds_{\psi}(t)=0,\quad \forall m=0,1,2,\cdots.
 \end{align}
\end{lemma}
 
\begin{proof}
The case \(m=0\) is Lemma 3.5. For \(m\geq1\), the induction argument is identical to that of \cite[Section 3]{MV}, and the details will be omitted for brevity. The hypotheses used there are precisely the uniform bounds \eqref{as:1}, the area bound \eqref{as:2}, and the length lower bound \eqref{as:3}. Since the flow is contained in the smooth surface $\mathcal{O}^{{\rm reg}}$, the argument is local and therefore applies unchanged.
\end{proof}

Finally, we have the following $C^m$-estimates on the flow.

\begin{lemma}[Uniform curvature decay and derivative bounds]\label{lem:nabbd}
Under the assumptions of this subsection:
\begin{itemize}
    \item[(i)] $\lim_{t\to \infty} \|\partial_s^m\kappa_\psi(s,t)\|_{L^\infty} \, = \, 0 \quad \forall m=0,1,2\ldots$
    \item[(ii)] Defining $\onab^{m}_{\p_{s_{t}}}\gamma_{t}:=\onab^{m-1}_{\p_{s_{t}}\gamma_{t}}\p_{s_{t}}\gamma_{t}$ for $m=1,2,\cdots$, there exists $C_{m}>0$ such that $|\onab^{m}_{\p_{s_{t}}}\gamma_{t}|\leq C_{m}$ 
for any $t\in [0,\infty)$.
\end{itemize}
\end{lemma}

\begin{proof}
By the same argument given in \eqref{eq:supk1}--\eqref{eq:supk3}, we observe that 
\[
|\partial_{s}^{m}\kappa_{\psi}(s,t)|^{2}\leq c_{1}\int_{S^{1}} |\p_{s}^{m}\kappa_{\psi}|^{2}\,ds_{\psi}(t)+c_{2}\int_{S^{1}}|\partial_{s}^{m+1}\kappa_{\psi}|^{2}\,ds_{\psi}(t)
\]
for any $m=0,1,2,\cdots$, where $c_{1},c_{2}$ are some positive constants independent of $t$. Thus, as a consequence of  Lemma \ref{lem:w2ndecay}, we see that $\p_{s}^{m}\kappa_{\psi}(t)$ uniformly converges to 0 as $t\to \infty$ for any $m=0,1,2,\cdots$. This proves (i).

Note that $\onab_{\p_{s}}^{1}\gamma=\p_{s}\gamma$ and $\onab_{\p_{s}}^{2}\gamma=\kappa N$.  Moreover, we have $\onab_{\p_{s}}N=-\kappa\p_{s}\gamma$. Therefore, by induction, we see (see also  \cite[eq.(79)]{MV}), 
\[
\onab^{m}_{\p s_{t}}\gamma_{t}=\alpha_{m}(\kappa,\partial_{s}\kappa,\ldots, \partial^{m-3}_{s}\kappa)\cdot \partial_{s}\gamma+\beta_{m}(\kappa,\partial_{s}\kappa,\ldots, \partial^{m-2}_{s}\kappa)\cdot N,
\]
where $\alpha_{m}$ and $\beta_{m}$ are some polynomials depending only on $m$.
Since $\p_{s}^{m}\kappa_{\psi}(t)$ uniformly converges to 0 as $t\to \infty$, we have $|\p_{s}^{m}\kappa_{\psi}(t)|\leq C_{m}'$ for some $C_{m}'>0$. Moreover, by using the assumption \eqref{as:1},  a direct computation  (see \cite[eq. (51)]{MV}) shows that $|\partial_{s}^{m}\kappa|=|\partial_{s}^{m}(\kappa_{\psi}+\langle\overline{\nabla}\psi, N\rangle)|$ is also bounded by a constant independent of $t$. Therefore, we see 
\[
|\onab^{m}_{\p s_{t}}\gamma_{t}|^{2}={\alpha_{m}(\kappa,\partial_{s}\kappa,\ldots, \partial^{m-3}_{s}\kappa)^{2}+\beta_{m}(\kappa,\partial_{s}\kappa,\ldots, \partial^{m-2}_{s}\kappa)^{2}}\leq C_{m}
\] for some $C_{m}>0$.
This proves (ii).
\end{proof}

\subsection{Subconvergence of unparametrised solutions}\label{subsec:subcon}

In this subsection, we aim to prove the following main result of this section.

\begin{theorem}\label{thm:subcon}
Let $(\mathcal{O},g, {\psi})$ be a compact, orientable $2$-dimensional weighted Riemannian orbifold, and $\gamma:S^1\times[0,\infty)  \to \mathcal{O}^{{\rm reg}}$ be a long-time embedded solution of $\psi$-CSF contained in the regular part of $\mathcal{O}$.  Then $\kappa_{\psi}(t)$ uniformly converges to $0$, and there exists a subsequence $\{\gamma_{t_{i}}\}_{i=1}^{\infty}$ which converges to a continuous map $\gamma_{\infty}: S^{1}\to \mathcal{O}$ in $C^{0}(S^{1},\mathcal{O})$. Moreover, we have the following.
\begin{enumerate}
 \item  $\gamma_{\infty}$ does not pass through any cone point of order $\geq 3$. 
\item  If $\gamma_{\infty}$ passes through a cone point $x_0$ of order $2$, then there exists a cone point $y_0$ (where possibly $y_0=x_0$) of order $2$ and $\gamma_\infty$ goes back and forth on a geodesic path between $x_0$ and $y_0$.
\item If $\gamma_{\infty}$ does not pass through any singular point (i.e.\ $\gamma_{\infty}(S^{1})\subset \mathcal{O}^{{\rm reg}}$), then there exists a subsequence of $\{\widehat{\gamma}_{t_{i}}\}_{i=1}^{\infty}$ of an  unparametrised solution \eqref{eq:unpCSF}
such that $\widehat{\gamma}_{t_{i}}: S^{1}\to \mathcal{O}^{{\rm reg}}$ smoothly converges to a $\psi$-minimal embedding $\widehat{\gamma}_{\infty}: S^{1}\to \mathcal{O}^{{\rm reg}}$ whose image coincides with that of $\gamma_{\infty}$. 
As a consequence, the image of $\gamma_{\infty}$ coincides with a $\psi$-geodesic contained in $\mathcal{O}^{{\rm reg}}$. 
%More precisely, there exists a parameter $s\in [0,l]$ of $S^1$ so that $\gamma_\infty(s)=\overline{\gamma}_\infty(s)$ for any $s$, where $\overline{\gamma}_\infty(s):=\gamma_\infty(l-s)$ is the reversed curve of $\gamma_\infty$.
\end{enumerate}
\end{theorem}

\begin{remark}
We remark briefly on the geometric idea behind the conclusion of the theorem. Away from singularities, the compactness gives subsequential smooth convergence to a weighted geodesic. If the flow reached a cone point of order $k \geq 3$, the rotated lifts in the orbifold chart would have to cross transversely; this cannot happen by the embeddedness assumption. In contrast, at a cone point of order 2, the lifted geodesic is left invariant under a rotation by $\pi$ which reverses the orientation. The limiting curve is therefore `reflected' at this point, and thus  by periodicity it links one or two cone points of order two.

{\rm This idea was already observed by Lange in \cite[Lemma 4.4, Proposition 4.5]{Lange} for the (usual) CSF in a simply-connected spindle orbifold. While his proof is based on a geometric insight, some of the analytic details are omitted. Our proof below provides in particular an alternative and detailed proof of his specialised result.}
\end{remark}

Our proof of the item (iii) below is a generalisation of the proofs given in \cite{Gage, Grayson} and \cite[\S 3]{MV}. We especially follow the argument given by Miquel-Vi\~nado-Lereu  \cite[\S 3]{MV}, in which a similar result is proved for the weighted CSF in a smooth surface. On the other hand, both (i) and (ii) are phenomena specific to orbifolds and require additional arguments.

\begin{proof}
The proof proceeds in stages.\\

\noindent\textbf{
{1. Construct the limiting curve.}
} 
Let $\gamma: S^{1}\times [0,\infty)\to \mathcal{O}^{\rm reg}$ be a solution of $\psi$-CSF contained in the regular part of a compact orbifold $\mathcal{O}$. By the avoidance principle (Proposition \ref{prop:ap}), we may assume that $\gamma_{t}$ is embedding for all $t\in [0,\infty)$. 

We fix a parameter of $S^{1}$ by $\theta\mapsto (\cos \theta, \sin\theta)$. Using the relation  $ds_{\psi}(t)=e^{\psi}|\partial_{\theta}\gamma_{t}|\,d\theta$, \eqref{eq:fv} yields  that
$
\p_{t}(\log (e^{\psi}|\p_{\theta}\gamma_{t}|))=-\kappa_{\psi}^{2}(t)\leq 0
$
and hence, $e^{\psi}|\partial_{\theta}\gamma_{t}|$ is not increasing. Since $e^{\psi}$ is bounded from below, this shows $|\partial_{\theta}\gamma_{t}|\leq C$ for some positive constant $C$, and this implies that $\{\gamma_{t}\}_{t\in [0,\infty)}$ is equicontinuous. Thus, using the Arzel\`a-Ascoli theorem, there exists a convergent subsequence $\{\gamma_{t_{i}}\}_{i=1}^{\infty}$ with respect to the compact open topology of the set of continuous maps $C^{0}(S^1,\mathcal{O})$, where we note that $\mathcal{O}$ can be regarded as a metric space equipped with the distance $d_{\mathcal{O}}$ induced from the orbifold Riemannian metric $g$.  We denote the limiting continuous function by $\gamma_{\infty}: S^{1}\to \mathcal{O}$, and set $\gamma_{i}:=\gamma_{t_{i}}$.  Note that we have $\kappa_{\psi}(t)\to 0$ uniformly as shown in Lemma \ref{lem:nabbd}. However, strictly speaking,  these results do not show that  $\gamma_{\infty}$ is a $\psi$-minimal immersion (or a $\psi$-geodesic). \\

\noindent\textbf{
{2. Construct an unparametrised solution.}
} To prove the items (i)--(iii),  we consider a reparametrisation of $S^{1}$  used in \cite{MV}. For each $t$, we define a diffeomorphism $\varphi_{t}: S^{1}\to S^{1}$ by
\[
\varphi_{t}(\cos\theta,\sin\theta):= \Big(\cos \frac{2\pi s_{t}(\theta)}{L_{t}}, \sin \frac{2\pi s_{t}(\theta)}{L_{t}}\Big)
\]
where $s_{t}$ is the arclength parameter of $\gamma_{t}$ and $L_{t}=L(\gamma_{t})$ is the length of $\gamma_{t}$. 
 Note that $\varphi_{t}$ is defined by a composition of the following coordinate transformations
\[
(0,2\pi) \xrightarrow{s_{t}} (0,L_{t}) \xrightarrow{l_{t}} (0,2\pi),
\]
where $l_{t}(s_{t}):=2\pi s_{t}/L_{t}$. We denote the parameter of the last interval $(0,2\pi)$ by $\eta$ distinguishing from the parameter $\theta$ of the first interval $(0,2\pi)$. 

Now, we define  a family of embeddings by
\[
\widehat{\gamma}_{t}: S^{1}\to \mathcal{O}^{{\rm reg}}\quad \widehat{\gamma}_{t}:=\gamma_{t}\circ \varphi_{t}^{-1}.
\]
By definition, $\{\widehat{\gamma}_{t}\}_{t}$ is an unparametrised solution of the weighted CSF. 
We consider the subsequence $\{\widehat{\gamma}_{i}\}_{i=1}^{\infty}$ of the unparametrised solution  defined by $\widehat{\gamma}_{i}:=\gamma_{t_{i}}\circ \varphi_{t_{i}}^{-1}$.
 Since $|\p_{\eta}\widehat{\gamma}_{i}|=|\p_{s_{i}}{\gamma}_{i}\cdot \p_{\eta} l_{i}^{-1}|=L_{i}/2\pi\leq L_{\psi}(t_{i})/2\pi E\leq L_{\psi}(0)/2\pi E$, there exists a subsequence $\{\widehat{\gamma}_{i_{j}}\}_{j=1}^{\infty}$ which converges to a continuous map $\widehat{\gamma}_{\infty}: S^{1} \to \mathcal{O}$.
 On the other hand, since $|\p_{\theta}\gamma_{t}|\leq C$ and we have $L_{t}\geq c$  by Lemma \ref{lem:leng}, we see $|\p_{\theta}\varphi_{i}|=(2\pi/L_{i})|\p_{\theta} s_{i}|=(2\pi/L_{i})|\p_{\theta}\gamma_{i}|\leq 2\pi C/c$ for all $i$, and hence, there exists a subsequence $\{\varphi_{i_{j_{k}}}\}_{k=1}^{\infty}$ of $\{\varphi_{i_{j}}\}_{j=1}^{\infty}$ such that it also converges to a continuous map $\varphi_{\infty}: S^{1}\to S^{1}$. We remark that since $\varphi_{i}: S^{1}\to S^{1}$ is a homeomorphism, $\varphi_{\infty}:S^{1}\to S^{1}$ is  surjective. 
 Now, we consider the following subsequence of $\{\widehat{\gamma}_{i}\}_{i=1}^{\infty}$:
\[
\widehat{\gamma}_{i_{k}}:=\gamma_{i_{j_{k}}}\circ \varphi_{i_{j_{k}}}^{-1}.
\]
Equivalently, we put $\gamma_{i_{j_{k}}}=\widehat{\gamma}_{i_{k}}\circ \varphi_{i_{j_{k}}}$. By taking $k\to \infty$, we obtain that 
$
\gamma_{\infty}=\widehat{\gamma}_{\infty}\circ \varphi_{\infty}
$
and this shows that the image of $\widehat\gamma_{\infty}$ coincides with that of $\gamma_{\infty}$. Thus, in the following, we may assume that the subsequence $\{\widehat{\gamma}_{i}\}_{i=1}^{\infty}$ converges to a continuous map $\widehat{\gamma}_{\infty}: S^{1}\to \mathcal{O}$ whose image coincides with $\gamma_{\infty}$. \\

\noindent\textbf{
3. Proof of (iii).
} We first consider the item (iii), namely, the case  when $\gamma_{\infty}$ (or $\widehat{\gamma}_{\infty}$) does not pass through any singular point of $\mathcal{O}$. Our claim is that $\widehat{\gamma}_{\infty}: S^{1}\to \mathcal{O}^{{\rm reg}}$ is a smooth $\psi$-minimal embedding and there is a subsequence of $\{\widehat{\gamma}_i\}_{i=1}^\infty$ such that it smoothly converges to $\widehat{\gamma}_\infty$. To prove this, we take a finite open covering $\bigcup_{\lambda=1}^N I_{\lambda}$ of $S^1$ so that each $I_\lambda$ is homeomorphic to a bounded connected interval and there exists an isothermal coordinate $(U_{\lambda},\psi_{\lambda})$ of $\mathcal{O}^{{\rm reg}}$ such that $\widehat{\gamma}_{\infty}(I_{\lambda})\subset U_{\lambda}$. Since $\widehat{\gamma}_{i}$ converges to $\widehat{\gamma}_{\infty}$ in $C^{0}$-topology,  we see that $\widehat{\gamma}_{i}(I_{\lambda})\subset U_{\lambda}$ for any $\lambda=1,\ldots, N$ and any large $i$. Because we are interested in the limiting behaviour,  we may assume that  $\widehat{\gamma}_{i}(I_{\lambda})\subset U_{\lambda}$ for any $\lambda=1,\ldots, N$ and any $i=1,2,\ldots$ in the following.

Fix arbitrary $\lambda$ and set 
\[
\widehat{\gamma}_{i}^{\lambda}:=\psi_{\lambda}\circ \widehat{\gamma}_{i}|_{I_{\lambda}}: I_{\lambda}\to \R^{2},
\quad \widehat{\gamma}_{i}^{\lambda}=(\widehat{\gamma}_{i}^{\lambda 1}, \widehat{\gamma}_{i}^{\lambda 2}).
\]
Obviously, the sequence $\{\widehat{\gamma}_{i}^{\lambda}\}_{i=1}^{\infty}$ converges to a continuous map $\widehat{\gamma}_{\infty}^{\lambda}:=\psi_{\lambda}\circ \widehat{\gamma}_{\infty}|_{I_{\lambda}}$.  We shall show that there exists a subsequence of  $\{\widehat{\gamma}_{i}^{\lambda}\}_{i=1}^{\infty}$ which smoothly converges to $\widehat{\gamma}_{\infty}^{\lambda}$. To show this, it is sufficient to prove that, for any $m$ there exists a positive constant $C_{m}$ such that 
\begin{align}\label{eq:keybd}
\Big|\frac{\p ^{m} \widehat{\gamma}_{i}^{\lambda}}{\p \eta^{m}}\Big|^{2}\leq C_{m}\quad \forall i=1,2,\ldots,
\end{align}
Indeed, if \eqref{eq:keybd} holds, then the sequence of each coordinate function $\{\widehat{\gamma}_{i}^{\lambda l}\}_{i=1}^{\infty}$ $(l=1,2)$ is a bounded sequence in the Sobolev space $W^{m,2}(I_{\lambda})$ and hence, there is a weakly convergent subsequence $\{\widehat{\gamma}_{i_{j}}^{\lambda l}\}_{j=1}^{\infty}$ in $W^{m,2}(I_{\lambda})$. By the Sobolev embedding theorem, $\{\widehat{\gamma}_{i_{j}}^{\lambda l}\}_{j=1}^{\infty}$ is a sequence in a compact set of $C^{m-1,\alpha}(\overline{I_{\lambda}})$, and hence, there is a (strongly) convergent subsequence in $C^{m-1,\alpha}(\overline{I_{\lambda}})$. Since the limit coincides with $\widehat{\gamma}_{\infty}^{\lambda l}$, this shows that $\widehat{\gamma}_{\infty}^{\lambda l}$ is $C^{m-1}$ for each $l=1,2$, and this implies that there is a subsequence of $\{\widehat{\gamma}_{i}^{\lambda}\}_{i=1}^{\infty}$ which converges to $\widehat{\gamma}_{\infty}^{\lambda}$ in $C^{m-1}$.
 Because $n$ is arbitrary, by the standard diagonal argument, we conclude the existence of  a smoothly convergent subsequence $\widehat{\gamma}_{i_{j}}^{\lambda}\to \widehat{\gamma}_{\infty}^{\lambda}$.

The inequality \eqref{eq:keybd} is proved as follows. Since $\widehat{\gamma}_{i}^{\lambda}$ is regarded as  $\widehat{\gamma}_{i}^{\lambda}={\gamma}_{i}^{\lambda}\circ (s_{i}^{-1}\circ l_{i}^{-1})$, where $\gamma_{i}^{\lambda}=\psi_{\lambda}\circ \gamma_{i}|_{I_{\lambda}}$. Then, we compute 
\begin{align*}
\frac{\p \widehat{\gamma}_{i}^{\lambda}}{\p \eta}&=\frac{\p (\gamma_{i}^{\lambda}\circ s_{i}^{-1})}{\p s_{i}}\cdot \frac{\p l_{i}^{-1}}{\p \eta}=\frac{L_{i}}{2\pi} \frac{\p \gamma_{i}^{\lambda}}{\p s_{i}}\\
\frac{\p^{2} \widehat{\gamma}_{i}^{\lambda}}{\p \eta^{2}}&=\frac{L_{i}}{2\pi}\frac{\p^{2} (\gamma_{i}^{\lambda}\circ s_{i}^{-1})}{\p s_{i}^{2}}\cdot \frac{\p l_{i}^{-1}}{\p \eta}=\Big(\frac{L_{i}}{2\pi}\Big)^{2} \frac{\p^{2} \gamma_{i}^{\lambda}}{\p s_{i}^{2}}.
\end{align*}
since $l_{i}^{-1}(\eta)=(L_{i}\eta)/2\pi$. 
Thus, by induction, we obtain $\p^{m}_{\eta} \widehat{\gamma}_{i}^{\lambda}=(L_{i}/2\pi)^{m}\p^{m}_{s_{i}}\gamma_{i}^{\lambda}$. In particular, we see
\[
\Big|\frac{\p ^{m} \widehat{\gamma}_{i}^{\lambda}}{\p \eta^{m}}\Big|^{2}= \Big(\frac{L_{i}}{2\pi}\Big)^{2m} \Big|\frac{\p^{m} \gamma_{i}^{\lambda}}{\p s_{i}^{m}}\Big|^{2}\leq \Big(\frac{L_{\psi}(0)}{2\pi E}\Big)^{2m} \Big|\frac{\p^{m} \gamma_{i}^{\lambda}}{\p s_{i}^{m}}\Big|^{2}.
\]
Thus it suffices to show $|\p_{s_{i}}^{m}\gamma_{i}^{\lambda}|$ is bounded for any $n$, but this follows from Lemma \ref{lem:nabbd}. Indeed, using the chart $(U_{\lambda},\psi_{\lambda}=(x_{\lambda}^{1},x_{\lambda}^{2}))$, we have 
\begin{align*}
\onab_{\p_{s_{i}}}\gamma_{i}&=\frac{\p \gamma_{i}}{\p s_{i}}=\sum_{l=1}^{2}\frac{\p \gamma_{i}^{\lambda l}}{\p s_{i}}\frac{\p }{\p x_{\lambda}^{l}},\\
\onab_{\p_{s_{i}}}^{2}\gamma_{i}&=\onab_{\p_{s_{i}}}\frac{\p \gamma_{i}}{\p s_{i}}
=\sum_{l=1}^{2}\Big(\frac{\p^{2} \gamma_{i}^{\lambda l}}{\p s_{i}^{2}}\frac{\p }{\p x_{\lambda}^{l}}
+\frac{\p \gamma_{i}^{\lambda l}}{\p s_{i}}\onab_{\p_{s_{i}}}\frac{\p }{\p x_{\lambda}^{l}}\Big)
=\sum_{l=1}^{2}\Big{\{}\frac{\p^{2} \gamma_{i}^{\lambda l}}{\p s_{i}^{2}}
+P_{2}^{l}\Big(\frac{\p \gamma_{i}^{\lambda 1}}{\p s_{i}}, \frac{\p \gamma_{i}^{\lambda 2}}{\p s_{i}}, \overline{\Gamma}\Big)\Big{\}}\frac{\p }{\p x_{\lambda}^{l}},
\end{align*}
where $P_{2}^{l}$ is some polynomial, and $\overline{\Gamma}$ means the Christoffel symbols of $\onab$. By induction, we obtain
\begin{align*}
\onab_{\p_{s_{i}}}^{m}\gamma_{i}&=\sum_{l=1}^{2}\Big{(}\frac{\p^{m} \gamma_{i}^{\lambda l}}{\p s_{i}^{m}}
+\widetilde{P}_{m}^{l}\Big{)}\frac{\p }{\p x_{\lambda}^{l}},\\
\widetilde{P}_{m}^{l}&=P_{m}^{l}\Big(\frac{\p^{m-1} \gamma_{i}^{\lambda 1}}{\p s_{i}^{m-1}}, \frac{\p^{m-1} \gamma_{i}^{\lambda 2}}{\p s_{i}^{m-1}},\ldots, \frac{\p \gamma_{i}^{\lambda 1}}{\p s_{i}}, \frac{\p \gamma_{i}^{\lambda 2}}{\p s_{i}}, \frac{\p^{m-2} \overline{\Gamma}}{\p (x_{\lambda}^{1})^{m-2}}, \frac{\p^{m-2} \overline{\Gamma}}{\p (x_{\lambda}^{2})^{m-2}},\ldots, \overline{\Gamma}\Big)
\end{align*}
where $P_{m}^{l}$ is some polynomial depending only on $n$ and $l$. 

We may assume that $\widehat{\gamma}_{i}(I_{\lambda})\subset V_{\lambda}$ for some compact set $V_{\lambda}\subset U_{\lambda}$ so that $|\p/\p x_{\lambda}^{l}|$ and  $|\p^{k} \overline{\Gamma}/\p (x_{\lambda}^{l})^{k}|$ $(k=1,2,\ldots,)$ are all bounded by positive constants independent of $t$ (but depend on $\lambda$).
Since $(U_{\lambda},\psi_{\lambda})$ is isothermal, we see
\begin{align*}
\Big|\frac{\p \gamma_{i}^{\lambda}}{\p s_{i}}\Big|^{2}=|\onab_{\p_{s_{i}}}\gamma_{i}|^2\cdot e^{-2\rho}\leq C_{1}(\lambda)
\end{align*}
where $e^{\rho}=|\p/\p x_{\lambda}^{1}|=|\p/\p x_{\lambda}^{2}|$. Hence, by induction and using Lemma \ref{lem:nabbd}, we obtain 
\begin{align*}
\Big|\frac{\p^{m} \gamma_{i}^{\lambda}}{\p s_{i}^{m}}\Big|^{2}=\Big|\onab_{\p_{s_{i}}}^{m}\gamma_{i}-\sum_{l=1}^{2}\widetilde{P}_{m}^{l}\frac{\p}{\p x_{\lambda}^{l}}\Big|^{2}\cdot e^{-2\rho}\leq |\onab_{\p_{s_{i}}}^{m}\gamma_{i}|^{2}\cdot e^{-2\rho}+\sum_{l=1}^{2}(\widetilde{P}_{m}^{l})^{2}\leq C_{m}(\lambda).
\end{align*}
Now, we take $C_{m}:={\rm max}\{C_{m}(1),\ldots, C_{m}(N)\}$, then we obtain $|\p_{s_{i}}^{m} \gamma_{i}^{\lambda}|\leq C_{m}$ for any $m$ and $\lambda$, as required. This proves \eqref{eq:keybd}.

We have proved that there exists a subsequence of $\{\widehat{\gamma}_{i}\}_{i=1}^{\infty}$ so that $\widehat{\gamma}_{i_{j}}^{\lambda}\to \widehat{\gamma}_{\infty}^{\lambda}$ smoothly for the fixed $\lambda$.
By taking a subsequence in order from $\lambda=1$ to $N$, we obtain a subsequence $\{\widehat{\gamma}_{i_{k}}\}_{k=1}^{\infty}$ such that $\widehat{\gamma}_{i_{k}}^{\lambda}\to \widehat{\gamma}_{\infty}^{\lambda}$ smoothly for any $\lambda=1,\ldots, N$. This implies that $\widehat{\gamma}_{\infty}$ is a smooth map and the convergence $\widehat{\gamma}_{i_{k}}\to \widehat{\gamma}_{\infty}$ is also smooth.  Moreover, we see
\[
\Big|\frac{\p \widehat{\gamma_{i}}}{\p \eta}\Big|=\Big|\frac{\p \gamma_{i}}{\p s_{i}}\cdot \frac{\p l_{i}^{-1}}{\p \eta}\Big|=\frac{L_{i}}{2\pi}\geq \frac{c}{2\pi}>0
\]
by Lemma \ref{lem:leng}. This implies that $|\p_{\eta}\widehat{\gamma}_{\infty}|>0$ and hence, $\widehat{\gamma}_{\infty}$ is an immersion. Furthermore, the immersion $\widehat{\gamma}_{\infty}$ must be $\psi$-minimal since the weighted mean curvature vector $K_{\psi}=\kappa_{\psi}N$ of $\widehat{\gamma}_{i_{k}}$ coincides with that of $\gamma_{i_{k}}$, and $\kappa_{\psi}(t)$ uniformly converges to $0$. 

It remains to show that the immersion  $\widehat{\gamma}_{\infty}$ is actually an embedding. This fact follows from the same argument given in \cite[Theorem 4.12]{Gage}, however,  we give a detailed proof for the convenience of the reader. Without loss of generality, we may assume that $\{\widehat{\gamma}_{i}\}_{i=1}^{\infty}$ smoothly converges to $\widehat{\gamma}_{\infty}$, and we parametrise $\widehat{\gamma}_{\infty}: [0,\widehat{l}]\to \mathcal{O}^{{\rm reg}}$  by the arclength parameter $\widehat{s}$ with respect to the weighted metric $g_\psi=e^{2\psi}g$, where $\widehat{l}$ is the (weighted) length of $\widehat{\gamma}_{\infty}:S^{1}\to \mathcal{O}^{{\rm reg}}$.
 
  If $\widehat{\gamma}_{\infty}$ transversally intersects itself, then so does $\widehat{\gamma}_{i}$ for sufficiently large $i$, and this contradicts the fact that $\widehat{\gamma}_{i}$ is an embedding. Thus, if $\widehat{\gamma}_{\infty}$ admits self-intersection points, then $\widehat{\gamma}_{\infty}$ tangentially intersects at every intersection point. On the other hand, since $\widehat{\gamma}_{\infty}$ is $\psi$-minimal, by taking the arclength parameter $\widehat{s}$, it satisfies the geodesic equation (w.r.t.\  $g_\psi$). Thus, the uniqueness of geodesic implies that the image $\Gamma:=\widehat{\gamma}_{\infty}(S^{1})$ is an embedded closed geodesic and $\widehat{\gamma}_{\infty}: S^{1}\to \Gamma$ must be a finite covering. In particular, we have $L_{\psi}(\widehat{\gamma}_{\infty})=mL_{\psi}(\Gamma)$ for some positive integer $m\geq 1$. 

Since $L_{\psi}(\widehat{\gamma}_{i})=L_{\psi}(\gamma_{i})$ is non-increasing, we may assume that $L_{\psi}(\widehat{\gamma}_{i})\leq aL_{\psi}(\widehat{\gamma}_{\infty})=amL_{\psi}(\Gamma)$ for any sufficiently large $i$, where $a\geq 2$ is arbitrary constant. Moreover, for any positive constant $\delta\leq \frac{\pi}{4amL_{\psi}(\Gamma)}$, we obtain that $|\widehat{\kappa}(t_{i})|\leq \delta$ for sufficiently large $i$ since $\kappa_{\psi}(t_{i})\to 0$ and $|e^{-\psi}|$ is bounded,  where $\widehat{\kappa}=e^{-\psi}{\kappa}_{\psi}$ is the geodesic curvature of $\widehat{\gamma}_{t_{i}}$ w.r.t.\ the weighted metric $g_{\psi}$. Then by Lemma \ref{lem:graph}, there exists an $\epsilon$-neighbourhood $U_{\epsilon}$ of $\Gamma$ such that  the embedding $\widehat{\gamma}_{i}$ is described by a graph over the embedded closed geodesic $\Gamma$ whenever $i$ is sufficiently large and  $\widehat{\gamma}_{i}$ is contained in $U_{\epsilon}$. Because $\widehat{\gamma}_{i}\to \widehat{\gamma}_{\infty}$, this implies that $\widehat{\gamma}_{i}$ is regarded as a single covering of $\Gamma$ for any sufficiently large $i$, and this shows that $\widehat{\gamma}_{\infty}: S^{1}\to \Gamma$ is also a single covering. Therefore, $\widehat{\gamma}_{\infty}: S^{1}\to \Gamma$ is bijective, and this proves that $\widehat{\gamma}_{\infty}:S^{1}\to \mathcal{O}^{{\rm reg}}$ is an embedding  since $\mathcal{O}^{{\rm reg}}$ is Hausdorff and $S^{1}$ is compact. This finishes the proof of (iii).\\

\noindent\textbf{
4. Proof of (i) and (ii).
}
Next, we consider the case where the continuous map $\gamma_{\infty}: S^{1}\to \mathcal{O}$ passes through some cone points  of $\mathcal{O}$. Then the map $\widehat{\gamma}_{\infty}$ does as well,  since its image coincides with that of $\gamma_{\infty}$.
We divide $S^{1}$ into a finite  closed subsets $S^{1}=\bigcup_{\lambda=1}^{N} S_{\lambda}$, where each $S_{\lambda}$ is homeomorphic to a connected closed interval, so that each $\widehat\gamma_{\infty}(S_{\lambda})$ contains at most one cone point and if there is a cone point in $\widehat\gamma_{\infty}(S_{\lambda})$, then the cone point belongs to the interior of $\widehat\gamma_{\infty}(S_{\lambda})$.
 Moreover, we define an $\epsilon$-neighbourhood $V_{\lambda}$ of the image $\widehat\gamma_{\infty}(S_{\lambda})$ by 
\[
V_{\lambda}:=\overline{\{x\in \mathcal{O}\mid \textup{$d_{\mathcal{O}}(x,y)<\epsilon$ for some $y\in \widehat\gamma_{\infty}(S_{\lambda})$}\}},
\]
where $d_{\mathcal{O}}$ is the distance function on $(\mathcal{O},g)$.
There are two possibilities:
\begin{itemize}
\item If $\widehat{\gamma}_{\infty}(S_{\lambda})$ does not contain any cone point,  then we  may assume that $V_{\lambda}\subset \mathcal{O}^{{\rm reg}}$. Moreover, by using a similar argument given in the proof of item (iii), we see that the restricted map $\widehat{\gamma}_{\infty}|_{S_{\lambda}}: S_{\lambda}\to \mathcal{O}^{{\rm reg}}$ is a smooth $\psi$-minimal immersion.

\item 
 If $\widehat{\gamma}_{\infty}(S_{\lambda})$ contains a cone point $x_{\lambda}\in \mathcal{O}$, we take an orbifold chart $(\widetilde{U}_{\lambda}, \mathbb{Z}_k, \varphi)$ around $x_\lambda$ on which a cyclic group $\mathbb{Z}_{k}$ acts  by a rotation centred at the origin $0\in \widetilde{U}\subset \R^{2}$.  By taking the partition of $S^{1}$ sufficiently fine and taking $\epsilon$ sufficiently small, we may assume that $\widehat\gamma_i(S_{\lambda})\subset V_{\lambda}\subset  \varphi(\widetilde{U}_{\lambda})$ for any large $i$ and $i=\infty$. 
Since $\widehat\gamma_{i}$ is an embedded curve contained in $\mathcal{O}^{\rm reg}$, the pre-image $\varphi^{-1}(\widehat\gamma_{i}(S_{\lambda}))$ consists of disjoint union of $k$ copies of $\widehat\gamma_{i}(S_{\lambda})$. More precisely, there exists a smooth embedded curve 
\[
\widetilde{\gamma}_{i\xi}: S_{\lambda}\to \widetilde{U}_{\lambda}\subset \R^{2},\quad  \textup{where $\xi=1,\ldots, k$}
\]
 such that each $\widetilde{\gamma}_{i\xi}$ is the lift of $\widehat\gamma_{i}|_{S_{\lambda}}$ and $\widetilde{\gamma}_{i\xi}\cap \widetilde{\gamma}_{i\xi'}=\emptyset$ if $\xi\neq \xi'$. Moreover,  the image of $\widetilde{\gamma}_{i\xi}$  is contained in the compact subset $\widetilde{V}_{\lambda}:=\varphi^{-1}(V_{\lambda})$ in $\widetilde{U}_{\lambda}$. 

We regard $\widetilde{U}_{\lambda}$ as a smooth Riemannian manifold equipped with a $\mathbb{Z}_{k}$-invariant Riemannian metric $\widetilde{g}$ associated with the orbifold Riemannian metric $g_{\mathcal{O}}$. 
Since $\widetilde{\gamma}_{i\xi}$ is a local lift of $\widehat\gamma_{i}$, it follows that there is a subsequence $\{\widetilde{\gamma}_{\alpha}\}_{\alpha=1}^{\infty}$ in the set $\{\widetilde{\gamma}_{i\xi}\}_{i=1,2\ldots, \xi=1,\ldots, k}$ which converges to a continuous map $\widetilde{\gamma}_{\infty}: S_{\lambda}\to  \widetilde{U}_{\lambda}$. Note that $\widetilde{\gamma}_{\infty}$ must be a lift of $\widehat\gamma_{\infty}|_{S_{\lambda}}$. Moreover, by using a similar argument given in the proof of (iii), we see that $\widetilde{\gamma}_{\infty}: S_{\lambda}\to  \widetilde{U}_{\lambda}$ is a $\psi$-minimal immersion.
\end{itemize}

Summarizing the argument, we see  that if $\widehat{\gamma}_{\infty}(S_{\lambda})\cap \mathcal{O}^{\rm sing}= \emptyset$ then $\widehat{\gamma}_{\infty}|_{S_{\lambda}}$ is a $\psi$-minimal immersion into $\mathcal{O}^{\rm reg}$ and if $\widehat{\gamma}_{\infty}(S_{\lambda})\cap \mathcal{O}^{\rm sing}\neq \emptyset$, then there exists a local lift $\widetilde{\gamma}_{\infty}: S_{\lambda}\to \widetilde{U}_{\lambda}$ of $\widehat{\gamma}_{\infty}$ around the singular point such that $\widetilde{\gamma}_{\infty}$ is a $\psi$-minimal immersion into $\widetilde{U}_{\lambda}$. 

Now, we give a proof of  items  (i) and (ii).  

(i) Suppose that $\gamma_{\infty}$ (or $\widehat{\gamma}_{\infty}$) passes through a cone point $x_{0}\in \mathcal{O}$ and the order $k$ of $x_{0}$ satisfies $k\geq 3$. We may assume that there exists a $\lambda$ such that $\widehat{\gamma}_{\infty}(\theta_{0})=x_{0}$ for some interior point $\theta_{0}$ in $S_{\lambda}$. Let $\widetilde{\gamma}_{\infty}: S_{\lambda}\to \widetilde{U}_{\lambda}$ be the local lift of $\widehat{\gamma}_{\infty}$  constructed above.
 Since $\widehat\gamma_{\infty}(\theta_0)=x_0$, we have $\widetilde{\gamma}_{\infty}(\theta_{0})=0\in \widetilde{U}_{\lambda}$, i.e.\ $\widetilde{\gamma}_\infty$ passes through the origin $0$.  Moreover, we note that $\p_{\theta}\widetilde{\gamma}_{\infty}(\theta_{0})\neq 0$ since $\widetilde{\gamma}_{\infty}$ is an immersion. Let $R\in \mathbb{Z}_{k}$ be a generator of $\mathbb{Z}_{k}$ represented by the $2\pi/k$-rotation centred at $0\in \widetilde{U}_{\lambda}$. Since $k\geq 3$, it follows that the rotated curve $R\circ \widetilde{\gamma}_{\infty}: S_{\lambda}\to \widetilde{U}_{\lambda}$  transversally intersects to $\widetilde{\gamma}_{\infty}$ at $0=\widetilde{\gamma}_{\infty}(\theta_{0})=R\circ \widetilde{\gamma}_{\infty}(\theta_{0})$. Note that the intersection point $0$ is not a boundary point of  $\widetilde{\gamma}_{\infty}(S_{\lambda})$ since $\theta_{0}$ is an interior point of $S_{\lambda}$. This implies that if we take a sufficiently large $\alpha$, then two embedded curves $\widetilde{\gamma}_{\alpha}$ and $R\circ \widetilde{\gamma}_{\alpha}$ also have an intersection since $\widetilde{\gamma}_{\alpha}$ (resp. $R\circ \widetilde{\gamma}_{\alpha}$) uniformly converges to $\widetilde{\gamma}_{\infty}$ (resp. $R\circ \widetilde{\gamma}_{\infty}$).  In particular, the lift $\widetilde{\gamma}_{\alpha}$ also intersects to $R\circ \widetilde{\gamma}_{\alpha}$ at some interior point. However, this contradicts the fact that the both  $\widetilde{\gamma}_{\alpha}$ and $R\circ \widetilde{\gamma}_{\alpha}$ are lifts of $\widehat{\gamma}_\alpha|_{S_{\lambda}}$ that means they must be disjoint to each other. Hence,  $\widehat\gamma_{\infty}$ can not pass through a cone point of order $\geq 3$. Since the image of $\gamma_\infty$ coincides with that of $\widehat{\gamma}_\infty$, this proves  (i).

(ii) Finally, we suppose that  $\gamma_{\infty}$ passes through a cone point $x_{0}\in \mathcal{O}$ of order $k=2$. Our first claim is that $\gamma_{\infty}$ is ``reflected'' at $x_{0}$, namely, the curve $\gamma_{\infty}$ turns back on the same curve when it reaches the point $x_{0}$. To show this, we consider the image of the local lift $\widetilde{\gamma}_{\infty}: S_{\lambda}\to \widetilde{U}_{\lambda}$ of $\widehat{\gamma}_{\infty}$. We shall show that the image of $\widetilde{\gamma}_\infty$ is $\mathbb{Z}_2$-invariant around the origin $0\in \widetilde{U}_\lambda$.

We denote the induced $\mathbb{Z}_{2}$-invariant weighted metric on $\widetilde{U}_{\lambda}$ by $\widetilde{g}_{\psi}=e^{2\widetilde{\psi}}\widetilde{g}$, where $\widetilde{\psi}$ is a $\mathbb{Z}_{2}$-invariant smooth function which is a lift of $\psi$. 
Since $\widetilde{\gamma}_{\infty}$ is an immersion, we can take the arclength parameter $\widetilde{s}\in [0, \widetilde{l}]$ of $\widetilde{\gamma}_{\infty}: S_{\lambda}\to \widetilde{U}_{\lambda}$ with respect to $\widetilde{g}_{\psi}$, where $\widetilde{l}$ is the length of $\widetilde{\gamma}_{\infty}$.
Since $\gamma_{\infty}$ passes through the cone point $x_{0}$, there exists an $s_{0}\in (0,\widetilde{l})$ such that $\widetilde{\gamma}_{\infty}(s_{0})=0$. 
Let $R$ be the generator of $\mathbb{Z}_{2}$, i.e.\ the $\pi$-rotation centred at $0$. Then, the rotated curve $R\circ \widetilde{\gamma}_{\infty}(\widetilde{s})$ is a smooth immersion such that 
\[R\circ \widetilde{\gamma}_{\infty}(s_{0})=0\quad {\rm and}\quad \p_{\widetilde{s}}(R\circ \widetilde{\gamma}_{\infty})(s_{0})=R(\p_{\widetilde{s}}\widetilde{\gamma}_{\infty}(s_{0}))=-\p_{\widetilde{s}}\widetilde{\gamma}_{\infty}(s_{0}).\]
 Since $\widetilde{\gamma}_{\infty}$ is a $\psi$-minimal immersion, $\widetilde{\gamma}_{\infty}(\widetilde{s})$ is a geodesic with respect to $\widetilde{g}_{\psi}$.  Moreover, $R\circ \widetilde{\gamma}_{\infty}(\widetilde{s})$ is also a geodesic because the metric $\widetilde{g}_{\psi}$ is $\mathbb{Z}_{2}$-invariant.  Thus, by the uniqueness of geodesic, it holds that $R\circ \widetilde{\gamma}_{\infty}(\widetilde{s})=\widetilde{\gamma}_{\infty}(2s_{0}-\widetilde{s})$ on the interval $[s_{0}-\epsilon, s_{0}+\epsilon]$ for some small $\epsilon>0$. In particular, the image $\widetilde{\gamma}_{\infty}([s_{0}-\epsilon, s_{0}+\epsilon])\subset \widetilde{U}_{\lambda}$ is $\mathbb{Z}_{2}$-invariant. Since $\widetilde{\gamma}_{\infty}$ is a lift of  $\widehat{\gamma}_{\infty}$,   this implies that the curve $\widehat{\gamma}_{\infty}: S^{1}\to \mathcal{O}$ is reflected at the cone point $x_{0}$, as required.
 
 Next,  let  $\overline{\gamma}_{\infty}(\theta):=\widehat\gamma_{\infty}(2\pi-\theta)$ be the reversed curve of $\widehat{\gamma}_{\infty}$. We can take the arclength parameter $\widehat{s}$ w.r.t.\  $g_{\psi}=e^{2\psi}g$ defined on some small interval $[0,\epsilon)$ so that $\widehat{\gamma}_{\infty}(0)=x_{0}$. Then it holds that $\widehat\gamma_{\infty}(\widehat{s})=\overline{\gamma}_{\infty}(\widehat{s})$ for any $\widehat{s}\in [0,\epsilon)$ since $\widehat{\gamma}_{\infty}$ reflects at $x_{0}$. Our second claim is that this can be extended to whole $[0,\widehat{l}]$, where $\widehat{l}$ is the length of $\widehat{\gamma}_{\infty}$ (in other words, two curves $\widehat\gamma_{\infty}$ and $\overline{\gamma}_{\infty}$ are never branched).
This also follows from the uniqueness of geodesics. Indeed, since both $\widehat{\gamma}_{\infty}$ and $\overline{\gamma}_{\infty}$ are $\psi$-minimal immersions near any regular point and $\widehat\gamma_{\infty}(\widehat{s})=\overline{\gamma}_{\infty}(\widehat{s})$ on $[0,\epsilon)$, the uniqueness of geodesics implies that we have $\widehat\gamma_{\infty}(\widehat{s})=\overline{\gamma}_{\infty}(\widehat{s})$ on $[0,l_{1})$ as long as  $\widehat\gamma_{\infty}((0,l_{1}))\subset \mathcal{O}^{\rm reg}$ holds. 
Moreover, if there were no singular point in the interval $(0,\widehat{l}]$, then the geodesic $\widehat\gamma_{\infty}$ could not go back via the same trajectory, and this would contradict the fact that $\widehat\gamma_\infty=\overline{\gamma}_\infty$ around $x_0$. Therefore, there must be a cone point $y_{0}$ (where possibly $y_0=x_0$) in the interval $(0,\widehat{l})$, and two curves $\widehat\gamma_{\infty}(\widehat{s})$ and $\overline{\gamma}_{\infty}(\widehat{s})$ reach the point $y_{0}$ at the same time. Let $l_{1}>0$ be the first arrival time. Then,  by (i), the order of $y_{0}$ must be equal to $2$, and thus both $\widehat\gamma_{\infty}$ and $\overline{\gamma}_{\infty}$ reflect again at $y_{0}$. Thus, it holds that $\widehat\gamma_{\infty}(\widehat{s})=\overline{\gamma}_{\infty}(\widehat{s})$ beyond the time $l_{1}$, and both curves go back along the same trajectories because of the uniqueness of geodesics. This shows that  they move back and forth on the geodesic path between $x_{0}$ and $y_{0}$. Since the image of $\widehat{\gamma}_{\infty}$ coincides with that of $\gamma_{\infty}$, this proves (ii), and the proof is complete.
 \end{proof}
 
 \section{Finite-time behaviour of weighted CSF in compact 2D orbifolds}\label{sec:csf2}
 
 The well-known  theorem of Grayson \cite{Grayson} asserts that the (usual) CSF on a closed oriented surface either shrinks to a point in finite time or exists for infinite time.  Furthermore, using the general theory developed by Angenent \cite{Ang}, Oaks \cite{Oaks} generalised Grayson's theorem  to more general flows, including the weighted CSF on a smooth surface. More precisely, we have the following theorem as a special case of Oaks's result.

\begin{theorem}[cf.\ Corollary 6.2 in \cite{Oaks}]\label{thm:Oaks}
Let $(M,g, \psi)$ be a $2$-dimensional smooth orientable complete weighted Riemannian manifold such that both $\overline{\nabla}\psi$ and $\overline{\nabla}^{2}\psi$ are bounded, and $\gamma_{0}: S^{1}\to M$ be a $C^{2}$-embedded curve. Then the solution of $\psi$-CSF \eqref{eq:CSF} starting from $\gamma_{0}$ either shrinks to a point in finite time or exists for infinite time.
\end{theorem}

It should be noted that the regular part of an orbifold $\mathcal{O}^{{\rm reg}}$ may not be a complete manifold in general, therefore we {\it cannot} apply Theorem \ref{thm:Oaks} directly to the weighted CSF contained in $\mathcal{O}^{{\rm reg}}$. In this section, we aim to prove the following Grayson type theorem for the weighted CSF contained in $\mathcal{O}^{{\rm reg}}$,  which says that at a finite-time singularity of a $\psi$-CSF in a $2$-orbifold, the flow must disappear at a point.

\begin{theorem}\label{thm:beh}
Let $(\mathcal{O},g, {\psi})$ be a closed, orientable $2$-dimensional weighted Riemannian orbifold and $\gamma: S^{1}\to \mathcal{O}^{{\rm reg}}$ be a smooth embedding into the regular part of $\mathcal{O}$. We denote the maximal existence time of the $\psi$-CSF $\gamma_{t}$ starting from $\gamma=\gamma_{0}$ in $\mathcal{O}^{{\rm reg}}$ by $T_{\max}$.
If $T_{\max}<\infty$, then $\lim_{t\to T_{\max}}L_\psi(\gamma_t)=0$.
\end{theorem}

To prove this theorem, we establish the following result.

\begin{proposition}\label{prop:key}
Suppose the same assumptions given in Theorem \ref{thm:beh} with $T_{\max}<\infty$. Furthermore, we assume that the following two conditions hold: 
\begin{itemize}
\item[(a)] $\gamma:S^{1} \times [0,T_{\max})\to \mathcal{O}^{{\rm reg}}$ may be continuously extended to a map $\gamma:  S^{1}\times [0,T_{\max}]\to \mathcal{O}$.
\item[(b)] $\gamma_{T_{\max}}(S^1)\cap \mathcal{O}^{{\rm sing}}\neq \emptyset$.
\end{itemize}
Then, we have $\lim_{t\to T_{\max}}L_\psi(\gamma_t)=0$.
\end{proposition}

In other words, the weighted CSF starting from a regular curve on a compact oriented $2$-dimensional orbifold $\mathcal{O}$ does not reach continuously to any singular point in finite time, unless the weighted length converges to $0$.

Theorem \ref{thm:beh} follows from Proposition \ref{prop:key} together with Theorem \ref{thm:Oaks}, as follows:

\begin{proof}[Proof of Theorem \ref{thm:beh}]
Suppose that $T_{\max}<\infty$. Since $L_{\psi}(t)$ is a decreasing function by \eqref{eq:fv} and $L_{\psi}(t)>0$ for any $t\in [0,T_{\max})$, there exists a limit $L_{\psi}(T_{\max}):=\lim_{t\to T_{\max}}L_{\psi}(t)$. Our claim is that $L_{\psi}(T_{\max})=0$. 

Suppose that the contrary were true, i.e.\ $L_{\psi}(T_{\max})>0$. Then a similar argument given in \cite[Theorem 6.2]{Ang} shows that $\gamma_{t}: S^{1}\to \mathcal{O}^{{\rm reg}}$ converges to a continuous curve $\gamma_{T_{\max}}: S^{1}\to \mathcal{O}$ in the $C^{0}$-topology (Notice that the limiting curve is not necessarily contained in the regular part $\mathcal{O}^{{\rm reg}}$). In particular,  the $\psi$-CSF contained in $\mathcal{O}^{{\rm reg}}$ may be continuously extended to a map $\gamma:S^{1}\times [0,T_{\max}] \to \mathcal{O}$. Moreover, we have $L_{\psi}(T_{\max})\neq 0$ by assumption.  Therefore, by Proposition \ref{prop:key}, it must hold that $\gamma_{T_{\max}}(S^1)\cap \mathcal{O}^{{\rm sing}}=\emptyset$.  Namely, the flow $\gamma_{t}$ is contained in the regular part $\mathcal{O}^{{\rm reg}}$ for every $t\in [0,T_{\max}]$, including the maximal existence time $T_{\max}$. This implies that there exists a compact subset $D\subset \mathcal{O}^{{\rm reg}}$ such that $\gamma(S^{1}\times [0,T_{\max}])$ is contained in the interior $D^{int}$ of $D$.  

 Since $\mathcal{O}^{{\rm sing}}$ consists of finitely many cone points, we can remove small neighbourhoods of these singular points (each homeomorphic to an open disk) and smoothly cap off the resulting boundary components to obtain a closed orientable surface $M$ so that $D\subset M$ (cf.\ \cite[Theorem 9.29]{Lee}). Moreover,  the weighted Riemannian metric on $D$ can be smoothly extended to $M$.
Let $\widetilde{\gamma}_{t}: S^{1}\times [0,\widetilde{T}_{\max})\to M$ be the maximal solution of the weighted CSF  in $M$ with $\widetilde{\gamma}_{0}=\gamma_{0}$. By the uniqueness of solution, it must hold that $\widetilde{\gamma}_{t}=\gamma_{t}$ for any $t\in [0,T_{\max})$ since $\gamma(S^{1}\times [0,T_{\max}))\subset D^{int}$. If $\widetilde{T}_{\max}>T_{\max}$, then $\gamma_{t}$ can be extended beyond $T_{\max}$ and this is a contradiction. Therefore, we have $\widetilde{T}_{\max}\leq T_{\max}<\infty$.  However,  Theorem \ref{thm:Oaks} shows that when $\widetilde{T}_{\max}<\infty$, then $L_{\psi}({\gamma}_{t})=L_{\psi}(\widetilde{\gamma}_{t})\to 0$  as $t\to \widetilde{T}_{\max}$ (see also Remark \ref{rem:key1}), and this contradicts the assumption that $L_{\psi}(T_{\max})>0$. Therefore, we obtain $L_{\psi}(T_{\max})=0$. This proves Theorem \ref{thm:beh}.
\end{proof}

In the rest of this section, we prove Proposition \ref{prop:key}.

 \subsection{Distance function}\label{ssec:dist}
Let $(\mathcal{O},g,\psi)$ be a compact, oriented $2$-dimensional weighted Riemannian orbifold. In this subsection, we summarize preliminary results concerning the distance function between two $\psi$-CSFs, most of which were proved in \cite[Section 3, 4]{Ma} for the usual CSF on a Riemann surface. 

Throughout this section, we set $\widehat{g}:=g_\psi$ and use the hat accent $\widehat{\,\,}$ to denote the geometric quantities with respect to the weighted metric $\widehat g$, e.g.\ $\widehat{d}(x,y)$ denotes the distance between two points measured by $g_{\psi}$.

Let $\Gamma^1, \Gamma^2\colon S^1 \times [0,T) \to \mathcal{O}^{{\rm reg}}$  be two $\psi$-CSFs contained in $\mathcal{O}^{{\rm reg}}$. We define the {\it distance function between $\Gamma^{1}$ and $\Gamma^{2}$} by 
\[
d_\psi: S^1 \times S^1 \times [0,T) \to \mathbb{R}_{\geq 0}, \quad d_\psi(u_1, u_2, t) := \widehat{d}(\Gamma^1(u_1, t), \Gamma^2(u_2, t)).
\]
 In this section, we collect some formulas and properties related to distance function, which will be used in the next subsection. 

For the remainder of the section, we assume that there exists a neighbourhood $\widehat U \subset \mathcal{O}^{\rm reg}$ which is strongly convex with respect to $\widehat g$, a constant $0 < T_0 < T$ and fixed closed intervals $I_i$ ($i=1,2$) such that the image of the restricted maps $\Gamma^i: I_i\times [T_0,T)\to \mathcal{O}^{{\rm reg}}$ are contained in $\widehat{U}$ with respect to $g_{\psi}$. Here, $\widehat{U}$ is {\it strongly convex} if for any $p,q\in \widehat{U}$ there exists a unique minimizing geodesic (with unit speed) between $p$ and $q$. It is known that we can take such a strongly convex neighbourhood around any point on a smooth Riemannian manifold (see \cite[Ch.\ IV, Theorem 5.3]{Sakai}).

\begin{lemma}\label{lem:derd}
The square of the distance function $\widehat{d}$ is smooth on $\widehat{U}\times \widehat{U}$. In particular, for any $x,y\in \widehat{U}$ with $x\neq y$,   the function $\widehat{d}=\sqrt{\widehat{d}^2}$ is smooth around $(x,y)$. 
\end{lemma}
\begin{proof}
See \cite[Ch.\ IV, Theorem 3.6]{KN}.
\end{proof}

\begin{remark}\label{rem:derd}
Lemma \ref{lem:derd} does not hold in general near a singular point of an orbifold. Indeed, a singular point need not admit a strongly convex neighbourhood with respect to the orbifold distance function $d$, and $d$ itself may fail to be differentiable. This difference requires some additional arguments in the proof of Proposition \ref{prop:distcomp} (given in Appendix \ref{app:distcomp}). 
\end{remark}

We denote the unit tangent vector field (resp. normal vector field) along $\Gamma^{i}_{t}$ with respect to the weighted metric $\widehat{g}$ by $\widehat{T}_i(t)$ (resp. $\widehat{N}_{i}(t)$), and the geodesic curvature of $\Gamma^{i}_{t}$ with respect to $g_{\psi}$ is given  by  $\widehat{\kappa}_{i}(t)$. Also, the arclength parameter of $\Gamma^{i}_{t}$ w.r.t.\ $\widehat{g}$ is denoted by $\widehat{s}_{i}(t)$. Note that by Remark \ref{rem:geocurv}, $\widehat{\kappa}_{i}(t)$ is related to the $\psi$-geodesic curvature of $\Gamma^{i}_{t}$ by $\widehat \kappa = e^{-\psi}\kappa_\psi$.

For $(u_1, u_2, t)\in I_1\times I_2\times [T_0,T)$, let 
\begin{equation*}
\Phi(\cdot; u_1, u_2, t) \colon [0,1] \to \widehat{U}
\end{equation*}
be the minimizing geodesic from $\Gamma^1_t(u_1)$ to $\Gamma^2_t(u_2)$ with constant speed $|\dot\Phi(\alpha; u_1, u_2, t)|_\psi=d_\psi(u_1,u_2,t)$, where $|,|_\psi$ is the norm with respect to $\widehat{g}=g_{\psi}$. 
For our purpose, we suppose that $\Gamma^1_t(u_1)\neq \Gamma^2_t(u_2)$ so that $d_\psi(u_1,u_2,t)>0$. Note that if this is the case, the  function $d_\psi$ is differentiable at $(u_1,u_2,t)$ by Lemma \ref{lem:derd}.

 We define an orthonormal frame $\{\tau(\alpha), \nu(\alpha)\}$ along the geodesic $\Phi(\alpha)=\Phi(\alpha; u_1,u_2,t)$ by 
\begin{equation*}
\tau (\alpha) = \frac{  \dot{\Phi}(\alpha)  }{  \sqrt{\widehat{g} \left( \dot{\Phi}(\alpha), \dot{\Phi}(\alpha) \right) }  } = \frac{ \dot{\Phi} (\alpha) }{d_\psi(u_1, u_2, t) }
\end{equation*}
and $\nu(\alpha)$ is equal to a unit normal vector at $\Phi(\alpha)$.
Note that, since $\mathcal{O}$ is $2$-dimensional, $\nu(\alpha)$ is a parallel normal vector field along $\Phi$. We define
$
\tau_1:=\tau(0), \nu_1:=\nu(0) 
$
and $\tau_2:=\tau(1), \nu_2:=\nu(1)$.

We also consider the (normal) Jacobi field $Y(\alpha)$ along $\Phi$. Namely, we consider a function $J: [0,1]\to \R$ satisfying the following second order ordinary differential equation:
\begin{align}\label{eq:Jacobi}
J''(\alpha) + d^2 \widehat{K}(\Phi(\alpha)) J(\alpha) = 0,
\end{align}
where $d:=d_\psi(u_1,u_2,t)$ and $\widehat{K}$ is the Gaussian curvature of $(\widehat{U}, \widehat g)$. Note that \eqref{eq:Jacobi} is equivalent to the Jacobi equation of the form $Y(\alpha)=J(\alpha)\nu(\alpha)$.

The following lemma is a refinement of \cite[Lemma 3.2]{Ma}. 

\begin{lemma}\label{lem:J}
Suppose that $\Gamma^1(u_1, t)$ and $\Gamma^2(u_2, t)$ are $\psi$-CSFs in $\mathcal{O}^{\rm reg}$, and $\widehat U \subset \mathcal{O}^{\rm reg}$ is a $\widehat g$-strongly convex neighbourhood such that $\Gamma^1(u_1, t), \Gamma^2(u_2,t)$ are contained in $\widehat U$ for all $(u_1,u_2,t) \in I_1 \times I_2 \times [T_0,T)$. Fix a point $(u_1, u_2, t)\in I_1\times I_2\times [T_0,T)$ with $\Gamma^1_t(u_1) \neq \Gamma^2_t(u_2)$, and let $\Phi$ be the geodesic as above.
 
Let $J_i:[0,1]\to \R$ $(i=1,2)$ be the function satisfying \eqref{eq:Jacobi} with initial conditions
\begin{align*}
 J_1(1) = 0,\ J'_{1}(1) = - d,\quad J_{2}(0)=0,\ J'_2(0) = - d,
\end{align*} 
We put $\widehat{K}_{{\mathcal{O}}}:={\rm sup}_{\mathcal{O}}|\widehat{K}|>0$. 
Then if $0<d<\frac{1}{2\sqrt{\widehat{K}}_{\mathcal{O}}}$, we have
\begin{equation}\label{eq:J}
J_i'(\alpha) \leq -\frac{d}{2}<0,\quad  J_1(0)  \geq \frac{d}{2}>0, \quad J_2(1)  \leq -\frac{d}{2}<0.
\end{equation}

\end{lemma}

\begin{proof}
We set $Y_i(\alpha)=J_i(\alpha)\nu(\alpha)$. Since $Y_2$ is the Jacobi field with $Y_2(0)=0$ and $|\nabla Y_2|(0)=|J_2'(0)|=d$, the Rauch comparison theorem (see \cite[Theorem 2.3]{Sakai}) shows that we have
\[
\abs{J_2(\alpha)} \leq \frac{1}{\sqrt{\widehat{K}_{\mathcal{O}}}} \sinh \Big(d\sqrt{\widehat{K}_{\mathcal{O}}}\cdot  \alpha \Big ),
\]
where the right-hand side is the norm of the comparison Jacobi field in the $2$-dimensional hyperbolic space  of constant Gaussian curvature $-\widehat{K}_{\mathcal{O}}$. By taking $\widetilde{J}_1(\beta):=J_1(1-\beta)$, we see that $\widetilde{Y}_1(\beta)=\widetilde{J}_1(\beta)\nu(1-\beta)$ is a Jacobi field along the reversed geodesic with $\widetilde{Y}_1(0)=0$ and $|\nabla \widetilde{Y}_1|(0)=|J_1'(1)|=d$, and hence, $\widetilde{J}_1(\beta)$ satisfies the same inequality.  Therefore, in either case, we obtain the following bound:
\[
\abs{J_i(\alpha)} \leq \frac{1}{\sqrt{\widehat{K}_{\mathcal{O}}}} \sinh \Big(d\sqrt{\widehat{K}_{\mathcal{O}}} \Big ).
\]
Using Taylor's theorem for $\sinh$,  we see that there exists $\theta \in (0,1)$ for which
\begin{align*}
\abs{J_i(\alpha)} 
&\leq \frac{1}{\sqrt{\widehat{K}_{\mathcal{O}}}} \Big{\{} d \sqrt{\widehat{K}_{\mathcal{O}}} + \frac{1}{3!} \cosh\Big( \theta\cdot d \sqrt{\widehat{K}_{\mathcal{O}}}\Big) \Big(d\sqrt{\widehat{K}_{\mathcal{O}}}\Big)^3   \Big{\}} \\
 &\leq d+ \frac{1}{6} \exp\Big(d \sqrt{\widehat{K}_{\mathcal{O}}} \Big{)} \cdot d^3\widehat{K}_{\mathcal{O}}
 \leq d+d^3\widehat{K}_{\mathcal{O}},
\end{align*}
 where we use the assumption that $d<\frac{1}{2\sqrt{\widehat{K}}_{\mathcal{O}}}<\frac{1}{ \sqrt{\widehat{K}_{\mathcal{O}}}}\ln 6$ for the second line. Thus, we have
\begin{equation*}
\abs{J_i'(\alpha) - J_i'(\beta)} = \abs{\int_\alpha^\beta J''_i(t)\,dt }= d^2 \abs{\int_{\alpha}^\beta \widehat{K} J_i(t) d t  } \leq d^2 \widehat{K}_{\mathcal{O}} \left( d+ d^3\widehat{K}_{\mathcal{O}}  \right)
\end{equation*} 
for any $\alpha, \beta\in [0,1]$,
and hence,
\begin{equation*}
J_i'(\alpha) \leq - d + d^3 \widehat{K}_{\mathcal{O}} + d^5 \widehat{K}_{\mathcal{O}}^2 
\end{equation*}
since $J_1'(1) =J_2'(0)= - d$. Using $d<\frac{1}{2\sqrt{\widehat{K}_{\mathcal{O}}}}$, it is easy to see that we have $J_i'(\alpha)\leq -d/2$. Furthermore, we see
\begin{align*}
J_1(0) &= J_1(0) - J_1(1)= -{\int_0^1 J_1'(\alpha)\, d \alpha} \geq \frac{ d }{2}.
\end{align*}
Similarly, $J_2(1) = J_2(1)-J_2(0) \leq -\frac{ d }{2}$. 
This proves the lemma.
\end{proof}

We can take the convex set $\widehat{U}$ sufficiently small  so that the diameter of $\widehat{U}$ is smaller than the constant $1/(2\sqrt{\widehat{K}_\mathcal{O}})$.   Then we have $d_\psi(u_1,u_2,t)<1/(2\sqrt{\widehat{K}_\mathcal{O}})$ for any $(u_1,u_2,t)\in I_1\times I_2\times [T_0,T)$ and hence, we may assume that \eqref{eq:J} holds on the convex subset $\widehat{U}$.

In this situation, with the above notations, the derivatives of the distance function are computed as follows.
The following formulas are extensions of \cite[Lemma 3.1]{Ma}. 

\begin{lemma}\label{lem:formulas}
Assume we are in the setting of Lemma \ref{lem:J}, and $J_i \colon [0,1] \to \mathbb{R}$ ($i=1,2$) are functions satisfying the ODE \eqref{eq:Jacobi} with initial conditions
\begin{align*}
 J_1(1) = 0,\ J'_{1}(1) = - d,\quad J_{2}(0)=0,\ J'_2(0) = - d,
\end{align*}
where $d=d_\psi(u_1,u_2,t)$.

Then at the point $(u_1,u_2,t)$ we have the following formulas:
\begin{align*}
\frac{\p d_\psi}{\p t}
&=  \widehat{g} \left(  e^{2{\psi}(\Gamma_t^2(u_2))} \widehat{\kappa}_2(u_2, t) \widehat{N}_2,  \tau_2 \right) -  \widehat{g} \left(  e^{2 {\psi} (\Gamma_t^1(u_1))} \widehat{\kappa}_1(u_1, t) \widehat{N}_1,  \tau_1  \right), \\
\frac{\p d_\psi}{\p \widehat{s}_1} 
&=  - \widehat{g} \left(\widehat{T}_1, \tau_1 \right),\quad
\frac{\p  d_\psi}{\p \widehat{s}_2} 
= \widehat{g} \left(\widehat{T}_2, \tau_2 \right), \\
\frac{\p^{2} d_\psi}{\p \widehat{s}^{2}_1} 
&=  - \widehat{g} \left( \widehat{\kappa}_1(u_1, t) \widehat{N}_1 , \tau_1 \right) 
- \frac{1}{d_\psi}\widehat{g} \left(\widehat{T}_1, \nu_1 \right)^2 \frac{J'_1(0)}{J_1(0)}, \\
\frac{\p^{2} d_\psi}{\p \widehat{s}^{2}_1} 
&=  \widehat{g} \left( \widehat{\kappa}_2(u_2, t) \widehat{N}_2 , \tau_2 \right) 
+ \frac{1}{d_\psi} \widehat{g} \left(\widehat{T}_2, \nu_2 \right)^2 \frac{J'_2(1)}{J_2(1)},  \\
\frac{\p^{2} d_\psi}{\p \widehat{s}_1\p\widehat{s}_{2}} 
&= -  \frac{1}{d_\psi}  \widehat{g} \left(\widehat{T}_1, \nu_1 \right) \widehat{g} \left(\widehat{T}_2, \nu_2 \right) \frac{J'_2(0)}{J_2(1)}
=  \frac{1}{d_\psi} \widehat{g} \left(\widehat{T}_1, \nu_1 \right) \widehat{g} \left(\widehat{T}_2, \nu_2 \right) \frac{J'_1(1)}{J_1(0)}.
\end{align*}
\end{lemma}

\begin{proof}
The derivatives with respect to $\widehat{s}_i$ ($i=1,2$) are completely the same ones given in \cite[Lemma 3.1]{Ma} (see also \cite[Theorem 3.1]{Gage}), where we take the Riemannian metric on a surface by $g_\psi$. On the other hand, the derivative with respect to $t$ (i.e.\ in the direction of the flow) is different from \cite{Ma} because we deal with {\it weighted} CSF, whereas in \cite{Ma} the derivative is computed under the usual CSF. However, because the computation is similar to that of \cite[Section 3, pp.5--6]{Ma}, we omit the proof.
\end{proof}

\begin{remark}\label{rem:wronskian}
The equivalence of the alternative expressions for $\frac{\partial^2 d_\psi}{\partial \widehat s_1 \widehat s_2}$ comes from equality of mixed partials, but also follows from the following observation. Since $J_1, J_2$ are both solutions to the ODE \eqref{eq:Jacobi} which has no first-order term, the Wronskian $W(\alpha) = J_1(\alpha)J_2'(\alpha) - J_2(\alpha)J_1'(\alpha)$ is constant. In particular, $W(0) = W(1)$, which together with \eqref{eq:Jacobi} gives
\begin{equation}
    J_1(0)J_2'(0) = -J_2(1)J_1'(1).\label{eq:Jidentity}
\end{equation}
\end{remark}

 The following proposition is proved by using Lemma \ref{lem:formulas}.

\begin{proposition}\label{prop:distcomp}
Let $\gamma: S^1\times [0,T_{\max})\to \mathcal{O}^{{\rm reg}}$ be a $\psi$-CSF starting from an embedded curve $\gamma_0:S^1\to \mathcal{O}^{{\rm reg}}$, where $T_{\max}$ is the maximal existence time of the $\psi$-CSF in $\mathcal{O}^{{\rm reg}}$. We consider the following quantity:
\begin{equation}\label{def:R}
R(t) := \sup_{x,y\in S^{1};\, x \neq y} R(x,y,t) := \sup_{x,y\in S^{1};\, x \neq y} \left[ \frac{L_{\psi}(t) e ^{-\mathcal{K} t} }{\pi {d}_{\psi}(x,y,t)} \sin \left( \frac{\pi l_{\psi}(x,y,t)}{L_{\psi}(t)} \right) \right],
\end{equation}
where  $\mathcal{K}$ is a non-negative constant, $d_{\psi}(x,y,t):=d_{{\psi}}(\gamma_{t}(x), \gamma_{t}(y))$, $l_\psi(x,y,t)$ is an intrinsic distance between $x,y\in S^1$ with respect to the induced metric by $\gamma_t$, and $L_\psi$ is the length functional of \eqref{eq:lengthfunctional}.

Suppose $L_{\psi}(T_{\max}):=\lim_{t\to T_{\max}}L_\psi(t) \geq \delta_0 > 0$ for a positive constant $\delta_0 > 0$.
If we choose a constant $\mathcal{K}$ sufficiently large, 
then we have
\begin{equation}\label{eq:R}
\sup_{t \in [0, T_{\max})} R(t) < \infty.
\end{equation}
\end{proposition}

\begin{remark}
{\rm 
The quantity $R(t)$ is a generalisation of {\it Huisken's comparison function} (\cite{Hui}), first introduced by Ma \cite{Ma}. Huisken proved that, for the CSF in $\R^2$,  the quantity $R(t)$ (with $\mathcal{K}=0$) is a non-increasing function; this gives \eqref{eq:R} for the CSF in $\R^2$. Moreover, by using the comparison function, Huisken provided a new proof of Grayson's theorem for the CSF in $\R^2$  (\cite[Theorem 2.4]{Hui}, see also Theorem \ref{thm:beh}).  
Later, Johnson-Muraleetharan \cite{JM}, Edelen \cite{Edelen} and Ma \cite{Ma} generalised the distance comparison argument to the CSF in a surface. The definition of our comparison function is based on the one introduced by Ma.

 We remark that, in Proposition \ref{prop:distcomp}, we assume that $L_\psi(T_{\max})>0$ to obtain \eqref{eq:R}, whereas this assumption is not required in either of the previous results  in \cite{Hui, Ma}. In this sense, Proposition \ref{prop:distcomp} provides only a weaker form of extension of those results. However, imposing the assumption $L_\psi(T_{\max})>0$ slightly simplifies the argument (see Appendix \ref{app:distcomp}). Proposition \ref{prop:distcomp} will be used in the proof of Lemma \ref{lem:key1}, where the assumption $L_\psi(T_{\max})>0$ is automatically satisfied.
}
\end{remark}

The proof of Proposition \ref{prop:distcomp} proceeds by a similar argument used in \cite[Section 4]{Ma} in the case of usual CSF in a Riemann surface. However, we require some additional argument in order to extend it to our orbifold setting. We provide a proof in Appendix \ref{app:distcomp} for the sake of completeness.

\subsection{Proof of Proposition \ref{prop:key}}
We begin with some lemmas. The first lemma is an extension of \cite[Lemma 6.1]{Grayson}. However, our proof below is different from the original one, but it is based on the distance comparison estimate, Proposition \ref{prop:distcomp}.

\begin{lemma}\label{lem:key1}
Suppose $T_{\max}<\infty$, $\gamma_{t}: S^{1}\to \mathcal{O}^{{\rm reg}}$ is continuously extended to $\gamma_{T_{\max}}:S^{1}\to \mathcal{O}$ and $L_{\psi}(T_{\max})>0$.
Then the preimage of a point under $\gamma_{T_{\max}}$ is connected.
\end{lemma}
\begin{proof}
For simplicity, we put $T = T_{\max}$ in the following proofs. 
Suppose that $L_{\psi}(T) > 0$. If this is the case, Proposition \ref{prop:distcomp}  shows that, by taking a sufficiently large constant $\mathcal{K}$, we have 
\begin{equation}\label{eq:supR}
\sup_{t \in [0,T)}R(t) < \infty,
\end{equation}
where $R(t)$ is given by \eqref{def:R}.

Let $J := \gamma_{T}^{-1} (p) \neq \emptyset$. We may assume $\# J\geq 2$ and take arbitrary $u, v \in J$ with $u\neq v$. Then,
\begin{equation*}
d_{\psi}(u, v, t) = \widehat{d}(\gamma_t (u), \gamma_t (v) ) \to 0 \quad (t \to T),
\end{equation*}
and hence, \eqref{eq:supR} implies that
\begin{equation*}
\sin \left( \frac{\pi l_{\psi}(u,v,t)}{L_{\psi}(t)} \right) \to 0,
\end{equation*}
and thus,  $l_{\psi}(u,v,t) \to 0$ as $t\to T$.

Let $I(t)\subset S^{1}$ be the unique connected closed subset for which $\partial I=\{u,v\}$ and $l_{\psi}(u,v,t)=\int_{I}\, ds_{\psi}(t)$.
 Note that $I(t)$ is one of only two possible intervals with boundary $\{u,v\}$; therefore there exists a sequence of times $t_i \to T$ for which the interval $I(t_i) =: I$ is constant. 
Then, for any $x\in I$, we see
\begin{align*}
\widehat{d} (\gamma_{t_i} (x), p ) 
&\leq \widehat{d} (\gamma_{t_i} (x), \gamma_{t_i} (v) ) + \widehat{d} (\gamma_{t_i} (v), p ) \\
&\leq l_{\psi}(x,v,t_i)+ \widehat{d} (\gamma_{t_i} (v), p )
\leq l_{\psi}(u,v,{t_i})+ \widehat{d}(\gamma_{t_i} (v), p )  \to 0,
\end{align*}
therefore, $\gamma_{T}(x) = p$ and $I\subset J$. 
Since $u,v\in J$ is arbitrary, this proves the lemma.
\end{proof}

\begin{remark}\label{rem:key1}
Proposition \ref{prop:distcomp} also implies that  if $L_{\psi}(T_{\max}) > 0$, the image $\gamma_{T_{\max}}(S^{1})$ has more than one point. In fact, if $\gamma_{T_{\max}}(S^{1})$ consists of only one point, then $d_{\psi}(u,v, t) \to 0$ as $t \to T_{\max}$ for any $u, v \in S^{1}$.
For every $t < T_{\max}$, we can take points $u_t, v_t \in S^{1}$ such that 
$
2 l_{\psi}(u_t, v_t, t) =L_{\psi}(t).
$
Then, by the same argument of the proof of Lemma \ref{lem:key1}, we conclude
$
l_{\psi}(u_t, v_t, t) \to 0.
$
This, however, contradicts $L_{\psi}(T_{\max}) > 0$. 
\end{remark}

\begin{lemma}\label{lem:key2}
Suppose $T_{\max}<\infty$ and $\gamma_{t}$ is continuously extended to $\gamma_{T_{\max}}$. We fix an arbitrary point $u_{0}\in S^{1}$ so that $\gamma_{T_{\max}}(u_{0})=p_{0}\in \mathcal{O}$, and let $U_{0}=\varphi(\widetilde{U}_0)$ be an open subset obtained by an orbifold chart $(\widetilde{U}_0, G,\varphi)$ around $p_{0}$. If $L_{\psi}(T_{\max})>0$, then there exists an interval $[a,b] \subset S^{1} $ and $\delta > 0$ such that 
$u_0 \in (a, b)$ and
\begin{equation*}
\begin{cases}
\gamma_t (a) \notin \gamma_t ( (a, b] ) \\
\gamma_t (b) \notin \gamma_t ([a, b) )
\end{cases}
\quad {\rm and}\quad 
\gamma_t ( [a,b] ) \subset U_0
\end{equation*}
for any $t \in [T_{\max} - \delta, T_{\max}]$.   
\end{lemma}

\begin{proof}
We put $T=T_{\max}$. First, we note that there is a point $x_0$ such that $\gamma_{T}(x_0) \neq p_0$ by Remark \ref{rem:key1}. 
By Lemma \ref{lem:key1}, we can write $J(p_0) := \gamma_{T}^{-1} (p_0)  = [a_0, b_0] \subsetneq S^{1} \simeq [0,1]/\{0,1\}$ such that $x_0$ corresponds to $\{0,1\}$ and $0 < a_0 \leq b_0 < 1$.
By the continuity of $\gamma \colon S^{1} \times [0, T] \to \mathcal{O}$, 
there exist $\delta > 0$ and an interval $[a_1, b_1] \subset S^{1}$ such that $0 < a_1 < a_0 \leq u_0 \leq b_0 < b_1 < 1$ and 
\begin{equation*}
\gamma_t ( [a_1, b_1] ) \subset U_0\quad \textup{for $t \in [T - \delta, T]$. }
\end{equation*}

Since the curve $\gamma_t$ is an embedding for any $t \in [0, T)$, the required statement automatically holds for any $t\in [0,T)$. Thus, it is sufficient to consider the case when $t=T$. We set 
\begin{align*}
A := \{ x \in [a_1, a_0) \mid \gamma_{T} (x) \notin \gamma_{T} ( (x, b_1] ) \},\quad  B:= \{ y \in (b_0, b_1] \mid \gamma_{T} (y) \notin \gamma_{T} ( [a_1,y) ) \}.
\end{align*}
If both sets $A$ and $B$ are not empty, then we can choose $a \in A$ and $b \in B$ to obtain an interval $[a,b]$ satisfying the required conditions.

We shall prove that $A$ is not empty by contradiction. 
Suppose $A = \emptyset$. Then, for every $x \in [a_1,a_0)$, we have $x \notin J(p_0)$, so
\begin{equation*}
\gamma_T(x) \in \gamma_T( (x,b_1]  \setminus J(p_0)) = \gamma_T( (x,a_0) \sqcup
(b_0,b_1))
\end{equation*}
Take any point $x_1 \in [a_1, a_0)$ and we first prove 
\begin{align}\label{eq:key21}
\gamma_{T}(x_1) \in \gamma_{T}( (b_0, b_1] ).
\end{align}
If not, suppose 
\begin{equation*}
\gamma_{T}(x_1) \in \gamma_{T}( (x_1, a_0) ) 
\text{ and } 
\gamma_{T}(x_1) \notin \gamma_{T}( (b_0, b_1] ).
\end{equation*}
Then, it follows that there exists $x \in (x_1,a_0)$ such that $ \gamma_T (x) = \gamma_t(x_1)$.
Let 
\begin{equation*}
x_2 := \sup \{ x \in (x_1, a_0) \mid \gamma_{T}(x) = \gamma_{T}(x_1) \}.
\end{equation*}
By the continuity of $\gamma_{T}$, we have
$
\gamma_{T}(x_2) =  \gamma_{T}(x_1)  \notin \gamma_{T} ( (b_0, b_1] ).
$
We also have $\gamma_{T}(x_2) =  \gamma_{T}(x_1) \neq p_0$, thus $x_2 \notin J(p_0) = [a_0, b_0]$.
Since $x_2 \neq a_0$ and $x_2 \in (x_1, a_0]$, we have
\begin{equation*}
\gamma_{T} (x_2) \in \gamma_{T} ( (x_2, a_0) ).
\end{equation*}
Thus, for some point $x_3 \in  (x_2, a_0)$, we have
$
\gamma_{T}(x_3) = \gamma_{T}(x_2) = \gamma_{T}(x_1).
$
However, this contradicts the choice of $x_2$.  This proves \eqref{eq:key21} for any $x_1 \in [a_1, a_0)$. 

Next, take a convergent sequence $x_n \in (a_1, a_0)$ such that $x_n \to a_0$. 
Then, there exists a point $y_n \in (b_0, b_1]$ such that $\gamma_{T}(x_n) = \gamma_{T}(y_n)$ for all $n \in \mathbb{N}$. 
If needed, we take a convergent subsequence $\{ y_n' \}$ and we write it again as $\{ y_n \}$.  
This sequence $\{ y_n \}$ converges to the point $b_0$ since 
\begin{equation*}
\gamma_{T}(y_n) = \gamma_{T}(x_n) \to \gamma_{T}(a_0) = p_0 \text{ as } n \to \infty
\end{equation*}
and $\gamma_{T}(y) \neq p_0$ for $y \in (b_0, b_1]$. 
Let $I_n := S^1 \setminus (x_n, y_n)$, we have $\bigcup_{n \in \mathbb{N}} I_n = S^{1} \setminus [a_0, b_0] = S^{1} \setminus J(p_0)$.

Put $p_n := \gamma_T(x_n) = \gamma_T(y_n)$, then $p_n \neq p_0$ since $x_n, y_n \notin J(p_0) = [a_0, b_0]$. 
The preimage $\gamma_T^{-1}(p_n)$ is connected by Lemma \ref{lem:key1} and contains $x_n$ and $y_n$, so we have $\gamma_T(I_n) = \{ p_n \}$ since the interval $[x_n, y_n]$ contains $J(p_0)$.
In particular, $a_1 \in I_n$ for every $n \in \mathbb{N}$, we have $\gamma_{T}(I_n) = \{ \gamma_T(a_1) \}$.

We then obtain 
\begin{equation*}
\gamma_{T}( S^{1}\setminus J(p_0) ) = \bigcup_{n \in \mathbb{N}} \gamma_{T} (I_n) = \{ \gamma_{T}(a_1) \}.
\end{equation*}
However, we have $\gamma_{T} ( J(p_0) ) = \{ p_0 \} \neq \{ \gamma_{T}(a_1) \}$ and this contradicts the continuity of $\gamma_{T}$.
Thus, $A \neq \emptyset$ is obtained, and by the same argument, $B \neq \emptyset$ also follows.
Therefore, we can choose an interval $[a,b]$ as in the statement.
\end{proof}

Now, we give a proof of Proposition \ref{prop:key}. The proof below is inspired by the argument due to Gage \cite[Theorem 3.1]{Gage}.

\begin{proof}[Proof of Proposition \ref{prop:key}]
The proof proceeds by contradiction. Thus, we suppose  the following three assumptions; (a) $\gamma_{t}$ continuously extends to $\gamma_{T}$, (b) $\gamma_{T}(S^{1})\cap \mathcal{O}^{{\rm sing}}\neq \emptyset$ and (c) $L_{\psi}(T)=\lim_{t\to T}L_{\psi}(t)>0$ with $T=T_{\max}$.

Suppose $\gamma_{T} (u_*) = p_0 \in \gamma_{T} (S^{1}) \cap  \mathcal{O}^{{\rm sing}}$.  Recall that since $\mathcal{O}$ is oriented and without boundary, $p_0$ is a cone singularity. 
Since $L_{\psi}(T)>0$,  Lemma \ref{lem:key2} shows that there exists an interval $[a,b] \subset S^{1}$ and $\delta > 0$ such that 
$u_* \in (a, b)$ and
\begin{equation*}
\begin{cases}
\gamma_t (a) \notin \gamma_t ( (a, b] ) \\
\gamma_t (b) \notin \gamma_t ([a, b) )
\end{cases}
\quad \textup{and}\quad \gamma_t ( [a,b] ) \subset U
\end{equation*}
for any $t \in [T - \delta, T]$, where $U=\varphi(\widetilde{U})$ is an open subset obtained by an orbifold chart $(\widetilde{U}, G, \varphi)$ around the cone point $p_0$ of order $k\geq 2$. 
We define the $G$-invariant Riemannian metric $\widetilde{g} := \varphi^*g$ and the $G$-invariant function $\widetilde \psi := \varphi^*(\psi)$ on $\widetilde{U}$, so that $\widehat{g}:=\widetilde g_{\widetilde \psi} =e^{2\widetilde{\psi}}\widetilde{g} = \varphi^*(g_\psi)$.
As mentioned in the previous subsection, we may assume that $\widetilde{U}$ is $\widehat g$-strongly convex and the diameter of $\widetilde{U}$ is sufficiently small so that we can apply Lemmas \ref{lem:J} and \ref{lem:formulas}.

If $t\in [T-\delta, T)$, then $\gamma_t([a,b])\subset \mathcal{O}^{{\rm reg}}\cap U=U\setminus \{p_0\}$. Since $p_0$ is a cone point of order $k$, the restricted map $\varphi: \widetilde{U}\setminus \{0\}\to  U\setminus \{p_0\}$ is a $k$-fold covering map. Therefore, there exists a lift $\widetilde{\gamma}: [a,b]\times [T-\delta,T)\to \widetilde{U}\setminus \{0\}$ of $\gamma$, that is a smooth map such that $\varphi\circ \widetilde{\gamma}=\gamma$. The lift is uniquely determined if we choose the initial point $\widetilde{\gamma}_{T-\delta}(a)\in \varphi^{-1}({\gamma}_{T-\delta}(a))$. Because $\varphi: \widetilde{U}\to U$ is a local isometry, each lift is a $\widetilde{\psi}$-CSF on $\widetilde{U}$. Moreover, by assumption (a), $\widetilde{\gamma}_t$ can be continuously extended to a lift $\widetilde{\gamma}_T$ of $\gamma_T$.
 
 Since $k\geq 2$, we can choose two distinct initial points, and thus, we obtain two lifts $\Gamma^1, \Gamma^2 \colon [a,b] \times [T - \delta, T] \to \widetilde{U}$ of $\gamma: [a,b]\times [T-\delta,T]\to U$.
We consider the distance function $d_{\widetilde{\psi}} \colon [a,b] \times [a,b] \times [T - \delta, T] \to [0, \infty)$ between $\Gamma^{1}$ and $\Gamma^{2}$ with respect to $\widehat{g}$. By the choice of the interval $[a,b]$, the function $d_{\widetilde{\psi}}$ has the following properties:
\begin{equation*}
\begin{cases}
d_{\widetilde{\psi}} (u_*, u_*, T) = 0 \\
d_{\widetilde{\psi}} (u_1, u_2, t) > 0  \quad \text{ for $t \in  [T - \delta, T)$ } \\
d_{\widetilde{\psi}}(u_1, u_2, t) > 0  \quad \text{ for $(u_1, u_2, t) \in \partial([a,b] \times [a,b]) \times [T - \delta, T]$ }.
\end{cases}
\end{equation*}
The third condition is derived as follows. 
Suppose $d_{\widetilde{\psi}} (u_1, u_2, t) = 0$ for some point $(u_1, u_2, t) \in \partial([a,b] \times [a,b]) \times [T - \delta, T]$. 
We may assume $u_1 = a$ and $u_2 \in (a,b]$ since $\Gamma^1_t$ and $\Gamma^2_t$ are the distinct lifts of $\gamma_t$, and $\gamma_T(a)\neq p_0$ implies $\Gamma^1_t(a) \neq \Gamma^2_t(a)$ ($\gamma_T(b) \neq p_0$ and $\Gamma^1_t(b) \neq \Gamma^2_t(b)$ also hold). 
Since $\Gamma^1_t(a) = \Gamma^2_t(u_2)$, we have $\gamma_t(a) = \gamma_t(u_2)$ for $u_2 \in (a, b]$. This contradicts the choice of the interval $[a, b]$ by Lemma \ref{lem:key2}.

Let $B$ be a positive constant to be determined later and 
\begin{equation*}
c := \min \{ e^{B t} d_{\widetilde{\psi}} (u_1, u_2, t) \mid  (u_1, u_2, t) \in \partial([a,b] \times [a,b]) \times [T - \delta, T]  \} > 0,
\end{equation*}

Since $d_{\widetilde{\psi}}(u_*, u_*, T) = 0$, there exists a time $t$ and $u_1, u_2 \in (a,b)$ such that $e^{B t} d_{\widetilde{\psi}} (u_1, u_2, t) = c/2$. Let $\overline t$ be the first such time, and $(\overline x, \overline y)$ points satisfying 
\[\frac{c}{2} = e^{B \bar{t}} d_{\widetilde{\psi}} (\bar{x},\bar{y},\bar{t}) = \min_{ [a,b] \times [a,b]\times [T - \delta, \bar{t}] } e^{B t} d_{\widetilde{\psi}}(u_1, u_2, t),\]  
we have
\begin{align*}
& \frac{\partial ( e^{B t} d_{\widetilde{\psi}} ) }{\partial \widehat{s}_1} (\bar{x},\bar{y},\bar{t}) = \frac{\partial ( e^{B t} d_{\widetilde{\psi}} )}{\partial \widehat{s}_2} (\bar{x},\bar{y},\bar{t}) = 0 =  \frac{\partial d_{\widetilde{\psi}} }{\partial \widehat{s}_1} (\bar{x},\bar{y},\bar{t}) = \frac{\partial d_{\widetilde{\psi}}}{\partial \widehat{s}_2} (\bar{x},\bar{y},\bar{t}), 
\end{align*}
where $\widehat{s}_i$ denotes the arclength parameter of $\Gamma^i(t)$ with respect to $\widehat{g}$. Moreover,  we obtain
\begin{align}
\label{eq:dt}&\frac{\partial ( e^{B t} d_{\widetilde{\psi}})}{\partial t}(\bar{x},\bar{y},\bar{t}) \leq 0, \quad 
\frac{\partial^2 e^{Bt}d_{\widetilde \psi}}{\partial \widehat s_1^2}(\overline x, \overline y, \overline t) \geq 0, \quad \frac{\partial^2 e^{Bt}d_{\widetilde \psi}}{\partial \widehat s_1^2}(\overline x, \overline y, \overline t) \geq 0,
\\
\label{eq:ds2}&\frac{\partial^2 ( e^{B t} d_{\widetilde{\psi}} )}{\partial \widehat{s}^2_1}(\bar{x},\bar{y},\bar{t}) \frac{\partial^2 ( e^{B t} d_{\widetilde{\psi}} )}{\partial \widehat{s}^2_2} (\bar{x},\bar{y},\bar{t}) -  \left( \frac{\partial^2 ( e^{B t} d_{\widetilde{\psi}} )}{\partial \widehat{s}_1 \partial \widehat{s}_2} (\bar{x},\bar{y},\bar{t})  \right)^2 \geq 0.
\end{align}
Let $e_1 := e^{\widetilde{\psi} (\Gamma_{\overline t}^1(\bar{x}))}$, $e_2 := e^{\widetilde{\psi} (\Gamma_{\overline t}^2(\bar{y}))}$, and 
\begin{align}\label{def:Lpm}
\mathcal{L}_{+} 
&:= \frac{\partial}{\partial t} - \left( e_1^2 \frac{\partial^2}{\partial \widehat{s}_1^2} + e_2^2 \frac{\partial^2}{\partial \widehat{s}_2^2} + 2 e_1 e_2 \frac{\partial^2}{\partial \widehat{s}_1 \partial \widehat{s}_2} \right), 
\quad 
\mathcal{L}_{-} 
:= \frac{\partial}{\partial t} - \left( e_1^2 \frac{\partial^2}{\partial \widehat{s}_1^2} + e_2^2 \frac{\partial^2}{\partial \widehat{s}_2^2} - 2 e_1 e_2 \frac{\partial^2}{\partial \widehat{s}_1 \partial \widehat{s}_2} \right).
\end{align}
Then, at $(\bar{x},\bar{y},\bar{t})$, 
we have by \eqref{eq:ds2} that
\begin{equation*}
e_1^2 \frac{\partial^2 ( e^{B t} d_{\widetilde{\psi}} )}{\partial \widehat{s}_1^2} + e_2^2 \frac{\partial^2 ( e^{B t} d_{\widetilde{\psi}} )}{\partial \widehat{s}_2^2} 
\geq 2 e_1 e_2 \sqrt{ \Big{|}\frac{\partial^2 ( e^{B t} d_{\widetilde{\psi}})}{\partial \widehat{s}^2_1} \frac{\partial^2 ( e^{B t} d_{\widetilde{\psi}} )}{\partial \widehat{s}^2_2} }\Big{|}
\geq  2 e_1 e_2 \Big{|}\frac{\partial^2 ( e^{B t} d_{\widetilde{\psi}})}{\partial \widehat{s}_1 \partial \widehat{s}_2} \Big{|},
\end{equation*}
and from this and \eqref{eq:dt} it follows that
\begin{equation}\label{eq:Lpm}
\mathcal{L}_{\pm}( e^{B t} d_{\widetilde{\psi}})\biggm\vert_{(\bar{x},\bar{y},\bar{t})} 
= \frac{\partial ( e^{B t} d_{\widetilde{\psi}} )}{\partial t} - \left( e_1^2 \frac{\partial^2 ( e^{B t} d_{\widetilde{\psi}})}{\partial \widehat{s}_1^2} + e_2^2 \frac{\partial^2 ( e^{B t} d_{\widetilde{\psi}} )}{\partial \widehat{s}_2^2} \pm 2 e_1 e_2 \frac{\partial^2 ( e^{B t} d_{\widetilde{\psi}} )}{\partial \widehat{s}_1 \partial \widehat{s}_2} \right) \leq 0.
\end{equation}

We shall compute $\mathcal{L}_{\pm}( e^{B t} d_{\widetilde{\psi}} )|_{(\bar{x},\bar{y},\bar{t})}$ by using 
the formulas on derivatives of the distance function
given in Lemma \ref{lem:formulas} with $(u_1,u_2,t)=(\bar{x},\bar{y},\bar{t})$. Then we show that the result leads to a contradiction.
We use the same notation given in subsection \ref{ssec:dist}.

Since  
$
\p d_{\widetilde{\psi}}/\p \widehat{s}_1=\p d_{\widetilde{\psi}}/\p \widehat{s}_2=0
$
at $(\bar{x},\bar{y},\bar{t})$, Lemma \ref{lem:formulas} shows that $\widehat{g}(\widehat{T}_1,\tau_1)=\widehat{g}(\widehat{T}_2,\tau_2)=0$. Because $\{\tau_i,\nu_i\}$ ($i=1,2$) are  orthonormal bases of $\mathcal{O}^{{\rm reg}}$ and $\widehat T_i$ are unit vector fields, it follows that 
$\widehat{g} \left( \widehat{T}_1, \nu_1 \right)$ must be equal to either $\widehat{g} \left( \widehat{T}_2, \nu_2 \right)$ or $-\widehat{g} \left( \widehat{T}_2, \nu_2 \right)$ at $(\bar{x},\bar{y},\bar{t})$.
Suppose first $\widehat{g} \left( \widehat{T}_1, \nu_1 \right) = \widehat{g} \left( \widehat{T}_2, \nu_2 \right). $
Then, setting $d := d_{\widetilde \psi}(\overline x, \overline y, \overline t)$, we use Lemma \ref{lem:formulas} and Remark \ref{rem:wronskian} to calculate:
\begin{align*}
\mathcal{L}_{+}d_{\widetilde{\psi}} \biggm\vert_{(\bar{x},\bar{y},\bar{t})} 
&= \frac{\partial d_{\widetilde{\psi}}}{\partial t} - \left( e_1^2 \frac{\partial^2 d_{\widetilde{\psi}}}{\partial \widehat{s}_1^2} + e_2^2 \frac{\partial^2 d_{\widetilde{\psi}}}{\partial \widehat{s}_2^2} + 2 e_1 e_2 \frac{\partial^2 d_{\widetilde{\psi}}}{\partial \widehat{s}_1 \partial \widehat{s}_2} \right) \\
&= e^2_2 \widehat{g} \left(  \widehat{\kappa}_2(\bar{y}, \bar{t}) \widehat{N}_2,  \tau_2 \right) -  e^2_1 \widehat{g} \left( \widehat{\kappa}_1(\bar{x}, \bar{t}) \widehat{N}_1,  \tau_1  \right) \\
&- \biggl\{  - e^2_1 \widehat{g} \left( \widehat{\kappa}_1(\bar{x}, \bar{t}) \widehat{N}_1,  \tau_1  \right) - \frac{e^2_1}{d} \widehat{g} \left( \widehat{T}_1, \nu_1 \right)^2 \frac{J'_1(0)}{J_1(0)} +  e^2_2 \widehat{g} \left( \widehat{\kappa}_2(\bar{y}, \bar{t}) \widehat{N}_2,  \tau_2  \right) + \frac{e^2_2}{d} \widehat{g} \left( \widehat{T}_2, \nu_2 \right)^2 \frac{J'_2(1)}{J_2(1)}  \\
&\quad - \frac{e_1 e_2}{d}  \widehat{g} \left( \widehat{T}_1, \nu_1 \right) \widehat{g} \left( \widehat{T}_2, \nu_2 \right) \frac{J'_2(0)}{J_2(1)}
+ \frac{e_1 e_2}{d} \widehat{g} \left( \widehat{T}_1, \nu_1 \right) \widehat{g} \left( \widehat{T}_2, \nu_2 \right) \frac{J'_1(1)}{J_1(0)} \biggr\} \\
&= \frac{1}{d} \left( e^2_1 \frac{J'_1(0)}{J_1(0)} - e^2_2 \frac{J'_2(1)}{J_2(1)}  - e_1 e_2  \frac{J'_1(1)}{J_1(0)} + e_1 e_2 \frac{J'_2(0)}{J_2(1)} \right) \widehat{g} \left( \widehat{T}_1, \nu_1 \right)^2 \\
&= \frac{1}{d} \left( e_1^2 \frac{J'_1(0) - J'_1(1)}{J_1(0)} - e_2^2 \frac{J'_2(1) - J'_2(0)}{J_2(1)}  + e_1 (e_1 - e_2) \frac{J'_1(1)}{J_1(0)}  + e_2 (e_1 - e_2) \frac{J'_2(0)}{J_2(1)}   \right) \widehat{g} \left( \widehat{T}_1, \nu_1 \right)^2\\
&=\frac{1}{d} \left( e_1^2 \frac{J'_1(0) - J'_1(1)}{J_1(0)} - e_2^2 \frac{J'_2(1) - J'_2(0)}{J_2(1)}  + (e_1 - e_2)^2 \frac{J'_1(1)}{J_1(0)}    \right) \widehat{g} \left( \widehat{T}_1, \nu_1 \right)^2
\end{align*}

We recall that $J_1$ is a function satisfying \eqref{eq:Jacobi}. Since we take the convex subset $\widehat{U}$ sufficiently small,  we may assume that $J'_1(\alpha)<-d/2<0$ by Lemma \ref{lem:J}. 
In particular, we see $J_1(\alpha)\leq J_1(0)$ and 
\begin{align}\label{eq:L1}
\abs{ \frac{J'_1(0) - J'_1(1)}{J_1(0)} } 
&= \abs{ \frac{1}{J_1(0) } \int_0^1 J''_1(\alpha)\, d \alpha } = \abs{ \frac{1}{J_1(0) } \int_0^1 - d^2 \widehat{K} J_1(\alpha)\, d \alpha } \\
&\leq d^2 \sup \abs{\widehat{K}} \int_0^1 \abs{\frac{J_1(\alpha)}{J_1(0)} }\, d \alpha \leq d^2 \widehat{K}_{\mathcal{O}}. \nonumber
\end{align}
By similar arguments, we also have
\begin{align}\label{eq:L2}
\abs{ \frac{J'_2(1) - J'_2(0)}{J_2(1)} } \leq  d^2 \widehat{K}_{\mathcal{O}}.
\end{align}

On the other hand, we see
\begin{align*}
e_1-e_2&=
e^{\widetilde{\psi} (\Gamma_{\overline t}^1(\bar{x}))}-e^{\widetilde{\psi} (\Gamma_{\overline t}^2(\bar{y}))}
=e^{\widetilde{\psi} (\Phi(0))}-e^{\widetilde{\psi} (\Phi(1))}\\
&=-\int_0^1\frac{d}{d\alpha}e^{\widetilde{\psi} (\Phi(\alpha))}\,d\alpha
=-\int_0^1e^{\widetilde{\psi} (\Phi(\alpha))}\cdot d\widetilde{\psi}(\dot{\Phi}(\alpha))\,d\alpha.
\end{align*}
where
$\Phi(\alpha)=\Phi(\alpha;\bar{x},\bar{y},\bar{t}):[0,1]\to \widetilde{U}$ be the minimizing geodesic from 
$\Gamma_{\overline t}^1(\bar{x})$ to $\Gamma_{\overline t}^2(\bar{y})$
(see subsection \ref{ssec:dist}). Note that
\[
d\widetilde{\psi}(\dot{\Phi})=\widehat{g}(\widehat{\nabla}\widetilde{\psi}, \dot{\Phi})\leq |\widehat{\nabla}\widetilde{\psi}|_{\widetilde{\psi}}|\dot{\Phi}|_{\widetilde{\psi}}= |\widehat{\nabla}\widetilde{\psi}|_{\widetilde{\psi}}\cdot d.
\]
Thus, by letting $D:={\rm sup}_{\widetilde{U}}e^{\widetilde{\psi}}$ and $E:={\rm sup}_{\widetilde{U}}|\widehat{\nabla}{\widetilde{\psi}}|_{\widetilde{\psi}}$, we see
\[
(e_1-e_2)^2=\Big(\int_0^1e^{\widetilde{\psi} (\Phi(\alpha))}\cdot d\widetilde{\psi}(\dot{\Phi}(\alpha))\,d\alpha\Big)^2\leq D^2E^2d^2.
\]

Since $J_1'(1) = -d$ and $J_1(0)>d/2>0$ by Lemma \ref{lem:J}, we obtain
\begin{align}\label{eq:L3}
(e_1-e_2)^2\frac{J_1'(1)}{J_1(0)} \geq D^2E^2d^2\cdot (-d) \cdot \frac{2}{d} = -2 D^2E^2d^2 .
\end{align}

Thus, using the above estimates \eqref{eq:L1}--\eqref{eq:L3} with $\widehat{g}( \widehat{T}_1, \nu_1 )^2\leq 1$, we see
\begin{equation}\label{eq:Lp}
\mathcal{L}_{+} d_{\widetilde{\psi}} \biggm\vert_{(\bar{x},\bar{y},\bar{t})} \geq  \{-(e^2_1 + e^2_2) \widehat{K}_{\mathcal{O}}- 2D^2E^2 \}\cdot d\geq \{-2D^2\widehat{K}_{\mathcal{O}}- 2 D^2E^2 \}\cdot d.
\end{equation}
Therefore, if we take the positive constant $B$ so that
\begin{equation*}
B >2D^2\widehat{K}_{\mathcal{O}}+ 2 D^2E^2,
\end{equation*}
 then, we see
\begin{equation*}
\mathcal{L}_{+} \left( e^{B t} d_{\widetilde{\psi}} \right)  \biggm\vert_{(\bar{x},\bar{y},\bar{t})} 
=( B e^{B \bar{t}} d_{\widetilde{\psi}} + e^{B \bar{t}} \mathcal{L}_{+} d_{\widetilde{\psi}} )\biggm\vert_{(\bar{x},\bar{y},\bar{t})} \geq \left\{  B -2D^2\widehat{K}_{\mathcal{O}}-2 D^2E^2 \right\}  e^{B \bar{t}} d > 0
\end{equation*}
since $e^{B \bar{t}} d = c/2 > 0$. 
This contradicts \eqref{eq:Lpm}.
The same calculation can be applied in the case that $\widehat{g} \left( \widehat{T}_1, \nu_1 \right) = - \widehat{g} \left( \widehat{T}_2, \nu_2 \right)$ for the operator $\mathcal{L}_{-}$. This completes the proof of Proposition \ref{prop:key}.
\end{proof}

\section{Generalised Lagrangian Mean Curvature Flow with Symmetries}\label{sec:glmcf}

In this section, we move to our main topic: generalised Lagrangian mean curvature flow in almost-Einstein K\"ahler manifolds. We cover fundamental definitions and results relating to these topics, and describe how the \(G\)-equivariant GLMCF reduces to a weighted curve shortening flow on the $2$-dimensional orbifold quotient as studied in the previous sections.

\subsection{$f$-MCF and Generalised Lagrangian MCF} 
Let $(M^{n+k},g)$ be an $(n+k)$-dimensional Riemannian manifold and  $\varphi:\Sigma^n\to M^{n+k}$ be a smooth immersion of a closed $n$-dimensional manifold $\Sigma$. We consider a density function $f\in C^{\infty}(M)$ and 
define the \textit{canonical conformal change} of $g$ (with respect to $f$) to be $g_{f}:=e^{2f}g$.

The {\it second fundamental form} $A$ of $\varphi$ is  the tensor field on $\Sigma$ defined by $A(X,Y):=(\onab_{d\varphi(X)}d\varphi(Y))^{\perp}$, where $\overline{\nabla}$ is the Levi-Civita connection on $M$ and $\perp$ means the orthogonal projection onto the normal space of $\varphi$. 
The {\it weighted mean curvature vector} of $\varphi$ is defined by
\[
K:=H-n(\overline{\nabla}f)^{\perp},
\]
where $H:={\rm tr}A$ is the usual mean curvature vector of $\varphi$.  An immersion $\varphi$ will be called an {\it $f$-minimal} (or $f$-stationary) if $K=0$ along $\varphi$.  It is known that $\varphi$ is $f$-minimal if and only  if  $\varphi$ is a critical point of the {\it weighted volume functional} 
\[
{\rm Vol}_f(\varphi):=\int_{\Sigma} dv_f,\quad {\rm where}\quad dv_f:=e^{nf}dv_\Sigma.
\]
Indeed, for any smooth deformations $\{\varphi_s\}_{s\in (-\epsilon,\epsilon)}$ of $\varphi_0=\varphi$, it is known that 
\begin{align}\label{eq:fvv}
\frac{d}{ds}{\rm Vol}_f(\varphi_s)\Big{|}_{s=0}=-\int_\Sigma g(K,V)\,dv_f,
\end{align}
where $V=d\varphi_s/ds|_{s=0}$.
Note that ${\rm Vol}_{f}$ coincides with the Riemannian volume with respect to the weighted Riemannian metric $g_{f}:=e^{2f}g$, and the mean curvature vector $H_{f}$ of $\varphi$ w.r.t.\  $g_{f}$ is given by $H_{f}=e^{-2f}K$.  We remark that $K$ is different from the mean curvature vector $H_{f}$.

A family of smooth immersions $F_t:\Sigma \times [0,T) \to M$ is called an \textit{$f$-mean curvature flow starting from $\varphi$} if it satisfies
\begin{align*}
\frac{\partial F_t}{\partial t}=K_t,\quad F_0=\varphi,
\end{align*}
where $K_t$ is the weighted mean curvature vector of $F_t$.

\begin{remark}
 If $n=k=1$ and there is a globally defined unit normal vector field $N$ along $\varphi$, we see that 
$
K=(\kappa-\langle \overline{\nabla}f, N\rangle)N=\kappa_{f}N, 
$
where $\kappa=\langle H,N\rangle$ is the geodesic curvature of $\varphi: \Sigma^{1}\to M^{2}$. Thus, in this case, the $f$-MCF coincides with the $f$-CSF discussed in section \ref{sec:csf}
and \ref{sec:csf2}.
\end{remark}

We now fix a connected compact Lie subgroup $G\subset {\rm Isom}(M,g)$ and a $G$-action on $\Sigma$
such that the map $\varphi:\Sigma \to M$ is $G$-equivariant.
Furthermore, we assume $f$ is a $G$-invariant function (This is obviously satisfied if $f\equiv 0$ as in the case of usual MCF).
In this situation, we have the following basic fact.
 
\begin{lemma}\label{lem:equi}
Suppose the initial immersion  $\varphi:\Sigma \to M$ is $G$-equivariant and $f: M\to \R$ is a $G$-invariant function for a  Lie subgroup $G\subset {\rm Isom}(M,g)$. Then, the $f$-MCF $F_{t}$ starting from $F_{0}=\varphi$ is $G$-equivariant for any $t\in [0,T)$. 
\end{lemma}

\begin{proof}
Fix arbitrary $g\in G$.  Let $\widetilde{F}_t$ be the $f$-MCF starting from $\widetilde{\varphi}=g\circ\varphi$. 
Since $g$ is an isometry on $M$ and $f$ is a $G$-invariant function, it turns out that $g\circ F_t$ is  a solution of $f$-MCF  with initial immersion $\widetilde{\varphi}$.  Therefore,  by the uniqueness of solution of $f$-MCF, it holds that $\widetilde{F}_t=g\circ F_t$ for any $t\in [0,T)$.  

On the other hand, because the initial immersion $\varphi$ is $G$-equivariant (i.e.\ $g\circ \varphi=\varphi\circ g$), we have $\widetilde{F}_0=F_0\circ g$ and hence, $\widetilde{F}_t=F_t\circ g$.  As a consequence, we obtain $g\circ F_t=F_t\circ g$ for any $t\in [0,T)$ and $g\in G$. This proves the lemma.
\end{proof}

We now assume $(M^{2n},\omega, J,g)$ is a connected K\"ahler manifold of real dimension $2n$, where $J$ is the complex structure on $M$ and $\omega$ is the symplectic form compatible with $J$. Note that we define the Riemannian metric $g$ by the relation $\omega(\cdot,\cdot)=g(J\cdot,\cdot)$.
The \textit{Ricci form} on $M$ is the $2$-form $\rho \in \Omega^2(M)$ given by
\begin{equation}\label{def-ricciform}
    \rho(X,Y) \, := \, \mathrm{Ric}(JX,Y).
\end{equation}
The manifold $M$ is called  {\it almost-Einstein} if the Ricci form $\rho$ satisfies 
\begin{align}\label{eq:Ricci}
\rho=C\omega+ndd^{c}f
\end{align}
 for some constant $C$ and $f\in C^{\infty}(M)$, where $d^{c}f$ is a $1$-form defined by $d^{c}f(X):=-df(JX)$. 
 We note the following uniqueness result for the function $f$ in the case where $M$ is compact:
\begin{lemma}\label{lem-funiqueness}
    If $(M,J,\omega,g)$ is a compact K\"ahler manifold such that \eqref{eq:Ricci} holds for two different functions $f,\tilde f\in C^{\infty}(M)$, then $f - \tilde f$ is constant.
\end{lemma}
\begin{proof}
    The equation \eqref{eq:Ricci} for $f,\tilde f$ implies that $dd^c(f - \tilde f) = 0$, i.e.\ $f-\widetilde f$ is a pluriharmonic function. In particular, $f-\tilde f$ is harmonic, i.e.\ $\Delta(f - \tilde f) = 0$. By the strong maximum principle, it follows that $f = \tilde f + C$ for a constant $C$.
\end{proof}
In particular, for a compact almost-Einstein K\"ahler manifold $(M,J,\omega,g)$, the derivatives $\nabla f$, $df$, $d^c f$ do not depend on the choice of $f$.

 If $f$ is a constant function, $M$ is called {\it K\"ahler-Einstein}. Note that if $M$ is compact and has positive first Chern class,  then any K\"ahler metric $\omega$ in the K\"ahler class $[\rho]$ satisfies the condition \eqref{eq:Ricci}. 

An immersion $\varphi: L^{n}\to M^{2n}$ of an $n$-dimensional manifold $L$ is called {\it Lagrangian} if $\varphi^{*}\omega\equiv 0$.  In this case, we have the following well-known result:
\begin{theorem}[Smoczyk \cite{Smo1}, Behrndt \cite{Behrndt}]
 Let $(M,\omega, J)$ be an almost-Einstein K\"ahler manifold with $\rho=C\omega+ndd^{c}f$.  Then, the $f$-MCF $F_{t}$ in $M$ starting from a Lagrangian immersion $\varphi$ preserves the Lagrangian condition.
\end{theorem}

Because of this property, the $f$-MCF $F_{t}$  starting from a Lagrangian immersion in an almost-Einstein K\"ahler manifold  is called the {\it generalised Lagrangian mean curvature flow} ({\it GLMCF} for short). 
We note that by Lemma \ref{lem-funiqueness}, if $M$ is compact then GLMCF is independent of the choice of function $f$, and is therefore a canonical flow associated to the K\"ahler manifold.

For a vector field $V$ along $\varphi$, we define an associated $1$-form $\alpha_{V}$ by
\[
\alpha_{V}(X_{p}):=\omega_{\varphi(p)}(V,\varphi_{*}X)=g_{\varphi(p)}(JV,\varphi_{*}X)\quad \textup{for $X\in T_{p}L$}.
\]
Then, the {\it generalised mean curvature form} of a Lagrangian immersion $\varphi: L\to M$  is defined by $\alpha_{K}$, where $K=H-n(\onab f)^{\perp}$ is the weighted mean curvature vector. Equivalently, we put
\[
\alpha_{K}=\alpha_{H}-n\varphi^{*}d^{c}f,
\]
where $H$ is the usual mean curvature vector of $\varphi$.

Note that $\alpha_{K}$ is a closed $1$-form. Indeed, in a K\"ahler manifold, it holds that $d\alpha_{H}=\varphi^{*}\rho$ (known as Dazord's formula), and combining with the almost-Einstein condition, we have 
\[
d\alpha_{K}=\varphi^{*}(\rho-ndd^{c}f)=\varphi^{*}(C\omega)=0.
\]
  In particular, the generalised mean curvature form defines a cohomology class $[\alpha_{K}]\in H^1(L,\R)$.
The following property of GLMCF is also fundamental. 

\begin{proposition}[\protect{\cite[Prop.\ 4.3]{KK}}]\label{prop:exact}
Suppose the generalised mean curvature form $\alpha_K$ of the initial Lagrangian immersion $\varphi: L\to M$ is an exact $1$-form. Then, the GLMCF $F_{t}$ with $F_{0}=\varphi$ preserves the exactness of generalised mean curvature form, i.e.\ $[\alpha_{K_{t}}]=0$ for any $t\in [0, T)$.
\end{proposition}

Geometrically, this proposition means that if the initial Lagrangian immersion $\varphi$ has an exact generalised mean curvature form, then the GLMCF generates a  Hamiltonian deformation of $\varphi$.

\subsection{Hamiltonian Group Actions}\label{subsec:redu}

Let $(M,\omega)$ be a symplectic manifold, and suppose there exists a connected Lie group $G$ with Lie algebra $\fg$ such that $G$ acts on $M$ preserving the symplectic form, i.e.\ $g^{*}\omega=\omega$ for each $g\in G$. The action of $G$ on $M$ is called {\it Hamiltonian} if there exists a smooth map $\mu: M\to \fg^{*}$ satisfying the following two conditions: 
\begin{itemize}
    \item[(i)] $d\mu^{X}=\omega(X^{\sharp},\cdot)$ for any $X\in \fg$, where $\mu^{X}(p):=\langle\mu(p), X\rangle$ and $X^{\sharp}_p:= \frac{d}{dt}\Big|_{t=0}\exp(tX)\cdot p$ is the fundamental vector field on $M$ generated by $X$.
    \item[(ii)] $\mu(g \cdot p) = \mathrm{Ad}^*_g (\mu(p))$ for any $g \in G$, $p \in M$, where ${\rm Ad}^{*}: G\to GL(\fg^{*})$ is the coadjoint representation.
\end{itemize}

The map $\mu: M\to \fg^{*}$ is called the {\it moment map} associated with the Hamiltonian $G$-action. Note that in general, the moment map associated with the $G$-action is not unique.
However, if $M$ is connected then any two moment maps must differ by a constant $c \in Z(\mathfrak{g}^*) :=\{c\in \fg^{*}\mid {\rm Ad}^{*}(g)c=c\ \forall g\in G\}$.

We now assume that the Lie group $G$ acts on $M$ in a Hamiltonian way with moment map $\mu: M\to \fg^{*}$. The following properties are fundamental (See \cite[Ch.III]{Audin} for the proof).

\begin{lemma}\label{lem:mo}
The moment map $\mu: M\to \fg^{*}$ satisfies the following.
\begin{enumerate}
\item The $G$-action preserves the non-empty level set $\mu^{-1}(c)$ if and only if $c\in Z(\fg^{*})$. 
\item $p\in M$ is a regular point of $\mu$ if and only if the stabiliser subgroup $G_{p}$ is discrete. 
\end{enumerate}
In particular, if $c\in Z(\fg^{*})$ and $c$ is a regular value of $\mu$, then $G$ acts on $\mu^{-1}(c)$ and the action is locally free.
\end{lemma}

In the following, we assume that $G$ is compact. Suppose furthermore that $c\in Z(\fg^{*})$ and $c$ is a regular value of $\mu$, so that $G$ acts locally freely on the smooth manifold $\mu^{-1}(c)$. 
We denote the set of principal $G$-orbits in $\mu^{-1}(c)$ by $\mu^{-1}(c)^{{\rm pri}}$. Note that $\mu^{-1}(c)^{{\rm pri}}$ is an open submanifold of $\mu^{-1}(c)$ by the principal orbit theorem.
If $G$ acts on  $\mu^{-1}(c)^{{\rm pri}}$ freely,  we have the following by the well-known Marsden-Weinstein symplectic reduction.

\begin{proposition}[cf.\ \cite{Futaki2, K}]\label{prop:dime}
Suppose  $c\in Z(\fg^{*})$ is a regular value and the compact Lie group  $G$ acts on $\mu^{-1}(c)^{{\rm pri}}$ freely.   Then we have
\begin{enumerate}
\item $\mu^{-1}(c)$ is a submanifold of ${\rm dim}\mu^{-1}(c)=2n-k$, where $k:={\rm dim}G$.
\item $\overline{M}_{c}^{{\rm pri}}:=\mu^{-1}(c)^{{\rm pri}}/G$ is a  $2(n-k)$-dimensional smooth manifold. Moreover, $\overline{M}_{c}^{{\rm pri}}$ admits a symplectic structure $\overline{\omega}$ such that $\pi^{*}\overline{\omega}=\iota^{*}\omega$, where $\pi: \mu^{-1}(c)^{{\rm pri}}\to \overline{M}_{c}^{{\rm pri}}$ is the natural projection and $\iota: \mu^{-1}(c)^{{\rm pri}}\to M$ is an inclusion. 
\end{enumerate}
If furthermore $(M,J,\omega)$ is a K\"ahler manifold and $G \subset \mathrm{Aut}(M,J,\omega)$, then 
\begin{enumerate}
\item[(iii)] there exists a complex structure $\overline{J}$ on $\overline{M}_c^{{\rm pri}}$ so that $(\overline{M}_c^{{\rm pri}}, \overline{J}, \overline{\omega})$ is a K\"ahler manifold. 
\item[(iv)] The natural projection $\pi: \mu^{-1}(c)^{{\rm pri}}\to \overline{M}_{c}^{{\rm pri}}$ becomes a Riemannian submersion with respect to the induced metric $g$ on $\mu^{-1}(c)^{{\rm pri}}$ and the K\"ahler metric $\overline{g}$ on $\overline{M}_{c}^{{\rm pri}}$.
\end{enumerate}
In this case, $(\overline{M}_{c}^{{\rm pri}}, \overline{J},\overline{\omega})$ will be called a \emph{K\"ahler quotient} of $M$.  

\end{proposition}

We now consider an almost-Einstein K\"ahler manifold with $\rho=C\omega+ndd^{c}f$ and a connected closed subgroup $G$ of ${\rm Aut}(M,J,\omega)$. We note that if $M$ is compact, then $f$ must be $G$-invariant:

\begin{lemma}\label{lem-fisGinvariant}
    If $(M^{2n},J,\omega,g)$ is a closed almost-Einstein K\"ahler manifold, and $G \subset \mathrm{Aut}(M,J,\omega)$ is compact and connected, then any $f$ satisfying $\eqref{eq:Ricci}$ is $G$-invariant.
\end{lemma}
\begin{proof}
    Since $G \subset \mathrm{Aut}(M,J,\omega)$, $G$ also preserves the metric $g$ and the Ricci form $\rho$. Equation \eqref{eq:Ricci} therefore implies that $dd^cf$ is a $G$-invariant $2$-form.

    Now for every $g \in G$, we have
    \begin{align*}
        g^*(dd^cf) = dd^c f \, \implies \, dd^c(g^* f - f) \, = \, 0,
    \end{align*}
    so $g^*f-  f$ is a pluriharmonic function for each $g \in G$; in particular it is harmonic. Since $M$ is compact, by the maximum principle we have that $g^*f - f = c_g \in \mathbb{R}$ for each $g \in G$. Furthermore, the map $g \mapsto c_g$ is a homomorphism from $G$ to $\mathbb{R}$, since
    \begin{align*}
    (g\circ h)^*f \, &= \, h^* (f + c_g) \, = \, f + c_h + c_g.
    \end{align*}
    Since $G$ is compact, this implies $c_g = 0$ for all $g$, as required.
\end{proof}

If $C\neq 0$, it is known that there exists a moment map of the $G$-action $\mu_{{\rm can}}$ as follows, which we will refer to as the \textit{canonical moment map}.
\begin{proposition}{\cite[Prop. 3]{K}}\label{prop-mucan}
    If $(M^{2n},J,\omega,g)$ is a closed almost-Einstein K\"ahler manifold satisfying \eqref{eq:Ricci} with $C\neq 0$, and $G \subset \mathrm{Aut}(M,J,\omega)$ is compact and connected, then the $G$-action is Hamiltonian with a moment map given by
    \begin{equation}\label{eq-mucan}
        \mu_{{\rm can}}^X(p) \, = \, \langle \mu_{{\rm can}}(p), X\rangle=\frac{1}{C}\Big{\{}-\frac{1}{2}{\rm div}JX^{\sharp}+nd^{c}f(X^{\sharp})\Big{\}}(p).
    \end{equation}
    % Furthermore, $\mu_{{\rm can}}$ satisfies the following properties for any $X \in \mathfrak{g}$:
    % \begin{align*}%\label{eq-mucanproperties}
    % \Delta_{f}\mu_{{\rm can}}^{X}=2C\mu_{{\rm can}}^{X},\quad \int_{M}\mu_{{\rm can}}^{X}\,dv_{f}=0,
    % \end{align*}
    % where $\Delta_{f}$ is the weighted Laplacian defined by $\Delta_{f}u=\Delta u-2n\cdot g(du,df)$ for $u \in C^{\infty}(M)$ and $dv_{f}:=e^{2nf}dv$ is the weighted volume measure on $M$.
\end{proposition}

In general, even if $(M,J,\omega)$ is K\"ahler-Einstein, the K\"ahler quotient $(\overline{M}^{\rm pri}_c,\overline{J},\overline{\omega})$ is not necessarily Einstein as well. However, in our setting, the almost-Einstein property is inherited by the K\"ahler quotient at the zero level set, as follows:

\begin{proposition}[\cite{Futaki, K}]\label{prop:Ricci}
Let $(M^{2n},J,\omega,g)$
be a closed almost-Einstein K\"ahler manifold such that
$\rho=C\omega+ndd^{c}f$ with $C\neq 0$. Suppose that 
$0\in \fg^{*}$ is a regular value (with non-empty level set) and
that $G \subset \mathrm{Aut}(M,J,\omega)$ is a connected, compact subgroup acting on $\mu_{{\rm can}}^{-1}(0)^{{\rm pri}}$ freely.
 
Then, the K\"ahler quotient $\overline{M}^{{\rm pri}}=(\mu_{{\rm can}}^{-1}(0)^{{\rm pri}}/G, \overline J,\overline\omega)$
of the $0$-level set is almost-Einstein. The Ricci form $\overline{\rho}$ of $\overline{M}^{\rm pri}$ is given by
$
\overline{\rho}=C\overline{\omega}+(n-k)dd^{c}\overline{f},
$
where  $k={\rm dim}(G)$ and 
$\overline f$ is the function induced by the $G$-invariant function
\begin{align}\label{def:fbar}
\widetilde{f}:\mu_{{\rm can}}^{-1}(0)\to \R, \quad \widetilde{f}=\log V^{\frac{1}{n-k}}+\frac{n}{n-k} f
%\overline{f}(x):= \frac{1}{n-k}\Big(\log{\rm Vol}(\pi^{-1}(x))+n \check{f}(x)\Big),
\end{align}
where $V: \mu^{-1}_{{\rm can}}(0)\to \R$ is the orbit volume function defined by \eqref{def:vol} (see Appendix \ref{app-orbifold}), i.e.\ $\overline f \circ \pi = \widetilde f$.

\end{proposition}

\begin{remark}
Note that in general, if $M$ is connected and compact, and $G$ is a compact connected Lie group, then the level set $\mu^{-1}(c)$ of the moment map of Hamiltonian $G$-action is connected and compact for any $c\in \mu(M)$. It follows then that the K\"ahler quotient as in Proposition \ref{prop:Ricci} is also connected. The connectedness of $\mu^{-1}(c)$ is non-trivial, and is known as the Kirwan connectedness theorem (see Atiyah-Guillemin-Sternberg's work for the toric case \cite[Theorem 5.5.1]{MS}, and Kirwan's work for the general case \cite{kirwan_convexity_1984}).
\end{remark}

Recall that even if $G$ freely acts on the set of principal orbits $\mu_{{\rm can}}^{-1}(0)^{{\rm pri}}$,  the entire quotient space $\overline{M}=\mu_{{\rm can}}^{-1}(0)/G$ is not necessarily a smooth manifold, but rather an orbifold, and the orbifold structure $\mathcal{O}$ is as described in subsection \ref{subs:lco}. Since $\widetilde{f}: \mu_{{\rm can}}^{-1}(0)\to \R$
is a $G$-invariant smooth function, we obtain a well-defined orbifold smooth function $\overline{f}:\mathcal{O}\to \R$ 
such that $\overline f \circ \pi = \widetilde f$ and $\mathcal{O}$ is equipped with a weighted orbifold Riemannian metric $\overline{g}_{\overline{f}}$.  In fact,
the weighted Riemannian metric $\overline{g}_{\overline{f}}$ on $\overline{M}=\mu_{{\rm can}}^{-1}(0)/G$  coincides with the Hsiang-Lawson metric defined in subsection \ref{subs:lco} with respect to the weighted metric $g_{f}=e^{2f}g$ (See \cite[Section 2]{K} for details).

\subsection{$G$-equivariant Lagrangian immersions.}

Next, we recall some facts regarding Lagrangian immersions which are equivariant with respect to a Hamiltonian action.
We refer the reader to \cite{Audin, BCO, Futaki, Futaki2, K} for further details.

For a symplectic manifold $(M,\omega)$, an immersion $\varphi: S^{l}\to M^{2n}$ of an $l$-dimensional connected manifold $S^{l}$ is called {\it isotropic} if $\varphi$ satisfies $\varphi^{*}\omega=0$; in this case, it must hold that $l\leq n$, and $\varphi$ is Lagrangian if $l=n$. We have the following important result for equivariant isotropic immersions:

\begin{lemma}\label{lem:isotro}
Suppose that we have a Hamiltonian action of a connected Lie group $G$ on $(M,\omega)$, with moment map $\mu:M\to\mathfrak{g}^*$.
Suppose that $G$ also acts on $S^{l}$ and let $\varphi:S^l \to (M^{2n},\omega)$ be a $G$-equivariant isotropic immersion.

Then, there exists an element $c\in Z(\fg^{*})$ such that $\varphi(S)\subseteq \mu^{-1}(c)$. 
\end{lemma} 

\begin{proof}
For $p \in S$, $X \in \mathfrak{g}$ and $Y \in T_pS$,
\begin{equation}
d(\mu^X \circ \varphi)_p(Y) \, = \, d\mu^X_{\varphi(p)}\circ d\varphi_p(Y) \, = \, \omega(X^\sharp_{\varphi(p)}, d\varphi_{p}(Y)).
\end{equation}
Since $\varphi$ is $G$-equivariant, $X^\sharp_{\varphi(p)} \in d\varphi_p(T_p S)$, and since $\varphi(S)$ is isotropic, the right hand side vanishes. Therefore, $\mu \circ \varphi$ is constant. Equivariance of $\mu$ implies that this constant lies in $Z(\mathfrak{g}^*)$.
\end{proof}

We now assume that $(M,J,\omega,g)$ is a closed almost-Einstein K\"ahler manifold satisfying \eqref{eq:Ricci}, and consider a connected closed subgroup $G \subset \mathrm{Aut}(M,J,\omega)$ with canonical moment map $\mu_{\rm can}$.
Let $\varphi: L^{n}\to M^{2n}$ be a Lagrangian immersion. By Lemma \ref{lem:isotro}, if $\varphi$ is $G$-equivariant,  then $\varphi(L)\subset \mu_{{\rm can}}^{-1}(c)$ for some $c\in Z(\fg^{*})$ .  If furthermore $G$ acts on $L$ freely, then $\overline{L}:=L/G$ is a smooth manifold and we obtain an immersion $\overline{\varphi}: \overline{L}^{n-k}\to \overline{M}_{c}^{{\rm pri}}=\mu_{{\rm can}}^{-1}(c)^{{\rm pri}}/G$ such that $\overline{\varphi}\circ \pi_{L}=\pi\circ \varphi$, where $\pi: \mu_{{\rm can}}^{-1}(c)\to \overline{M}_{c}$ (resp. $\pi_{L}: L\to \overline{L}$) is the natural projection.
 It is easy to see that $\overline{\varphi}$ is also a Lagrangian immersion into $\overline{M}_{c}^{{\rm pri}}$ (see also the  following figure). 
\[
\xymatrix@C=25pt@R=25pt{
L\  \ar[r] \ar@/^16pt/[rr]^{\varphi}  \ar[d]_-{\pi_{L}} &\ \mu_{{\rm can}}^{-1}(c)^{{\rm pri}}\  \ar[r]_-{} \ar[d]_-{\pi}&\  M\\
\overline{L}\  \ar[r]^-{\overline{\varphi}} &\ \overline{M}_{c}^{{\rm pri}} &  
}
\]
 
Moreover, the following  fact is proved in \cite[Theorem 1]{K} by the second author. Note that, since 
the statement concerns a local property, it can be generalised to (Lagrangian) immersions.

\begin{theorem}[\cite{K}]\label{thm:K}
Assume we are in the setting of Proposition \ref{prop:Ricci}, and let $\varphi:L \to M$ be a $G$-equivariant Lagrangian immersion such that $G$ acts on $L$ freely.
Then, we have the following.
\begin{enumerate}
\item If the generalised mean curvature form $\alpha_{K}$ of $\varphi$ is an exact $1$-form, then $\varphi(L)$ is contained in the $0$-level set of the canonical moment map $\mu_{{\rm can}}:M\to \fg^{*}$.
\item If $\varphi(L)$ is contained in $\mu_{{\rm can}}^{-1}(0)^{{\rm pri}}$, then the generalised mean curvature form $\alpha_{K}$ of $\varphi: L\to M$ and $\overline \alpha_{\overline K}$
of $\overline{\varphi}: \overline{L}\to \overline{M}^{{\rm pri}}$ satisfy the relation 
$\pi^{*}\overline\alpha_{\overline{K}}=\alpha_{K}$. 

In particular, $K$ is tangent to $\mu_{\rm can}^{-1}(0)$ and $K$ coincides with the horizontal lift of $\overline{K}$ with respect to the Riemannian submersion $\pi: \mu_{\rm can}^{-1}(0)^{\rm pri}\to \overline{M}^{\rm pri}$.
\end{enumerate}
\end{theorem}
\begin{proof}
See \cite{K} for the proof. We only check that $K$ is the horizontal lift of $\overline{K}$. Since $\varphi$ is $G$-equivariant, for each $p\in L$, the tangent space $T_pL$ is decomposed into an orthogonal direct sum $T_pL=\mathcal{H}_p\oplus \fg_p$, where $\mathcal{H}_p$ is the horizontal subspace w.r.t.\  $\pi_L$ and $\fg_p$ is the tangent space of the $G$-orbit $G\cdot p$ at $p$. Using this notation, the tangent space of $\mu_{\rm can}^{-1}(0)^{\rm pri}$ is given by (see \cite{K})
\[
T_p\mu_{\rm can}^{-1}(0)^{\rm pri}=\mathcal{H}_p\oplus J\mathcal{H}_p\oplus \fg_p.
\]
Note that $\mathcal{H}_p\oplus J\mathcal{H}_p$ coincides with the horizontal subspace with respect to $\pi: \mu_{\rm can}^{-1}(0)^{\rm pri}\to \overline{M}^{\rm pri}$.\\
  Since  $K$ is a normal vector of the Lagrangian immersion $\varphi:L\to M$, we have $JK\in T_pL$.  By the relation $\pi^{*}\overline\alpha_{\overline{K}}=\alpha_{K}$, we see $g(JK,X)=\alpha_K(X)=\pi^{*}\overline\alpha_{\overline{K}}(X)=0$ for any $X\in \fg_p$ because $\pi_*X=0$, and hence, $JK\in \mathcal{H}_p$. This shows that $K\in J\mathcal{H}_p\subset T_p\mu_{\rm can}^{-1}(0)^{\rm pri}$ and $K$ is a horizontal vector with respect to $\pi$. Moreover, we have that 
\begin{align*}
    \overline\omega(\overline K, \pi_* X) \, = \, \overline\alpha_{\overline K}(\pi_* X) \, = \, (\pi^*\overline\alpha_{\overline K})(X) \, = \, \alpha_{K}(X) \, = \, \omega(K,X) \, = \, (\pi^*\overline\omega)(K,X) \, = \, \overline\omega(\pi_* K,\pi_* X).
\end{align*}
for any $X\in T_{p}L$. This shows that $\overline{K}=\pi_*K$, and hence, $K$ is the horizontal lift of $\overline{K}$.
\end{proof}

\begin{corollary}\label{cor:K}
If the $G$-equivariant Lagrangian immersion $\varphi: L\to M$ is $f$-minimal, then $\varphi(L)\subset \mu_{{\rm can}}^{-1}(0)$. Moreover, if $\varphi(L)\subset \mu_{{\rm can}}^{-1}(0)^{{\rm pri}}$, then $\varphi$ is $f$-minimal in $M$ if and only if $\overline{\varphi}$ is $\overline{f}$-minimal in $\overline{M}^{{\rm pri}}$.
\end{corollary}

\begin{proof}
If $\varphi$ is $f$-minimal, then $\alpha_K=0$, so Theorem \ref{thm:K}(i) gives $\varphi(L)\subset\mu^{-1}_{{\rm can}}(0)$. The second assertion then follows from Theorem \ref{thm:K}(ii).
\end{proof}

\begin{remark}
{\rm As a special situation, when the ($f$-)minimal Lagrangian submanifold $L$ is {\it homogeneous}, namely, if $L^{n}$ is a connected, compact $G$-orbit in $M$, then it must hold that  $L=\mu_{{\rm can}}^{-1}(0)$. If furthermore $G$ is compact and semi-simple, then $Z(\fg^{*})=\{0\}$ and in this case, the $G$-action admits at most one ($f$-)minimal Lagrangian $G$-orbit. These facts on homogeneous Lagrangian submanifolds were proved in \cite{BG, Pacini} under the assumption that $M$ is K\"ahler-Einstein. Theorem \ref{thm:K} is a  generalisation of the extremal case.
}
\end{remark}

We now consider the GLMCF $F_{t}$ starting from a $G$-equivariant Lagrangian immersion $F_{0}=\varphi: L\to M$. By Lemma \ref{lem:equi}, $F_{t}$ is $G$-equivariant for all time $t\in [0,T)$. Moreover, Theorem \ref{thm:K} implies the following. 

\begin{proposition}\label{prop:redu}
Assume we are in the setting of Proposition \ref{prop:Ricci}, and let $\varphi:L \to M$ be a $G$-equivariant Lagrangian embedding such that the generalised mean curvature form $\alpha_K$ is exact and $G$ acts on $L$ freely.

Then the $G$-equivariant GLMCF $F_t:L\times [0,T) \to M$ starting from $F_0 = \varphi$ satisfies the following:
\begin{enumerate}
\item  $F_t(L)\subset \mu_{{\rm can}}^{-1}(0)$ for any $t\in [0,T)$.
\item If furthermore $F_{t}(L)\subset \mu_{{\rm can}}^{-1}(0)^{{\rm pri}}$ for any $t\in [0,T)$, then $\overline{F}_{t}: 
\overline{L} \times [0,T) \to \overline{M}^{{\rm pri}}$ is a solution of GLMCF in the K\"ahler quotient space $\overline{M}^{{\rm pri}}=\mu_{{\rm can}}^{-1}(0)^{{\rm pri}}/G$ with $\overline{F}_{0}=\overline{\varphi}$.
\end{enumerate}
\end{proposition}
\begin{proof}
By Proposition \ref{prop:exact}, if $[\alpha_{K_{0}}]=0$, then we have $[\alpha_{K_{t}}]=0$ for any $t\in [0,T)$. Thus, Theorem \ref{thm:K}-(i) shows that $F_{t}(L)\subset \mu_{{\rm can}}^{-1}(0)$ for any $t$. This proves (i). 

%Since $F_{t}(L)$ is contained in $\mu_{{\rm can}}^{-1}(0)$, $K_{t}=\partial F_{t}/\partial t$ is tangent to $\mu_{{\rm can}}^{-1}(0)$ for any $t\in [0,T)$. 
% Note that in general,  $K$ is tangent to $\mu_{{\rm can}}^{-1}(0)$  if $\alpha_{K}$ is an exact $1$-form (see \cite{K}).
%Thus, by Theorem \ref{thm:K}-(ii), we have 
%\begin{align*}
%    \overline\omega(\overline K_t, \pi_* X) \, = \, \overline\alpha_{\overline K_t}(\pi_* X) \, = \, (\pi^*\overline\alpha_{\overline K_t})(X) \, = \, \alpha_{K_t}(X) \, = \, \omega(K_t,X) \, = \, (\pi^*\overline\omega)(K_t,X) \, = \, \overline\omega(\pi_* K_t,\pi_* X).
%\end{align*}
%for any $X\in T_{p}L$, where $\pi: \mu_{{\rm can}}^{-1}(0)^{{\rm pri}}\to \overline{M}^{\rm pri}$ is the natural projection. 

If $F_t(L)\subset \mu_{\rm can}^{-1}(0)^{\rm pri}$, then by Theorem \ref{thm:K}-(ii), we have $\overline{K}_{t}=\pi_{*}K_{t}$ for any $t$, and hence,
\[
\frac{\p \overline{F}_{t}}{\p t}=\frac{\p (\pi\circ F_{t})}{\p t}=\pi_{*}\Big(\frac{\p F_{t}}{\p t}\Big)=\pi_{*}K_{t}=\overline{K}_{t}.
\]
This proves (ii).
\end{proof}

 We also have a converse to the above proposition:

\begin{proposition}\label{prop:lift}
    Assume we are in the setting of Proposition \ref{prop:Ricci}, and let $\varphi: L \to M$ be a $G$-equivariant Lagrangian embedding such that $\alpha_K$ is exact and $G$ acts on $L$ freely. Assume furthermore that $\overline F_t:\overline L \to \overline M^{\rm pri}$ is a solution of GLMCF in the K\"ahler quotient space starting from $\overline\varphi = \pi \circ \varphi$.

    Then there exists a solution $F_t: L \times [0,T) \to \mu_{\rm can}^{-1}(0)^{\rm pri}$ to GLMCF starting from $\varphi$ with $\pi \circ F_t = \overline F_t$.
\end{proposition}
\begin{proof}
    Consider the generalised mean curvature vector $\overline K_t$ of $\overline F_t$. Extend $\overline K_t$ arbitrarily to a smooth time-dependent vector field $\overline{\mathcal{K}}_t$ on $\overline M^{\rm pri}$ which is compactly supported for all $t \in [0,T)$, and consider the horizontal lift $\mathcal{K}_t$ of $\overline{\mathcal{K}}_t$ with respect to the Riemannian submersion $\pi:\mu_{\rm can}^{-1}(0)^{\rm pri}\to \overline{M}^{\rm pri}$. Note that the restriction of $\mathcal{K}_t$ to $\pi^{-1}(\overline F_t(\overline L)) \subset \mu_{\rm can}^{-1}(0)^{\rm pri}$ is the horizontal lift of $\overline K_t$, which is the generalised mean curvature $K_t$ of $\pi^{-1}(\overline F_t(\overline L))$ by Theorem \ref{thm:K}-(ii).

    Now consider the flow $\Phi_t$ of $\mathcal{K}_t$ which by construction is well-defined for all $t \in [0,T)$, and define the time-dependent map $F_t: L \times [0,T) \to \mu^{-1}_{\rm can}(0)^{\rm pri}$ by $F_t := \Phi_t \circ \varphi$. Note then that:
    \begin{align*}
        \frac{\partial}{\partial t} (\pi \circ F_t(x)) = \pi_* \left( \mathcal{K}_t(F_t(x))\right) = \overline{\mathcal{K}}_t(\pi\circ F_t(x)),
    \end{align*}
    which implies by uniqueness of the flow of a vector field that $\pi \circ F_t = \overline F_t$. In particular, note that $F_t(L) = \pi^{-1}(\overline F_t(\overline L))$, and therefore $\frac{\partial}{\partial t}F_t = K_t$ as required.
\end{proof}

\subsection{Curvature bounds for $G$-equivariant Lagrangian immersions} The following lemma relates the curvature of a $G$-equivariant immersed Lagrangian $L$ and its quotient $\overline L$. 

\begin{lemma}\label{lem:curvatureestimate}
Let $(P,g)$ be a Riemannian manifold. Suppose that a compact Lie group $G$ acts on $P$ freely so that $Q=P/G$ is a smooth manifold. We endow $Q$ a Riemannian metric $\overline{g}$ such that $\pi: P\to Q$ is a Riemannian submersion. Let $\widetilde{\varphi}: \Sigma \to P$ be a $G$-equivariant immersion of an $n$-dimensional manifold $\Sigma$ on which $G$ acts freely, and $\overline{\varphi}: \overline{\Sigma}\to Q$ is the immersion of $\overline{\Sigma}:=\Sigma/G$ into $Q$ such that $\pi\circ \widetilde{\varphi}=\overline{\varphi}\circ \pi_\Sigma$, where $\pi_\Sigma: \Sigma\to \overline{\Sigma}$ is the natural projection. 

    Then there exists a non-negative continuous function $\Psi: P\to \R$ depending only on $\pi: P\to Q$ such that for any $G$-equivariant immersion $\widetilde{\varphi}: \Sigma\to P$ on which $G$ acts freely, the norms of the second fundamental form $\widetilde{B}$ of $\widetilde{\varphi}$ and the second fundamental form $\overline B$ of the quotient map $\overline{\varphi}: \overline{\Sigma}\to Q$ are related by
    \begin{equation}\label{eq:bdB}
        \|\widetilde{B}\|^2(w) \leq \|\overline B\|^2 (\pi(w)) + \Psi(w)\quad \forall w\in \Sigma.
    \end{equation}
\end{lemma}
\begin{proof}
%Let $L$ be as in the statement. Denoting by $B_L^U$ the second fundamental form of $L$ in $U$, and $B_U^M$ the second fundamental form of $U$ in $M$, we have
%$
%\|B\|^{2}\leq \|B_L^U\|^{2}+\|B_U^P\|^{2}.
%$
%Since $U$ is compact, the latter is bounded by a constant $C_1(M,g,\omega, c, U)$. Thus, it suffices  to consider $|B_L^U|^{2}$.

%For convenience, denote $P := \mu^{-1}_{\rm can}(c)^{\rm pri}$ and $Q := \overline M^{\rm pri}_c$. 
Since $\pi: P \to Q$ is a Riemannian submersion, according to O'Neill \cite{ONeill}, we may decompose the tangent space of $T_{w}P$ into a direct sum of vertical and horizontal subspaces, $T_{w}P=\mathcal{H}_{w}\oplus \mathcal{V}_{w}$. Note that $\pi_{*}: \mathcal{H}_{w}\to T_{\pi(w)}Q$ is an isomorphism. For a vector $X\in T_{\pi(w)}Q$, we denote the corresponding horizontal vector by $X^{\mathcal{H}}\in \mathcal{H}_{w}$, namely, $X^{\mathcal{H}}$ is defined by $\pi_{*}X^{\mathcal{H}}=X$.
If $w\in \Sigma$, any element in $\mathcal{V}_{w}$ is tangent to $\Sigma$, since $\Sigma$ is $G$-invariant and $\mathcal{V}_{w}$ consists of fundamental vector fields of the $G$-action. Thus, for each $w\in \Sigma$,  we have an orthogonal decomposition $T_{w}P=\mathcal{N}_{w}^\Sigma\oplus\mathcal{H}_{w}^\Sigma\oplus \mathcal{V}_{w}$ so that $T_{w}\Sigma=\mathcal{H}_{w}^\Sigma\oplus \mathcal{V}_{w}$ and $\mathcal{N}_{w}^\Sigma$ is the normal subspace of $\Sigma$ in $P$. 

Now, assume that the dimension of the group orbits is $m \leq n$. Let $\{U_1, \ldots, U_{n-m}, V_{1},\ldots, V_{m}\}$ be a local orthonormal frame of $T\Sigma$, where $U_i\in \mathcal{H}^\Sigma$ and $V_j\in \mathcal{V}$.  
Using O'Neill's formula \cite[Lemma 3]{ONeill}:
\begin{align*}
\widetilde{B}(U_i,U_j)&=(\nabla^P_{U_i}U_j)^{\perp}=\{(\nabla^Q_{(\pi_{*}U_i)}(\pi_{*}U_j))^{\mathcal{H}}\}^{\perp}=(\overline{B}(\pi_{*}U_i,\pi_{*}U_j))^{\mathcal{H}},\\
\widetilde{B}(U_i,V_{j})&=(\nabla^P_{U_i}V_{j})^{\perp}=(A_{U_i}V_{j})^{\perp},\quad 
\widetilde{B}(V_{i},V_{j})=(\nabla^P_{V_{i}}V_{j})^{\perp}=(T_{V_{i}}V_{j})^{\perp},
\end{align*}
where ${\nabla^P}$ (resp. $\nabla^Q$) is the Levi-Civita connection of $P$ (resp. $Q$), $\perp$ denotes the orthogonal projection onto $\mathcal{N}^{\Sigma}$, and $A$, $T$ are some tensor fields on $P$ defined purely by the Riemannian submersion $\pi: P\to Q$ (see \cite{ONeill} for details). Therefore, we have
\begin{align*}
 \|\widetilde{B}\|^{2}&= \sum_{i,j = 1}^{n-m} |(\overline{B}(\pi_{*}U_i,\pi_{*}U_j))^{\mathcal{H}}|^{2}+2\sum_{i=1}^{n-m}\sum_{j=1}^{m}|(A_{U_i}V_{j})^{\perp}|^{2}+\sum_{i,j=1}^{m}|(T_{V_{i}}V_{j})^{\perp}|^{2}\\
&\leq \|\overline{B}\|^{2}+C(\|A\|^{2}+\|T\|^{2}).
\end{align*}
We thus put $\Psi(w):=C(\|A(w)\|^{2}+\|T(w)\|^{2})$, then we obtain \eqref{eq:bdB}.
\end{proof}

It follows that if $F_t$ is a $G$-equivariant GLMCF and we have curvature bounds on the quotient flow $\overline F_t \subset \overline M^{\rm pri}$ of Proposition \ref{prop:redu}, we can deduce higher curvature bounds on the original flow:

\begin{proposition}\label{prop:bddB}
    Assume we are in the setting of Proposition \ref{prop:redu}, and let $T \in (0,\infty]$ be the maximal existence time for the GLMCF $F_t:L \times [0,T) \to M$. Let $\overline F_t: \overline L \times [0,T)\to \overline M^{\rm pri}$ be the quotient GLMCF as in the conclusion of that proposition. Assume further that:
    \begin{itemize}
        \item There exists a compact set $U \subset \mu_{\rm can}^{-1}(0)^{\rm pri}$ and $0 < T_0 < T$ for which $F_t(L) \subset U$ for $t \in [T_0,T)$,
        \item There exists a constant $\overline C_0$ for which $\|\overline B_t\| < \overline C_0$ for $t \in [0,T)$,
    \end{itemize}
    where $\overline B$ denotes the second fundamental form of $\overline F_t: \overline L \to \overline M^{\rm pri}$.
    
    Then for all $k \in \mathbb{Z}_{\geq 0}$ and $\varepsilon > 0$ there exist $C_{k, \varepsilon} = C_{k, \varepsilon}(M,g,\omega, U, \overline C_0)> 0$ such that 
    \begin{equation}\label{eq:bdB2}
        \|\nabla^kB_t\| \leq C_{k, \varepsilon} \text{ for all }t \in [T_0 + \varepsilon,T),
    \end{equation}
    where $B_t$ denotes the second fundamental form of $F_t$.
\end{proposition}

\begin{proof}
We put $P=\mu_{\rm can}^{-1}(0)^{\rm pri}$ and $Q=\overline{M}^{\rm pri}$. We first note that 
we have
$
\|B_t\|^{2}\leq \|\widetilde{B}_t\|^{2}+\|B^P\|^{2},
$
where $\widetilde{B}_t$ (resp. $B^P$) is the second fundamental form of $\widetilde{F}_t: L\to P$ (resp. $P\to M$).
By assumption, $\|B^P\|^{2}$ is bounded by a constant $C_1(M,g, U)$ on $[T_0,T)$. On the other hand, Lemma \ref{lem:curvatureestimate} and assumptions shows that $\|\widetilde{B}_t\|^{2}\leq \overline{C}_0+C_2(P,g,Q, U)$ on $[T_0,T)$.  Therefore, we obtain \eqref{eq:bdB2} for $k=0$.

 The case $k \geq 1$ now follows from standard higher curvature estimates for mean curvature flow (see for example \cite[Lemma 5.5]{KK}, \cite[Proposition 1.1]{Smo2}).
\end{proof}

\subsection{Topology of quotient orbifold}

Concerning  the topology of the K\"ahler quotient, we have the following result, which assumes positivity of the constant $C$.

\begin{proposition}\label{prop:topo}
Let $(M^{2n},\omega, J)$ be a compact almost-Einstein K\"ahler manifold with $\rho=C\omega+ndd^{c}f$ and $C>0$.  Suppose that $G$ is a connected compact subgroup of ${\rm Aut}(M,\omega,J)$ with canonical moment map $\mu_{{\rm can}}:M\to \fg^{*}$. 
Furthermore, we assume that
\begin{itemize}
\item[(i)] $0$ is a regular value of $\mu_{{\rm can}}$ and the restricted action on $\mu_{{\rm can}}^{-1}(0)^{{\rm pri}}$ is free.
\item[(ii)] The action of $G$ on $\mu_{{\rm can}}^{-1}(0)$ is cohomogeneity-two, that is, the principal $G$-orbit  has codimension two in $\mu_{{\rm can}}^{-1}(0)$. 
% \item[(iii)] Every $G$-orbit contained in $\mu_{{\rm can}}^{-1}(0)$ is orientable.
\end{itemize}
Then, the quotient space $\overline{M}=\mu_{{\rm can}}^{-1}(0)/G$ is homeomorphic to $2$-sphere, and has the structure of an orbifold with at most three cone points. 
\end{proposition}

\begin{proof}
By assumptions (i) and (ii) and Proposition \ref{prop-riemorbifoldquotient},
$\overline{M}$  gains the structure of a $2$-dimensional compact orbifold $\mathcal{O}$ 
with an orbifold atlas given by $\mathcal{A}=\{B_{p_i}^\perp(\epsilon_{i}), \widehat G_{p_{i}}, \pi\circ{\rm exp}_{p_{i}}\}_{i\in I}$ of $\mathcal{O}$, where  $B_{p}^\perp(\epsilon):=\{v\in T_{p}^{\perp}(G\cdot p)\mid \|v\|<\epsilon\}$ is the $\epsilon$-ball in the normal space $T^{\perp}_{p}(G\cdot p)$ of the $G$-orbit $G\cdot p$.
Since $G$ is a subgroup of ${\rm Aut}(M,\omega, J)$, any element in $G$ 
yields an orientation preserving diffeomorphism on $M$. 
If we fix an orientation of $\mathfrak g$, then since the action of $G$ on $\mu^{-1}_{{\rm can}}(0)$ is locally free and the stabiliser groups are all orientation preserving, the derivative of the group action gives an orientation on $T_p(G \cdot p)$ for every $p\in \mu_{{\rm can}}^{-1}(0)$ which is preserved by the \(G\)-action. Noting that $(T_p\mu_{{\rm can}}^{-1}(0))^\perp = J T_p(G \cdot p)$ (see e.g.\ \cite[Lem.\ 3.3]{MW}), we obtain a $G$-invariant orientation of the normal space $(T_p\mu_{{\rm can}}^{-1}(0))^\perp$ and therefore of $T_p \mu_{{\rm can}}^{-1}(0)$ and $T_p(G \cdot p)^\perp$ (where the latter is the perpendicular complement inside $T_p \mu^{-1}_{{\rm can}}(0)$). This gives a global orientation to our orbifold atlas $\mathcal{A}$, and so we have an oriented orbifold $\mathcal{O}$.
    
Next, we shall show that  the Euler characteristic $\chi^{{\rm orb}}(\mathcal{O})$ of the orbifold $\mathcal{O}$ is positive. 
% (see \cite{Satake, Thurston} for the definition of the Euler characteristic for orbifold)
We consider the smooth manifold $\overline{M}^{\rm pri}=\mu_{\rm can}^{-1}(0)^{\rm pri}/G$ which is a connected open subset in $\overline{M}$. Note that ${\rm dim}(\overline{M}^{{\rm pri}})=2$. By Proposition \ref{prop:Ricci}, the Ricci form $\overline{\rho}$ on $\overline{M}^{{\rm pri}}$  satisfies $\overline{\rho}=C\overline{\omega}+dd^{c}\overline{f}$, where $\overline{f}$ is the potential function on $\overline{M}^{{\rm pri}}$ induced from \eqref{def:fbar}. 
By using an isothermal coordinate $(x,y)$ such that $\frac{\p}{\p x}=\overline{J}\frac{\p}{\p y}$, it is easy to see that 
 the equation $\overline{\rho}=C\overline{\omega}+dd^{c}\overline{f}$ is equivalent to
\begin{align}\label{eq:Ric}
{\rm Ric}_{\overline{g}}=C\overline{g}+(\Delta \overline{f})\overline{g},
\end{align}
 where $\overline{g}:=\overline{\omega}(\cdot ,\overline{J}\cdot )$ and ${\rm Ric}_{\overline{g}}$ is the Ricci tensor of $\overline{g}$. 
Consider a conformal change of metric defined by $\overline{g}_{\overline{f}}:=e^{2\overline{f}}\overline{g}$. Then, by a general formula for the conformal change of metrics, we see that  the Ricci tensor ${\rm Ric}_{{\overline{f}}}$ of $\overline{g}_{\overline{f}}$ is given  by 
$
{\rm Ric}_{{\overline{f}}}={\rm Ric}_{\overline{g}}-(\Delta \overline{f})\overline{g}
$
since ${\rm dim}\overline{M}^{{\rm pri}}=2$, where ${\rm Ric}_{\overline{g}}$ is the Ricci curvature w.r.t.\  $\overline{g}$. Therefore, combining this with \eqref{eq:Ric}, we obtain
$
{\rm Ric}_{{\overline{f}}}=C\overline{g}=Ce^{-2\overline{f}}\overline{g}_{\overline{f}}.
$
This means the Gaussian curvature $K_{{\overline{f}}}$ with respect to the metric $\overline{g}_{\overline{f}}$ on the surface $\overline{M}^{{\rm pri}}$ is given by $Ce^{-2\overline{f}}$, which is positive since we assume $C>0$. 

Recall that,  as mentioned in Section \ref{subs:lco},  the metric $\overline{g}_{\overline{f}}$ on $\overline{M}^{{\rm pri}}$ extends to an orbifold Riemannian metric on $\mathcal{O}$.
Therefore,  the Gauss-Bonnet theorem for orbifold (Theorem \ref{thm:GB}) shows that
\[
\chi^{{\rm orb}}(\mathcal{O})=\frac{1}{2\pi}\int_{\mathcal{O}}K_{\overline{f}} \omega_{\overline{f}}>0
\]
as claimed, where $\omega_{\overline{f}}$ is the area-form w.r.t.\  $\overline{g}_{\overline{f}}$.

Since the singular sets of $\mathcal{O}$ consists of cone points, the Euler characteristic is given by equation \eqref{eq-eulercharacteristicorbifold}.
In particular, we  have $\chi(\overline{M})\geq \chi^{{\rm orb}}(\mathcal{O})>0$, and hence, $\overline{M}$
is homeomorphic to $2$-sphere. Moreover, %by using \eqref{eq-eulercharacteristicorbifold} and the fact that $\chi(\overline{M})=\chi(S^{2})=2$, one can easily show that
 if $\chi^{{\rm orb}}(\mathcal{O})>0$, then the classification of $2$-dimensional oriented compact orbifold shows that there exist at most three cone points (see Table \ref{tb:orbi} given in Section \ref{sec:2dorbifolds}).
\end{proof}

\subsection{Examples}\label{subsec:eg}
In this subsection, we provide some concrete examples of cohomogeneity-one Lagrangian submanifold (with exact generalised mean curvature form) by using a {\it torus action}.

Let $(M,\omega)$ be a $2n$-dimensional symplectic manifold, and assume that an $n$-dimensional torus $T^{n}$ acts on $M$  in a Hamiltonian way with moment map $\mu: M\to \R^{n}$.  Note that we have $\R^n\simeq \fg^{*}=Z(\fg^{*})$  since $G=T^{n}$ is abelian. Moreover,  every $T^{n}$-orbit is isotropic and $\mu$ is constant on each orbit (see \cite[Proposition III.2.12]{Audin}).  We may assume that the action of $T^{n}$ on $M$ is effective. In this case, it follows that $T^{n}$ freely acts on the set of principal orbits $M^{{\rm pri}}$. Indeed, for an element $p\in M^{{\rm pri}}$, 
it is easy to see that $G_{p}\subseteq G_{q}$ for any $q\in M$ since $G=T^{n}$ is abelian. Thus, if there were a 
non-trivial element $g\in G_{p}$, then $g$ fixes every point in $M$, namely the $T^{n}$-action is not effective. Therefore, if $T^{n}$ acts on $M$ effectively, then $G_{p}$ is trivial for any $p\in M^{{\rm pri}}$. 

In the following, we assume $T^{n}$ acts effectively on $M$. Since $T^n$ acts on $M^{\rm pri}$ freely, every principal $T^n$-orbits is Lagrangian. 
Moreover, we suppose that $(M,\omega)$ is equipped with a compatible $T^n$-invariant almost K\"ahler structure $(J,g)$.

\begin{proposition}\label{prop:mintorus}
With the above assumptions, we have the following. 
\begin{enumerate}
%\item Every principal $T^n$-orbit is Lagrangian. 
\item There exists no exceptional $T^n$-orbit in $M$. 
\item If $M$ is compact and $f\in C^{\infty}(M)$ is a $T^{n}$-invariant function, then there exists an $f$-minimal Lagrangian $T^n$-orbit contained in $M^{\rm pri}$.
\end{enumerate}
\end{proposition}
\begin{proof}
(i) We put $G=T^n$. Suppose that there exists an exceptional $G$-orbit $L_{\rm ex}=G\cdot p$. Then $L_{\rm ex}$ is Lagrangian and the isotropy subgroup $G_p\subset G$ must be finite, but $G_p\neq \{e\}$. Take arbitrary $g\in G_p$. For any $X\in T_p L_{\rm ex}\simeq \fg$, we see $dg_p(X)=\frac{d}{dt}g\exp(tX)p{|_{t=0}}=\frac{d}{dt}\exp(tX)gp{|_{t=0}}=\frac{d}{dt}\exp(tX)p{|_{t=0}}=X$, i.e.\ $dg_p|_{T_pL_{\rm ex}}=id$, where $\exp$ is the exponential map of the Lie group $G=T^n$ and we used the fact that $T^n$ is abelian. Moreover, we have $dg_p(JX)=Jdg_p(X)=JX$ for any $X\in T_pL_{\rm ex}$ since $J$ is $T^n$-invariant. Therefore, we obtain $dg_p=id$ on $T_pM$ because $L_{\rm ex}$ is Lagrangian which implies that we have an orthogonal decomposition $T_pM=T_pL_{\rm ex}\oplus JT_pL_{\rm ex}$. 

On the other hand, by the slice theorem, it follows that a $G$-invariant open neighbourhood $\mathcal{N}$ of $L_{\rm ex}$ is equivariantly diffeomorphic to $G\times_{G_p}T_p^\perp L_{\rm ex}$, where $T_p^\perp L_{\rm ex}$ is the normal space at $p$. We take arbitrary $[k,v]\in G\times_{G_p}T_p^\perp L_{\rm ex}$, and consider the isotropy subgroup $G_{[k,v]}\subset G$. If $h\in G_{[k,v]}$, then $[k,v]=h\cdot [k,v]=[hk,v]$, and the definition of equivalent relation implies that there exists $g\in G_p$ such that $(hk,v)=g\cdot (k,v)$. Since the action of $G_p$ on $G\times T_p^\perp L_{\rm ex}$ is defined by $g\cdot (k,v):=(kg^{-1}, dg_p(v))$, using the fact $dg_p=id$, we see $(hk,v)=(kg^{-1}, v)$ and hence, $h\in G_{[k,v]}$ if and only if $hk=kg^{-1}$ for some $g\in G_p$. However, since $G=T^n$ is abelian, this shows that $h\in G_{[k,v]}$ if and only if $h=g^{-1}$ for some $g\in G_p$. Namely, $G_{[k,v]}=G_p$ for any $[k,v]\in G\times_{G_p}T_p^\perp L_{\rm ex}$. This implies that for any $q\in \mathcal{N}$, we have $G_q=G_p\neq \{e\}$ however, this is a contradiction because the set of principal orbits $M^{\rm pri}$ is a dense subset of $M$ and hence, there exists $q\in \mathcal{N}\cap M^{\rm pri}$ with $G_q=\{e\}$. This proves (i). 

(ii) It is a general fact that for any connected compact Lie group $G$ acting by isometries on a compact Riemannian manifold $M$, there exists a regular minimal $G$-orbit (see \cite[Corollary 3]{Pacini}). Here a regular orbit means either a principal orbit or an exceptional orbit.  As shown (i), there is no exceptional $T^n$-orbit and hence, applying the general result for the $T^n$-action on the Riemannian manifold $(M,g_f)$, we see that there exists an $f$-minimal regular $T^n$-orbit contained in $M^{\rm pri}$.
\end{proof}

%Since $T^{n}$ acts on $M^{{\rm pri}}$ freely, the restricted action on $M^{{\rm reg}}=M^{{\rm pri}}\cup M^{{\rm ex}}$ is locally free, and each regular orbit is a Lagrangian torus.  

We shall consider a codimension-one subtorus $T^{n-1}$ of $T^{n}$ so that $T^{n-1}$ acts on each Lagrangian $T^{n}$-orbit as a cohomogeneity-one action. We denote the moment map of the $T^{n-1}$-action on $M$ by $\mu_{T^{n-1}}: M\to \R^{n-1}$. Since $T^{n-1}$ is also abelian, $T^{n-1}$ acts on each level set $\mu^{-1}_{T^{n-1}}(c)$ for any regular value $c$, and this action is  locally free by Lemma \ref{lem:mo}. Furthermore, we have the following.

%Notice that the action of $T^{n-1}$ on $\mu^{-1}_{T^{n-1}}(c)$ does not admit any singular $T^{n-1}$-orbit since $c$ is a regular value (and hence, the stabiliser subgroup at $p\in \mu^{-1}_{T^{n-1}}(c)$ is discrete as well as $T^{n}$-action). However, in general, the $T^{n-1}$-action may admit an exceptional orbit in $\mu^{-1}_{T^{n-1}}(c)$. More precisely, we see the following.

\begin{proposition}\label{prop:torus}
Suppose $c$ is a regular value of $\mu_{T^{n-1}}$ and $\mu^{-1}_{T^{n-1}}(c)$ is connected and compact. Then, the action of $T^{n-1}$ on $\mu^{-1}_{T^{n-1}}(c)$ admits at most two exceptional $T^{n-1}$-orbit. Moreover, if $p\in \mu^{-1}_{T^{n-1}}(c)$ is an exceptional point for the $T^{{n-1}}$-action, then $T^{n-1}\cdot p$ coincides with a singular $T^{n}$-orbit  in $M$.  
%In particular, any principal $T^{n}$-orbit contained in $\mu^{-1}_{T^{n-1}}(c)$ does not admit an exceptional point of the $T^{n-1}$-action.
\end{proposition}

%To prove this, we use the following facts on the cohomogeneity-one action . 
%
%\begin{theorem}[cf.\ \cite{Mos}]\label{thm:Mos}
%Suppose that a connected compact Lie group $G$ acts on a compact manifold $M$ as a cohomogeniety-one action, that is, the codimension of the principal $G$-orbit is equal to $1$.  Then, we have
%\begin{enumerate}
%\item The orbit space $M/G$ is homeomorphic to $S^{1}$ or the closed interval $[0,1]$. If $M/G\simeq S^{1}$, every orbit is principal. 
%\item Suppose $M/G\simeq [0,1]$ and there is an orientable principal $G$-orbit. Then, $M$ is non-orientable if and only if there exists an orientable exceptional $G$-orbit. 
%\end{enumerate}
%\end{theorem}
%We shall give a proof of Proposition \ref{prop:torus}.

\begin{proof}
First, we remark that  if $p\in \mu^{-1}_{T^{n-1}}(c)$, then $T^{n}\cdot p\subset \mu^{-1}_{T^{n-1}}(c)$ since $T^{n}\cdot p$ is $T^{n-1}$-invariant isotropic submanifold in $M$ (see Lemma \ref{lem:isotro}). This implies that  the $n$-dimensional torus $T^{n}$ also acts on $\mu^{-1}_{T^{n-1}}(c)$. Moreover, since ${\rm dim}\mu^{-1}_{T^{n-1}}(c)=n+1$, the action of $T^{n}$ on $\mu^{-1}_{T^{n-1}}(c)$ is  cohomogeneity-one.  Thus, by the result of \cite{Mos}, we see that the orbit space $\mu^{-1}_{T^{n-1}}(c)/T^{n}$ is homeomorphic to either $S^{1}$ or $[0,1]$ due to the compactness of $\mu^{-1}_{T^{n-1}}(c)$.  

If  $\mu^{-1}_{T^{n-1}}(c)/T^{n}\simeq S^{1}$, then  every $T^{n}$-orbit is principal in $\mu^{-1}_{T^{n-1}}(c)$ (see also Remark \ref{rem:trivialprincipalcase}), and hence, $T^{n}$ freely acts on $\mu^{-1}_{T^{n-1}}(c)$. This shows that the action of subtorus $T^{n-1}$ on $\mu^{-1}_{T^{n-1}}(c)$ is also free and thus, there is no exceptional $T^{n-1}$-orbit.  

We shall consider the case when  $\mu^{-1}_{T^{n-1}}(c)/T^{n}\simeq [0,1]$. Then there are exactly two non-principal $T^{n}$-orbits in  $\mu^{-1}_{T^{n-1}}(c)$ which correspond to boundaries of $[0,1]$. 
Since there is no exceptional $T^n$-orbit by Proposition \ref{prop:mintorus}-(i), the non-principal $T^{n}$-orbit must be a singular orbit.
%We first claim that the non-principal $T^{n}$-orbit must be a singular orbit, i.e.\ it is not an exceptional orbit. Indeed, if there were an exceptional point $p\in\mu^{-1}_{T^{n-1}}(c)$ with respect to the $T^{n}$-action, then the exceptional orbit $T^{n}\cdot p$ is orientable since $T^{n}\cdot p$ is homeomorphic to $T^{n}$, and  by \cite[Theorem 3]{Mos}, $\mu^{-1}_{T^{n-1}}(c)$ must be non-orientable (Note that the statement of \cite[Theorem 3]{Mos} does not hold if the orbit space is homeomorphic to $S^{1}$, see also the errata of \cite{Mos}. However, this is not our case).  This contradicts the general fact that any regular level set of a smooth map between oriented manifolds is also orientable. Therefore, there is no exceptional $T^{n}$-orbit in $\mu^{-1}_{T^{n-1}}(c)$.
If $p\in\mu^{-1}_{T^{n-1}}(c)$ is a principal point of  $T^{n}$-action (i.e.\ if $T^{n}\cdot p$ corresponds to an interior point of $[0,1]$), the stabiliser subgroup in $T^{n}$ at $p$ is trivial by assumption, and hence, so is the stabiliser subgroup in $T^{n-1}$ at $p$. Namely, $T^{n-1}$-orbit through $p$ is principal in $\mu^{-1}_{T^{n-1}}(c)$. 

Therefore, if $p$ is an exceptional point for $T^{n-1}$-action in $\mu^{-1}_{T^{n-1}}(c)$, then $p$ must be a singular point  of $T^{n}$-action.  If this is the case, we have ${\rm dim}(T^{n}\cdot p)\leq n-1$. On the other hand, we have ${\rm dim}(T^{n-1}\cdot p)=n-1$ and $T^{n-1}\cdot p\subseteq T^n\cdot p$, and hence, the exceptional orbit $T^{n-1}\cdot p$ coincides with the singular orbit $T^{n}\cdot p$.

Since there are exactly two singular $T^{n}$-orbits in $\mu^{-1}_{T^{n-1}}(c)$, this shows that there are at most two exceptional $T^{n-1}$-orbits in $\mu^{-1}_{T^{n-1}}(c)$. 
\end{proof}

%\begin{example}
%{\rm 
%Consider the standard  $T^{2}$-action on $\C^{2}$ defined by $(z_{1},z_{2})\mapsto (e^{\sqrt{-1}\theta_{1}}z_{1},e^{\sqrt{-1}\theta_{2}}z_{2})$. We take a subtorus $S^{1}\subset T^{2}$ so that $S^{1}$ acts on $\C^{2}$ by the Hopf action; $(z_{1},z_{2})\mapsto (e^{\sqrt{-1}\theta}z_{1},e^{\sqrt{-1}\theta}z_{2})$.  In this case, the regular level set $\mu^{-1}_{S^{1}}(c)$ coincides with the standard $3$-sphere $S^{3}(r)$ of some radius $r$, and $S^{1}$ acts on $\mu^{-1}_{S^{1}}(c)$ freely. Thus, there is no exceptional $S^{1}$-orbit in $\mu^{-1}_{S^{1}}(c)$, and the quotient space $\mu^{-1}_{S^1}(c)/S^1$ is nothing but the complex projective space $\C P^1$.
%}
%\end{example}

\begin{example}\label{eg:toric}
{\rm
Consider the standard $T^{n+1}$-action on the complex Euclidean space $\C^{n+1}$ defined by
\[
(z_1,z_2,\ldots, z_{n+1})\mapsto (e^{\sqrt{-1}\theta_1}z_1,e^{\sqrt{-1}\theta_2}z_2,\ldots, e^{\sqrt{-1}\theta_{n+1}}z_{n+1})
\]
This commutes with the Hopf $S^1$-action ${\bm z}\mapsto e^{\sqrt{-1}\theta}{\bm z}$ (${\bm z}\in \C^{n+1}$), and hence, the $T^{n+1}$-action naturally induces an action of $T^n=T^{n+1}/S^1$ on the complex projective space $\C P^n=S^{2n+1}/S^1$;
\[
[z_1:z_2:\cdots: z_{n+1}]\mapsto [e^{\sqrt{-1}\theta_1}z_1: e^{\sqrt{-1}\theta_2}z_2: \cdots :e^{\sqrt{-1}\theta_{n+1}}z_{n+1}],
\]
where  $T^n=T^{n+1}/S^1$ is identified with the subtorus $T^n=\{(e^{\sqrt{-1}\theta_1},\cdots, e^{\sqrt{-1}\theta_{n+1}})\in T^{n+1}\mid \theta_1+\theta_2+\cdots+\theta_{n+1}=0\}$. 

Identifying $\mathfrak{t}^{n+1} \cong \mathbb{R}^{n+1}$,
the moment map $\mu_0: \C^{n+1}\to \R^{n+1}$ of the standard $T^{n+1}$-action on $\C^{n+1}$ is given by
\[
\mu_0(z_1,\ldots,z_{n+1})=-\frac{1}{2}(|z_1|^2,\ldots, |z_{n+1}|^2),
\]
and the moment map ${\mu}_1: \C^{n+1}\to \R^n$ for the $T^n$-action   on $\C^{n+1}$ is obtained by 
\[
\mu_1(z_1,\ldots,z_{n+1})=-\frac{1}{2(n+1)}
\begin{pmatrix}
n|z_1|^2-|z_2|^2-\cdots-|z_{n+1}|^2\\
-|z_1|^2+n|z_2|^2-\cdots-|z_{n+1}|^2\\
\vdots\\
-|z_1|^2-\cdots -|z_n|^2+n|z_{n+1}|^2
\end{pmatrix}.
\] 
Then, 
the moment map $\mu: \C P^n\to \R^{n}$ for the $T^n$-action on $\C P^n$ is given by
\begin{align}\label{eq:cm}
\mu([z_1:\cdots: z_{n+1}])=\frac{1}{|{\bm z}|^2}\mu_1(z_1,\ldots, z_{n+1}),
\end{align}
where $|{\bm z}|^2=\sum_{i=1}^{n+1}|z_i|^2$.
We remark that $\mu$ is the same as the canonical moment map $\mu_{{\rm can}}$ of Proposition \ref{prop-mucan}.
Indeed, the Clifford torus $T^n_{C}=T^n\cdot [1:1:\cdots:1]$ is the unique minimal Lagrangian $T^n$-orbit, and it is easy to see that  $T^n_{C}=\mu^{-1}(0)$. 
Corollary \ref{cor:K} implies that $\mu_{{\rm can}}$ vanishes on the Clifford torus, and since moment maps differ only by a constant, they must coincide.

Let us consider the case when $n=2$ for simplicity.  We take positive integers $a,b\in \mathbb{Z}_{>0}$ such that $a$ and $b$ are relatively prime.   
We consider a subtorus $S^1_{a,b}$ of $T^2\subset T^3$ by 
\begin{align*}
&S^1_{a,b}=\{(e^{\sqrt{-1}\theta_1},e^{\sqrt{-1}\theta_{2}}, e^{\sqrt{-1}\theta_{3}})\in T^{3}\mid (\theta_1,\theta_2,\theta_3)=\theta{\bm v}_{a,b}, \theta\in \R\},\\
&{\rm where}\,\  {\bm v}_{a,b}:=a(1,0,-1)+b(0,1,-1)=(a,b,-(a+b)).
\end{align*}
% More precisely,
$S^1_{a,b}$ acts on $\C P^2$ by
\[
[z_1:z_2:z_3] \mapsto [e^{\sqrt{-1}a\theta}z_1: e^{\sqrt{-1}b\theta}z_2: e^{-\sqrt{-1}(a+b)\theta}z_3].
\]
 For simplicity, we restrict to the case where $a \not\equiv b \text{ (mod } 3)$ so that the $S^1_{a,b}$-action is effective; this ensures that the circle does not intersect the order $3$ kernel of the action of $T^2$ on $\mathbb{C}P^2$. Since the canonical moment map $\mu_{S^{1}_{a,b}}$ for the $S^1_{a,b}$-action
satisfies the relation $\langle \mu_{S^{1}_{a,b}}(p), {\bm v}_{a,b}\rangle=\langle \mu(p), {\bm v}_{a,b}\rangle$, where 
$\mu:\mathbb{CP}^2 \to \mathbb{R}^2$ is the canonical moment map for $T^{2}$-action given by \eqref{eq:cm}, 
the $0$-level set of $\mu_{S^1_{a,b}}: \C P^{2}\to \R$ is given by
\begin{align*}
\mu_{S^1_{a,b}}^{-1}(0)=\big\{[z_{1}:z_{2}:z_{3}]\in \C P^{2}\,\big|\, a|z_1|^2+b|z_2|^2-(a+b)|z_3|^2=0\big\}.
\end{align*}
It is easy to check that 
\begin{itemize}
\item  $[z_1:z_2:z_3]\in \mu_{S^1_{a,b}}^{-1}(0)$ is a principal point of the $S^1_{a,b}$-action\\
$\Longleftrightarrow$ $z_{1}z_{2}z_{3}\neq 0$.
\item $[z_1:z_2:z_3]\in \mu_{S^1_{a,b}}^{-1}(0)$ is an exceptional point of the $S^1_{a,b}$-action\\
$\Longleftrightarrow$ $z_{1}z_{2}= 0$ (Note that $z_3\neq 0$ on $\mu_{S^1_{a,b}}^{-1}(0)$).
\item $\mu_{a,b}^{-1}(0)$ has no singular point of the $S^1_{a,b}$-action (Indeed, $[z_1:z_2:z_3]\in \C P^2$ is a fixed point i.e.\ singular point of the $S^1_{a,b}$-action $\Longleftrightarrow$ $[z_1:z_2:z_3]=[1:0:0], [0:1:0]$ or $[0:0:1]$).
\end{itemize}

In particular,  the stabiliser subgroup $Z_p$ is discrete for any $p\in\mu_{a,b}^{-1}(0)$. This implies that $0$ is a regular value of $\mu_{S^1_{a,b}}$ by Lemma \ref{lem:mo}.
However, $S^1_{a,b}$-action has exactly two exceptional orbits in $\mu_{S^1_{a,b}}^{-1}(0)$; The orbit through $p_1=[\sqrt{\frac{a+b}{a}}:0:1]$ and $p_2=[0:\sqrt{\frac{a+b}{b}}:1]$.  Note that the stabiliser subgroup at $p_{1}$ (resp. $p_{2}$) is given by $Z_{p_{1}}=\mathbb{Z}_{2a+b}$ (resp. $Z_{p_{2}}=\mathbb{Z}_{a+2b}$), and the order of each stabiliser subgroup at exceptional point is greater than 2 since $a$ and $b$ are positive integers. In particular, the quotient space $\mu^{-1}_{S^1_{a,b}}(0)/S^1$ is an orbifold with two cone  points. 

\begin{remark}{\rm
One may consider the following type of $S^{1}$-action:
\[
[z_{1}:z_{2}:z_{3}]\mapsto [e^{\sqrt{-1}\theta}z_{1}:z_{2}:e^{-\sqrt{-1}\theta}z_{3}].
\]
In this case, $\mu^{-1}_{S^{1}}(0)=\{[z_{1}:z_{2}:z_{3}]\in \C P^{2}\mid |z_{1}|^{2}-|z_{3}|^{2}=0\}$ contains an exceptional $S^{1}$-orbit such that the order of the stabiliser subgroup is equal to $2$, namely,  the orbit through $[1:0:1]$.  However, $\mu^{-1}_{S^{1}}(0)$ contains also a fixed point $[0:1:0]$ of $S^{1}$-action, and this means $0$ is {\it not} a regular value of the moment map (see Lemma \ref{lem:mo}). Compare with the above examples and Proposition \ref{prop:torus}.
}
\end{remark}

We can easily generalise this example to the standard $T^{n}$-action on $\C P^{n}$ for arbitrary $n>2$, although the situation is somewhat more complicated than the case when $n=2$.  However, for any codimension-one subtorus $T^{n-1}\subset T^{n}$ such that $0$ is a regular value of $\mu_{T^{n-1}}$, the quotient space  $\overline{M}=\mu_{T^{n-1}}^{-1}(0)/T^{n-1}$ is a 2-dimensional orbifold with at most two singular points by Proposition \ref{prop:torus}. 
%If $\gamma$ is a closed curve embedded in the regular part of $\overline{M}$,
%then the pre-image $\pi^{-1}(\gamma)$ of the projection $\pi: \mu_{T^{n-1}}^{-1}(0)\to \overline{M}$ provides a $T^{n-1}$-invariant Lagrangian submanifold contained in $\mu_{T^{n-1}}^{-1}(0)^{{\rm pri}}\subset \C P^{n}$.
%We remark that, for the standard $T^{n}$-action on $\C P^{n}$ given above, the Clifford torus $T^n_{C}$ is always contained in $\mu_{T^{n-1}}^{-1}(0)^{{\rm pri}}$ since  there is no exceptional point of the $T^{n-1}$-action on $T^n_{C}$.
%Moreover, by Theorem \ref{thm:K}, $\gamma_{C}=T^{n}_{C}/T^{n-1}$ is a $\overline{f}$-minimal smooth curve ($\overline{f}$-geodesic) in $\overline{M}=\mu_{T^{n-1}}^{-1}(0)/T^{n-1}$ because the Clifford torus is a minimal Lagrangian submanifold in $\C P^n$.
}
\end{example}

More generally, one may find torus invariant Lagrangian submanifolds in {\it toric Fano manifolds}, that is, compact toric K\"ahler manifolds with positive first Chern class. If $M$ is toric Fano, then for any K\"ahler metric $\omega\in 2\pi c_{1}(M)$, the Ricci form $\rho$ satisfies $\rho=C\omega+ndd^{c}f$ for some $C>0$ and $f\in C^{\infty}(M)$ (Moreover, there exists a canonical K\"ahler metric in the class $2\pi c_{1}(M)$, namely the K\"ahler-Ricci soliton, see \cite{WZ}).  

In this case,  there exists an $f$-minimal Lagrangian $T^n$-orbits $L_{0}$ contained in $M^{\rm pri}$ by Proposition \ref{prop:mintorus}. Furthermore, if we take a codimension-one subtorus $T^{n-1}$ such that 
$0$ is a regular value of the canonical moment map $\mu_{T^{n-1}}$ of the $T^{n-1}$-action, then by Corollary \ref{cor:K}, $L_0$ is reduced to a weighted geodesic $\gamma_0=L_0/T^{n-1}$ contained in the regular part of the K\"ahler quotient $\overline{M}=\mu_{T^{n-1}}^{-1}(0)/T^{n-1}$. We remark that if $\gamma$ is a Hamiltonian deformation of $\gamma_0$, then the generalised mean curvature form $\alpha_{\overline{K}}$ of $\gamma$ is exact since the exactness is preserved under any Hamiltonian deformation (see \cite[Proposition 2.14]{KK}). Moreover, the preimage $L=\pi^{-1}(\gamma)$ is a $T^{n-1}$-invariant, cohomogeneity-one Lagrangian submanifold in $M$, where $\pi: \mu_{T^{n-1}}^{-1}(0)\to \overline{M}$ is the natural projection, and the generalised mean curvature form $\alpha_K$ of $L$ is also exact since $\alpha_K=\pi^*\alpha_{\overline{K}}$ by Theorem \ref{thm:K}. In this way, we obtain many examples of cohomogeneity-one Lagrangian submanifolds with exact generalised mean curvature form. 

In particular, all the results stated in the previous subsections, as well as the main results (Theorems \ref{thm:main1} and \ref{thm:main2}) to be presented in the next section, are applicable to the $T^{n-1}$-invariant Lagrangian submanifold $L^{n}$ in $M$ if $L^{n}$ satisfies appropriate assumptions described in each Theorem.

\section{Convergence of cohomogeneity-one LMCF}\label{sec:mainproof}

In this section, we prove our main results concerning a long-time behaviour of cohomogeneity-one Lagrangian mean curvature flow. Results are summarized as follows.

\begin{theorem}\label{thm:main1}
Let  $(M^{2n},J,\omega,g)$ be a closed almost-Einstein K\"ahler manifold such that
$\rho=C\omega+ndd^{c}f$ for some positive constant $C>0$ and $f\in C^{\infty}(M)$. Let $G$ be a connected closed Lie subgroup of ${\rm Aut}(M,\omega,J)$.
Suppose that $0$ is a regular value of the canonical moment map $\mu_{{\rm can}}:M\to \fg^*$ of the $G$-action.
% and $G$ acts on $\mu_{{\rm can}}^{-1}(0)$ locally freely and effectively and on $\mu_{{\rm can}}^{-1}(0)^{{\rm pri}}$ freely. 

Let  $\varphi: L^{n}\to M$ be a $G$-equivariant Lagrangian embedding of closed manifold $L^{n}$ such that the generalised mean curvature form $\alpha_{K}$ is exact. Suppose that the $G$-action on $L^{n}$ is free and cohomogeneity-one. Then we have the following:
\begin{enumerate}
\item[(i)] The $f$-LMCF $\{F_{t}\}_{t}$ starting from $\varphi=F_{0}$ exists for all  $t\in[0,\infty)$ and $F_{t}: L\to M$ is a Lagrangian embedding for every $t\in [0,\infty)$. Moreover, we have $L_{t}:=F_{t}(L)\subset \mu_{{\rm can}}^{-1}(0)$ for any $t\in [0,\infty)$.
\item[(ii)] There exists a $G$-invariant $f$-minimal Lagrangian submanifold $L_{\infty}$ contained in $\mu_{{\rm can}}^{-1}(0)$ and a subsequence $\{t_{i}\}_{i=1}^{\infty}$ with $t_{i}\to \infty$ such that each $L_{i}:=F_{t_{i}}(L)$ 
converges smoothly and graphically in $\mu_{{\rm can}}^{-1}(0)$ to $L_\infty$ as $i \to \infty$. 

More precisely, the convergence statement means the following:
there exists a diffeomorphism $\widetilde{\Phi}: L_{\infty}\times (-\epsilon,\epsilon)\to \widetilde{U}_{\mathcal{\epsilon}}$ onto an open neighbourhood $\widetilde{U}_{\epsilon}\subset \mu_{can}^{-1}(0)$ of $L_{\infty}$  such that:
\begin{itemize}
\item $L_{\infty}=\{\widetilde{\Phi}(p,0)\mid p\in L_{\infty}\}$.
\item For each $i$, there is a smooth function $\widetilde{h}_{i}: L_{\infty}\to \R$ so that 
$L_{i}=\{\widetilde\Phi(p,\widetilde{h}_{i}(p))\mid p\in L_{\infty}\}$, and $\widetilde{h}_{i}\to 0$ in $C^{\infty}$ as $i\to \infty$.
\end{itemize}
 % Moreover, we can take the subsequence $\{t_{i}\}_{i=1}^{\infty}$ so that $\widetilde{h}_{i}\to 0$ in $C^{\infty}$ as $i\to \infty$. In this sense, the sequence of Lagrangian submanifolds $\{L_{i}\}_{i=1}^{\infty}$ smoothly converges to $L_{\infty}$.
\end{enumerate}
\end{theorem}

%Finally, we use our curvature bounds to upgrade from smooth subsequential convergence to smooth convergence in time.
Under an additional assumption on $M$, we obtain a much stronger result.  Namely, we can upgrade from smooth subsequential convergence to smooth convergence in time.

\begin{theorem}\label{thm:main2}
    Assume we are in the setting of Theorem \ref{thm:main1}, and let $L_\infty$ be the limiting $f$-minimal Lagrangian submanifold from the conclusion of that theorem. Suppose furthermore that $(M,g, f)$ is real analytic. 

    Then $L_t$ converges smoothly and graphically to $L_\infty$ as $t \to \infty$.

    More precisely, the convergence statement means the following: Let $\widetilde \Phi: L_\infty \times (-\epsilon,\epsilon) \to \widetilde U_\epsilon$ be the diffeomorphism of Theorem \ref{thm:main1}. There exists $T_0>0$ such that for $t \in [T_0,\infty)$, there exists $u_t:L_\infty \to \mathbb{R}$ so that $L_t = \{\widetilde \Phi(p, u_t(p))\,:\,p \in L_\infty\}$ and $u_t\to 0$ in $C^\infty$ as $t \to \infty$.
\end{theorem}

Note that the extra assumption on the real analyticity for $(M,g,f)$ is automatically satisfied when $M$ is a K\"ahler-Einstein manifold with $f=0$, since any Einstein metric is real analytic (cf.\ \cite{DK}).

\subsection{Proof of Theorem \ref{thm:main1}}
In this subsection, we give a proof of Theorem \ref{thm:main1}. The result is essentially
a consequence of Propositions \ref{prop:redu}, \ref{prop:topo} and Theorems \ref{thm:subcon}, \ref{thm:beh}.

\begin{proof}[Proof of  Theorem \ref{thm:main1}]
Let $T_{\max}\in [0,\infty) \cup \{\infty\}$ be the maximal existence time of the $f$-LMCF $F:L^{n}\times [0,T_{\max})\to M^{2n}$ starting from $\varphi=F_{0}$. Since the generalised mean curvature form $\alpha_{K}$ of $\varphi$ is exact, by Proposition \ref{prop:redu}-(i), we have that $F_{t}(L)$ is contained in the $0$-level set of the canonical moment map $\mu_{{\rm can}}: M\to \fg^*$ of the $G$-action  for any $t\in [0,T_{\max})$. 
Since $G$ acts on $L$ freely and $L \subset \mu_{\rm can}^{-1}(0)$, it follows that the principal orbits of the $G$-action on $\mu_{\rm can}^{-1}(0)$ have trivial stabiliser, therefore $L \subset \mu_{\rm can}^{-1}(0)^{\rm pri}$ and $G$ acts on $\mu_{\rm can}^{-1}(0)^{\rm pri}$ freely. 

By Proposition \ref{prop-riemorbifoldquotient} the quotient space $\overline{M}=\mu_{{\rm can}}^{-1}(0)/G$ becomes a compact orbifold $\mathcal{O}$,
and $\overline{M}^{{\rm pri}}:=\mu_{{\rm can}}^{-1}(0)^{{\rm pri}}/G$ inherits a (smooth) K\"ahler structure $(\overline{\omega}, \overline{J}, \overline{g})$ by Proposition \ref{prop:dime}. Since we assume that the $G$-action on $L^{n}$ is cohomogeneity-one, we have that ${\rm dim}G=n-1$. Thus, by part (i) of that Proposition, we see ${\rm dim}(\mu_{{\rm can}}^{-1}(0))=n+1$ and  ${\rm dim}(\mathcal{O})=2$. Moreover, $\overline{L}=L/G\simeq S^{1}$. Note that since $\mathcal{O}$ is K\"ahler, it is in particular oriented.

For ease of notation, we set $P:=\mu_{{\rm can}}^{-1}(0)^{{\rm pri}}$ and  denote the natural projections by $\pi: P\to \overline{M}^{{\rm pri}}$ and $\pi_{L}: L\to \overline{L}$. We proceed with the proof by dividing it into several claims.
\setcounter{claim}{0}
\begin{claim}\label{cl:m1}
$F_{t}(L)\subset P$ for any $t\in[0,T_{\max})$.
\end{claim}

Suppose the contrary were true, and we put $t_{0}:={\rm inf}\{t\in [0,T_{\max})\mid F_{t}(L)\cap \mu_{{\rm can}}^{-1}(0)^{{\rm ex}}\neq \emptyset\}$, where  $\mu_{{\rm can}}^{-1}(0)^{{\rm ex}}$ denotes the set of exceptional $G$-orbits in $ \mu_{{\rm can}}^{-1}(0)$. Since  $F_{0}(L)\subset P$ by assumption and $P$ is an open subset in $\mu_{{\rm can}}^{-1}(0)$,  we may assume that $0<t_{0}<T_{\max}$. Then, $F_{t}(L)\subset P$ for any $t\in [0,t_{0})$, and hence, Proposition \ref{prop:redu}-(ii) shows that the reduced map $\overline{F}_{t}: S^{1}\times [0,t_{0})\to \overline{M}^{{\rm pri}}$ is an $\overline{f}$-LMCF. In our situation, $\overline{F}_{t}$ is nothing but an $\overline{f}$-CSF on the orbifold $\mathcal{O}$.  If $t_{0}<\infty$ is not a maximal existence time of $\overline{F}_{t}$, then Theorem \ref{thm:beh} shows that $\overline{F}_{t_{0}}(S^{1})\subset \mathcal{O}^{{\rm reg}}$. Therefore, $\pi^{-1}(\overline{F}_{t_{0}}(S^{1}))$ does not contain any exceptional $G$-orbit. However,  this contradicts the assumption that $F_{t_{0}}(L)\cap \mu_{{\rm can}}^{-1}(0)^{{\rm ex}}\neq \emptyset$ because it must hold that $\pi\circ F_{t_{0}}=\overline{F}_{t_{0}}\circ \pi_{L}$. If $t_{0}<\infty$ is the maximal existence time of $\overline{F}_{t}$, then Theorem \ref{thm:beh} again implies that $\overline{F}_{t}$ shrinks to a point, and this also yields a contradiction because the original flow $F_{t}$ collapses at the time $t_{0}<T_{\max}$. This proves Claim \ref{cl:m1}.\\

In particular, we obtain a solution of $\overline{f}$-CSF $\overline{F}:  S^{1}\times [0,T_{\max})\to \mathcal{O}^{{\rm reg}}$. According to  Section \ref{sec:csf}, we denote $\overline{F}_{t}$ by $\gamma_{t}$, and let $\overline{T}_{\max}(\geq T_{\max})$ be the maximal existence time of $\gamma_{t}$ in $\mathcal{O}^{{\rm reg}}$. 

\begin{claim}\label{cl:m2}
$F_{t}: L\to M$ is a smooth embedding for any $t\in[0,T_{\max})$.
\end{claim}

We first recall that $\gamma_{t}:S^{1}\to \mathcal{O}^{{\rm reg}}$ is a smooth embedding for any $t\in [0,\overline{T}_{max})$.  If $F_{t}(x)=F_{t}(y)$, then $\overline{F}_{t}\circ \pi_{L}(x)=\overline{F}_{t}\circ \pi_{L}(y)$. This shows $\pi_{L}(x)=\pi_{L}(y)$ since $\overline{F}_{t}=\gamma_{t}$ is an embedding. Thus, both $x$ and $y$ belong to the same fibre in $L$, which implies that there exists an element $g_y\in G$ such that $y=g_{y}\cdot x$. Then, we see that $F_{t}(x)=F_{t}(y)=F_{t}(g_{y}\cdot x)=g_{y}\cdot F_{t}(x)$ by  the $G$-equivariantness of $F_{t}$.  Since  $F_{t}(x)\in P=\mu_{{\rm can}}^{-1}(0)^{{\rm pri}}$ and the $G$-action is free on $P$, the isotropy subgroup $G_{F_{t}(x)}$of the $G$-action at $F_{t}(x)$ is trivial. Therefore, $g_{y}=e$ and this shows $y=x$. This proves that  $F_{t}: L\to M$ is injective, and this implies that the continuous bijective map $F_{t}: L\to F_{t}(L)$ is a homeomorphism since $L$ is compact and $M$ is Hausdorff. Therefore, $F_{t}$ is an embedding. \\

In the following, we set $L_{t}:=F_{t}(L)\subset M$.

\begin{claim}\label{cl:m3}
$T_{\max}=\infty$.
\end{claim}
By assumption and Proposition \ref{prop:topo},  the orbifold $\mathcal{O}$ is homeomorphic to $2$-sphere, and hence, we see that each embedding $\gamma_{t}: S^{1}\to \mathcal{O}$ encloses a region. Moreover, since the generalised mean curvature form $\alpha_{K}$ of the initial Lagrangian submanifold $F_{0}=\varphi$ is a $G$-invariant exact $1$-form, the compactness of $G$ implies that there is a $G$-invariant function $h\in C^{\infty}(L)$ so that $\alpha_{K}=dh$. Then, we obtain a well-defined function $\overline{h}\in C^{\infty}(S^{1})$ satisfying $\overline{h}\circ \pi_{L}=h$ and $\pi^{*}_{L}d\overline{h}=dh=\alpha_{K}$, where $\pi_{L}:L\to L/G\simeq S^{1}$ is the projection. On the other hand,  the generalised mean curvature form $\alpha_{\overline{K}}$ of $\gamma_{0}=\overline{F}_{0}$ satisfies that $\pi_{L}^{*}\alpha_{\overline{K}}=\alpha_{K}$ by Theorem \ref{thm:K}. Therefore, we see that $\alpha_{\overline{K}}=d\overline{h}$, namely, $\alpha_{\overline{K}}$ is also exact. Thus, Proposition \ref{prop:exact} implies that the $\overline{f}$-CSF $\gamma_{t}$ starting from $\gamma_{0}$ generates a Hamiltonian deformation and by Proposition \ref{prop:areapre}, $\gamma_{t}$ preserves the enclosed area. 

It now follows that $\lim_{t\to T_{\max}}L_{\overline{f}}(\gamma_t)\neq 0$. Indeed,  if $\lim_{t\to T_{\max}}L_{\overline{f}}(\gamma_t)=0$, then the usual length $L(\gamma_t)$ also goes to $0$ since $\overline{f}$ is bounded.  Using a similar argument as in Lemma \ref{lem:leng}, we see that $A(E_t)\to 0$ by the isoperimetric inequality, where $E_t$ is the region enclosed by $\gamma_t$; this contradicts preservation of area.

Thus, by Theorem \ref{thm:beh}, the maximal existence time $\overline{T}_{\max}$ of $\gamma_{t}$ in $\mathcal{O}^{{\rm reg}}$ must be $\infty$. Now by Proposition \ref{prop:lift}, there exists an $f$-LMCF $F'_t:L \times [0,\infty) \to M$ starting at $\varphi$ such that $\pi \circ F'_t = \gamma_t$; by uniqueness of solutions to $f$-MCF this solution is an extension of $F_t$, and therefore $T_{\rm max} = \infty$. This proves Claim \ref{cl:m3}, and part (i) of the theorem.\\

We now consider the (sub-)convergence of the $f$-LMCF.
By Theorem \ref{thm:subcon}, there exists a subsequence $\{\gamma_{t_{i}}\}_{i=1}^{\infty}$ such that it converges to a continuous map $\gamma_{\infty}:S^{1}\to \mathcal{O}$. This limiting curve does not pass through any cone point of $\mathcal{O}$. Indeed, if $\gamma_{\infty}$ passes through a cone point, then by Theorem \ref{thm:subcon}-(i) and (ii), the order of cone point must be $2$, and $\gamma_{\infty}$ goes back and forth on a geodesic path. This implies that the area enclosed by $\gamma_{t_i}$ goes to $0$.
However, as mentioned above, $\gamma_{t}$ preserves the enclosed area, and  this is a contradiction. 
Therefore, $\gamma_{\infty}$ does not pass through any cone point.

Then by Theorem \ref{thm:subcon}-(iii), there exists  a subsequence $\{\widehat\gamma_{t_{i}}\}_{i=1}^{\infty}$ of an unparametrised solution $\{\widehat{\gamma}_{t}\}_{t\in [0,\infty)}$ corresponding to $\{\gamma_t\}_{t \in [0,\infty)}$
such that $\widehat\gamma_{i}:=\widehat\gamma_{t_{i}}$ smoothly converges to a smooth $\overline{f}$-minimal embedding $\widehat{\gamma}_{\infty}: S^{1}\to \mathcal{O}^{{\rm reg}}$.  
Now, for any positive constants $a\geq 2$ and $\delta\leq \frac{\pi}{8aL_{\overline{f}}(\widehat{\gamma}_{\infty})}$, we take the $\epsilon$-tubular neighbourhood $\overline{U}_{\epsilon}$  given in Lemma \ref{lem:graph} with respect to the weighted metric $g_{\overline{f}}=e^{2\overline{f}}\overline{g}$ of the $\overline{f}$-geodesic $\widehat{\Gamma}:=\widehat{\gamma}_{\infty}(S^{1})$. Since $\widehat{\gamma}_{i}\to \widehat{\gamma}_{\infty}$, 
it follows that
$\widehat{\gamma}_{i}(S^{1})\subset U_{\epsilon}$ for any sufficiently large $i$. Moreover,  since $L_{\overline{f}}(\widehat{\gamma}_{i})$ is non-increasing and $\kappa_{\overline{f}}(t_{i})\to 0$ (see Subsection \ref{subsec:subcon}), 
we have
that $L_{\overline{f}}(\widehat{\gamma}_{i})\leq aL_{\overline{f}}(\widehat{\gamma}_{\infty})$ and $|\widehat{\kappa}(t_{i})|\leq \delta$ for any sufficiently large $i$, where $\widehat{\kappa}=e^{-\overline{f}}\kappa_{\overline{f}}$ is the geodesic curvature w.r.t.\  the weighted metric $g_{\overline{f}}$. Therefore, by Lemma  \ref{lem:graph}, the embedding $\widehat{\gamma}_{i}$ is described by a graph over  $\widehat{\Gamma}$ for any sufficiently large $i$.
Namely, there exists a smooth function $h_{i}: S^{1}\to 
% \R
(-\epsilon,\epsilon)
$ such that 
\begin{align}\label{eq:gami}
\widehat{\gamma}_{i}(S^{1})=\{{\rm exp}^{\overline{f}}_{\widehat\gamma_{\infty}(x)}(h_{i}(x)\overline{N}_{\widehat\gamma_{\infty}(x)})\mid x\in S^{1}\},
\end{align}
where  ${\rm exp}^{\overline{f}}$ is the exponential map on $(\mathcal{O}^{{\rm reg}},g_{\overline{f}})$ and $\overline{N}$ is the unit normal vector field along $\widehat{\Gamma}$ w.r.t.\  $g_{\overline{f}}$. 

\begin{claim}\label{cl:m4}
$h_{i}\to 0$ in $C^{\infty}$ as $i\to \infty$.
\end{claim}

% Recall that (see also Subsection \ref{sec:2dorbifolds})
Recall from subsection \ref{subsec:curvesonorbs} that
\[
\Phi: S^{1}\times (-\epsilon,\epsilon)\to \overline{U}_{\epsilon},\quad \Phi(x,y)={\rm exp}^{\overline{f}}_{\widehat{\gamma}_{\infty}(x)}(y\cdot \overline{N}_{\widehat{\gamma}_{\infty}(x)})
\]
defines a diffeomorphism onto the neighbourhood $\overline{U}_{\epsilon}$ of $\widehat{\Gamma}$. 
By construction of $h_{i}$ (see \eqref{def:h}), we have 
\[
(\pi_{1}|_{\widetilde{\gamma}_{i}(S^{1})})^{-1}(x)=(x,h_{i}(x)),
\]
where $\pi_{1}: S^{1}\times (-\epsilon,\epsilon)\to S^{1}$ is the projection, $\widetilde{\gamma}_{i}:=\Phi^{-1}\circ \widehat{\gamma}_{i}$ and $\pi_{1}|_{\widetilde{\gamma}_{i}(S^{1})}: \widetilde{\gamma}_{i}(S^{1})\to S^{1}$ is the restricted map which is in fact a diffeomorphism (see the proof of Lemma \ref{lem:graph}). In particular, we see 
\[
\widetilde{\gamma}_{i}(x)=\Big((\pi_{1}\circ \widetilde\gamma_{i})(x), h_{i}\circ(\pi_{1}\circ \widetilde\gamma_{i})(x)\Big).
\]
Since  $\pi_{1}|_{\widetilde{\gamma}_{i}(S^{1})}: \widetilde{\gamma}_{i}(S^{1})\to S^{1}$ is a diffeomorphism and $\widetilde{\gamma}_{i}=\Phi^{-1}\circ \widehat{\gamma}_{i}:S^{1}\to S^{1}\times (-\epsilon, \epsilon)$ is a smooth embedding, we see $\pi_{1}\circ \widetilde{\gamma}_{i}: S^{1}\to S^{1}$ is also a diffeomorphism. 
 Thus, using the projection $\pi_{2}:S^{1}\times (-\epsilon,\epsilon)\to (-\epsilon,\epsilon)$,  we obtain
\[ 
h_{i}=\pi_{2}\circ \widetilde{\gamma}_{i}\circ (\pi_{1}\circ \widetilde{\gamma}_{i})^{-1}.
\]
Since $\widehat{\gamma}_{i}\to \widehat{\gamma}_{\infty}$ smoothly, we have  $\pi_{1}\circ \widetilde{\gamma}_{i}\to \pi_{1}\circ \widetilde{\gamma}_{\infty}$ in $C^{\infty}$ as $i\to \infty$, where $\widetilde{\gamma}_{\infty}=\Phi^{-1}\circ \widehat{\gamma}_{\infty}$. Notice that, by definition of $\Phi$, we see  $\pi_{1}\circ \widetilde{\gamma}_{\infty}=id_{S^{1}}$.  In this situation, it is easy to see that $(\pi_{1}\circ \widetilde{\gamma}_{i})^{-1}\to id_{S^{1}}$ in $C^{0}$, $\p(\pi_{1}\circ \widetilde{\gamma}_{i})^{-1}/\p\theta\to 1$ and  $\p^{k}(\pi_{1}\circ \widetilde{\gamma}_{i})^{-1}/\p\theta^{k}\to 0$ for any $k\geq 2$. This implies that the inverse $(\pi_{1}\circ \widetilde{\gamma}_{i})^{-1}$ also converges to $id_{S^{1}}$ in $C^{\infty}$.
Therefore, we see
 \[
 h_{i}\to\pi_{2}\circ \widetilde{\gamma}_{\infty}\circ id_{S^{1}}=0 \quad (i\to \infty)
 \]
 in $C^{\infty}$, which proves Claim \ref{cl:m4}.\\

Since $\widehat{\gamma}_{i}(S^{1})=\gamma_{i}(S^{1})$ by construction of $\widehat{\gamma}_{i}$ (see Proof of Theorem \ref{thm:subcon}), we have $L_{i}:=F_{t_{i}}(L)=\pi^{-1}(\widehat{\gamma_{i}}(S^{1}))$. Moreover,  by Corollary \ref{cor:K}, $L_{\infty}:=\pi^{-1}(\widehat{\gamma}_{\infty}(S^{1}))$ is a $G$-invariant $f$-minimal Lagrangian submanifold in $M$. 

\begin{claim}\label{cl:m5}
There exists a diffeomorphism $\widetilde{\Phi}: L_{\infty}\times (-\epsilon,\epsilon)\to \widetilde{U}_{\mathcal{\epsilon}}\subset \mu_{can}^{-1}(0)$ such that  $L_{\infty}=\{\widetilde{\Phi}(p,0)\mid p\in L_{\infty}\}$ and for  sufficiently large $i$, there exists $\widetilde{h}_{i}\in C^{\infty}(L_{\infty})$ such that $L_{i}=\{\widetilde\Phi(p,\widetilde{h}_{i}(p))\mid p\in L_{\infty}\}$. Moreover, we have $\widetilde{h}_{i}\to 0$ as $i\to \infty$.
\end{claim}

To prove this, we define a weighted metric on $P=\mu_{{\rm can}}^{-1}(0)^{{\rm reg}}$ by $g_{\widetilde{f}}:=e^{2\widetilde{f}}g$, where $g$ denotes the induced metric on $P\subset M$ and $\widetilde{f}\in C^{\infty}(P)$ is a $G$-invariant function such that $\widetilde{f}=\pi^{*}\overline{f}$ (see \eqref{def:fbar} for more precise description). Then, $(P,g_{\widetilde{f}})\to (\overline{M}^{\rm pri}, g_{\overline{f}})$ is also a Riemannian submersion. Note that both the horizontal and vertical subspaces coincide with those of the original Riemannian submersion $(P,g)\to (\overline{M}^{\rm pri},\overline{g})$. 

Since $L_{\infty}$ is a $G$-invariant submanifold in $P$, we obtain a unit normal vector field
$\widetilde{N}_{p}:=\overline{N}_{\pi(p)}^{\mathcal{H}}$ w.r.t.\  $g_{\widetilde{f}}$ along $L_{\infty}$. By assumption, the codimension of $L_{\infty}$ in $P$ is equal to $1$, and hence, for sufficiently small $\epsilon>0$, the map
\[
\widetilde{\Phi}: L_{\infty}\times (-\epsilon,\epsilon)\to P,\quad \widetilde{\Phi}(p,r):={\rm exp}^{\widetilde{f}}_{p}(r\cdot \widetilde{N}_{p})
\]
defines a diffeomorphism onto an open neighbourhood $\widetilde{U}_{\epsilon}\subset P$  of $L_{\infty}$, where ${\rm exp}^{\widetilde{f}}$ is the exponential map of $(P, g_{\widetilde{f}})$. Obviously, $L_{\infty}=\{\widetilde{\Phi}(p,0)\mid p\in L_{\infty}\}$.

For sufficiently large $i$, the function $h_{i}: S^{1}\to \R$ yields a $G$-invariant smooth function $\widetilde{h}_{i}:=h_{i}\circ \widehat{\gamma}_{\infty}^{-1}\circ \pi: L_{\infty}\to \R$, where we regard $\widehat{\gamma}_{\infty}: S^{1}\to \widehat{\gamma}_{\infty}(S^{1})$ as a diffeomorphism. Our claim is that 
\begin{align}\label{eq:gL}
L_{i}=\{\widetilde{\Phi}(p, \widetilde{h}_{i}(p))\mid p\in L_{\infty}\}.
\end{align}
To show this, it suffices to check that the set of right hand side is $G$-invariant and its projection under the map $\pi: P\to \overline{M}^{\rm pro}$ coincides with $\widehat{\gamma}_{i}(S^{1})$. Since $\widetilde{f}$ is $G$-invariant, $G$ acts on $(P,g_{\widetilde{f}})$ isometrically, and hence we see that 
\[
g\cdot {\rm exp}_{p}^{\widetilde{f}}(\widetilde{h}_{i}(p)\widetilde{N}_{p})={\rm exp}_{g\cdot p}^{\widetilde{f}}(g_{*}\{\widetilde{h}_{i}(p)\widetilde{N}_{p}\})={\rm exp}_{g\cdot p}^{\widetilde{f}}(\widetilde{h}_{i}(g\cdot p)\widetilde{N}_{g\cdot p})
\]
since $\widetilde{h}_{i}$ is $G$-invariant and $g_{*}\widetilde{N}_{p}=\widetilde{N}_{g\cdot p}.$ This shows that the set of RHS of \eqref{eq:gL} is $G$-invariant. Moreover, since $\pi: (P, g_{\widetilde{f}})\to (\overline{M}^{\rm pri},g_{\overline{f}})$ is a Riemannian submersion, we have
$
\pi({\rm exp}_{p}^{\widetilde{f}}\overline{X}^{\mathcal{H}})={\rm exp}_{\pi(p)}^{\overline{f}}\overline{X}
$
for any $p\in P$ and $\overline{X}\in T_{\pi(p)}\overline{M}^{\rm pri}$ with sufficiently small norm $|\overline{X}|$. Thus, by \eqref{eq:gami} and definitions of $\widetilde{h}_{i}$ and $\widetilde{N}$, we see that the projection of the RHS of \eqref{eq:gL} coincides with $\widehat{\gamma}_{i}(S^{1})$. This proves \eqref{eq:gL}.

Since $h_{i}\to 0$ as $i\to \infty$, we also have $\widetilde{h}_{i}\to 0$ in $C^{\infty}$. In this sense, $L_{i}$ smoothly converges to $L_{\infty}$. This completes the proof of Theorem \ref{thm:main1}.
\end{proof}

\subsection{Proof of Theorem \ref{thm:main2}}
Next, we prove Theorem \ref{thm:main2}. 
Our proof below may be regarded as an adapted version of the proof of Simon's theorem \cite{Simon} for our setting.  A key ingredient to obtain a smooth convergence from a subsequential convergence is the following \L{}ojasiewicz--Simon inequality.
 
 \begin{theorem}[cf.\ Theorem 5.3 in \cite{Law}]\label{thm:LS}
 Let $(\Sigma,g)$ be a compact Riemannian manifold and $\mathcal{E}: C^1(\Sigma)\to \R$ be a functional such that it is written as
 \[
 \mathcal{E}(u)=\int_{\Sigma} E(p,u(p),\nabla u(p))\,dv_{\Sigma},
 \]
where $dv_{\Sigma}$ is the volume measure w.r.t.\  $g$. Here, $E(p,q,z)$ is a real valued smooth function for variables $p\in \Sigma$, $q\in \R$ and $z\in T_p\Sigma$.  
Suppose that $E$ satisfies the following two conditions:
\begin{enumerate}
\item For each $p\in \Sigma$, the function $E$ is uniformly convex with respect to $z$ when $q=0$, that is, there exists a positive constant $c$ independent of $z$ such that
\[
\frac{d^2}{ds^2}E(p,0,sz)\Big{|}_{s=0}\geq c|z|^2,\quad \forall z\in T_p\Sigma.
\]
\item For each $p\in \Sigma$, there exists a positive constant $\delta$ such that if $|q|, |z|<\delta$, then  $E(p,q,z)$ is real-analytic with respect to the latter two variables $q, z$.
\end{enumerate}

We define the {\rm Euler-Lagrange functional} $\mathcal{M}: C^2(\Sigma)\to C^0(\Sigma)$ by 
\[
\frac{d}{ds}\mathcal{E}(u+sv)\Big{|}_{s=0}=-\langle \mathcal{M}(u),v\rangle_{L^2(\Sigma)},\quad \forall v\in C^2(\Sigma),
\]
and suppose that $\mathcal{M}(0)=0$.

 Then, there exists a neighbourhood  $\mathcal{U}$ of $0$ in $C^{2,\alpha}(\Sigma)$, constants $\beta\in (\frac{1}{2}, 1)$ and $C_0>0$ depending only on $\Sigma$ and the form of $\mathcal{E}$ such that, we have the following inequality.
\begin{align}\label{eq:LS}
|\mathcal{E}(u)-\mathcal{E}(0)|^\beta\leq C_0\|\mathcal{M}(u)\|_{L^2},\quad \forall u\in \mathcal{U}.
\end{align}

 \end{theorem}

In the following, we give a proof of Theorem \ref{thm:main2}.

Let $L_\infty$ be an $f$-minimal Lagrangian submanifold in $M$ obtained in Theorem \ref{thm:main1}. 
 We fix arbitrary  $\alpha\in (0,1)$.  For a smooth function $u\in C^{2,\alpha}(L_\infty)$, we consider a normal graph over $L_\infty$ contained in $P=\mu_{\rm can}^{-1}(0)^{\rm pri}$ with respect to the weighted metric $g_{\widetilde{f}}$;
\[
L_u:=\{\widetilde{\Phi}(p,u(p))\mid p\in L_\infty\}=\{{\rm exp}_p^{\widetilde{f}}(u(p)\widetilde{N}_p)\mid p\in L_\infty\},
\]
where we used the same notation given in the proof of Theorem \ref{thm:main1}. Also, we set 
\[
G_u: L_\infty \to L_u,\quad G_u(p):=\widetilde{\Phi}(p,u(p))={\rm exp}_p^{\widetilde{f}}(u(p)\widetilde{N}_p)
\]

We consider a sufficiently small open subset in $C^{2,\alpha}(L_\infty)$ around the origin
\[
\mathcal{U}_R:=\{u\in C^{2,\alpha}(L_\infty)\mid \|u\|_{C^{2,\alpha}}< R\}
\]
so that $G_u$ defines a diffeomorphism for any $u\in \mathcal{U}_R$, where $R>0$. In this case, $L_u$ is identified with the image of the embedding $G_u: L_\infty\to M$. 
Then, we consider a functional $\mathcal{E}: \mathcal{U}_R\to \R$ defined by  
\[
\mathcal{E}(u):={\rm Vol}_f(L_u).
\]
Let $dv_\infty$ be the volume measure on $L_\infty$ w.r.t.\  $g$. We define a smooth function $J_u$ on $L_\infty$ by
\[
dv_u=J_udv_\infty,
\] 
where $dv_u$ is the volume measure w.r.t.\  $G_u^*g$. More precisely, 
 identifying $L_\infty$ with $L_u$ by the diffeomorphism $G_u$, we have
\[
\mathcal{E}(u)=\int_{L_\infty} e^{nf\circ G_u}\, dv_u=\int_{L_\infty}e^{nf\circ G_u}\cdot J_u\,dv_\infty.
\]
We thus put
\begin{align}\label{eq:E}
E(p, u, \nabla u):=e^{nf\circ G_u}\cdot J_u.
\end{align}
Although we are interested in a function $u\in \mathcal{U}_R\subset C^{2,\alpha}(L_\infty)$, one may extend the functional $\mathcal{E}$ to the functional defined on $C^{1}(L_\infty)$ with the real-valued function $E$ given by $\eqref{eq:E}$.

\begin{lemma}\label{lem:ana}
The real valued function $E$  satisfies assumptions {\rm (i)} and {\rm (ii)} given in Theorem \ref{thm:LS}.
\end{lemma}

\begin{proof}
We fix a local coordinate chart $(x_1,\ldots, x_n)$ of $L_\infty$, and take the Fermi coordinate $(x_1,\ldots, x_n, r)$ on the tubular neighbourhood of $L_\infty$ in $P$ with respect to the weighted metric $g_{\widetilde{f}}$. We set $\p_i(p,r):=\p \widetilde{\Phi}/\p x_i$ and $\p_r(p,r):=\p \widetilde{\Phi}/ \p r$, where $\widetilde{\Phi}(p,r)={\rm exp}_p^{\widetilde{f}}(r\widetilde{N}_p)$. Then, 
\begin{align*}
(G_u^*g)_{ij}&=g\Big(dG_u\Big(\frac{\p}{\p x_i}\Big), dG_u\Big(\frac{\p}{\p x_j}\Big)\Big)
=g\Big (\p_i+\frac{\p u}{\p x_i}\p_r, \p_j+\frac{\p u}{\p x_j}\p_r\Big)\\
&=e^{-2\widetilde f\circ G_u}\cdot g_{\widetilde{f}}\Big (\p_i+\frac{\p u}{\p x_i}\p_r, \p_j+\frac{\p u}{\p x_j}\p_r\Big)\\
&=e^{-2\widetilde f \circ G_u}\cdot \Big{\{} g_{\widetilde{f}}(\p_i, \p_j)+\frac{\p u}{\p x_i}\frac{\p u}{\p x_j}\Big{\}}
%&=g(\p_i, \p_j)+e^{-2f\circ G_u}\frac{\p u}{\p x_i}\frac{\p u}{\p x_j}
 \end{align*}
since $g_{\widetilde{f}}(\p_i,\p_r)=0$ and $g_{\widetilde{f}}(\p_r,\p_r)=1$. We put 
\begin{align*}
&g_{ij}(p,r):=g(\p_i, \p_j)(p,r),\quad \widetilde{g}_{ij}(p,r):=g_{\widetilde{f}}(\p_i, \p_j)(p,r),\\
& z_i:=g_{\infty}\Big(z, \frac{\p}{\p x_i}\Big), \quad {\bm v}(z):=\begin{pmatrix} z_1\\ \vdots \\ z_n\end{pmatrix},\quad  {v}_{ij}(p,z):=({\bm v}(z){}^t\!{\bm v}(z))_{ij}
\end{align*}
for $z\in T_p L_\infty$.
%Note that $g_{ij}(p,0)=g_{\infty}(\p/\p x_i, \p/\p x_j)$.  
%Then, we see 
%\begin{align*}
%{\rm det}((G_u^*g)_{ij})(p)&={\rm det}\Big(g_{ij}(p,u(p))+e^{-2f\circ G_u}\cdot v_{ij}(\nabla u) (p)\Big)\\
%&={\rm det}\Big(g_{ij}(p,u(p))\Big)\cdot \{1+{}^t\!{\bm v}(\nabla u)(g_{ij}(p,u(p)))^{-1}{\bm v}(\nabla u)\}
%\end{align*}
Then by definition, $J_u$ is locally expressed by
\[
J_u(p)=e^{-n \widetilde{f}\circ G_u}\sqrt{\frac{{\rm det}\Big(\widetilde{g}_{ij}(p,u(p))+ v_{ij}(p, \nabla u)\Big)}{{{\rm det}(g_{ij}(p,0))}}}.
\]
%Since $u_{,i}=g_{\infty}(\nabla u, \p_i(p,0))$, we see 
%\[
%E(p,u,\nabla u)=\sqrt{\frac{{\rm det}\Big{(}g_{\widetilde{f}}(\p_i, \p_j)(p,u(p))+g_{\infty}(\nabla u, \p_i(p,0))g_{\infty}(\nabla u, \p_j(p,0))\Big{)}}{{{\rm det}(g_{\infty})_{ij}}}}
%\]

We shall check that $E$ satisfies assumptions (i) and (ii) given in Theorem \ref{thm:LS} using this local expression.

\begin{itemize}
\item[(i)] Fix arbitrary $p\in L_\infty$ and take a local coordinate $(x_1,\ldots, x_n)$ around $p$. Then the function $E$ can be locally written by 
\begin{align*}
E(p,q,z)= {e^{n(f - \widetilde f)\circ G_u}}\sqrt{\frac{{\rm det}\Big(\widetilde{g}_{ij}(p,q)+ v_{ij}(p,z) \Big)}{{{\rm det}(g_{ij}(p,0))}}},
\end{align*}
where $q\in \R$, $z\in T_pL_\infty$. When $q=0$, the matrix determinant lemma shows that 
\begin{align*}
{\rm det}\Big(\widetilde{g}_{ij}(p,0)+v_{ij}(p,z)\Big)
&={\rm det}(\widetilde{g}_{ij}(p,0))\{1+{}^t\!{\bm v}(z) (\widetilde{g}_{ij}(p,0))^{-1} {\bm v}(z)\}\\
&=e^{2n \widetilde f(p)}{\rm det}(g_{ij}(p,0))\{1+e^{-2\widetilde f(p)}g^{ij}(p,0)z_i z_j \}\\
&=e^{2n \widetilde f(p)}{\rm det}(g_{ij}(p,0))\{1+e^{-2\widetilde f(p)}|z|^2 \}.
\end{align*}
Thus, we obtain 
\[
E(p,0,sz)=e^{nf(p)}\sqrt{1+e^{-2\widetilde f(p)}|z|^2\cdot s^2}.
\]
Then, it is easy to see that 
\[
\frac{d^2}{ds^2}E(p,0,sz)\Big|_{s=0}=e^{n f(p) -2 \widetilde f(p)}|z|^2\geq c|z|^2
\]
for some positive constant $c>0$ independent of $p, z$. This confirms the assumption (i).

\item[(ii)] We fix $p\in L_{\infty}$. We shall show that there exists a positive constant $\delta(p)>0$ such that $E(p,q,z)$ is real analytic on $(q,z)$ whenever $|q|, |z|<\delta(p)$. Since $(M,g,f)$ is real analytic, any smooth $f$-minimal submanifold is also real analytic (see \cite[Lemma 1]{Leung}). Thus we can take the Fermi coordinate $(x_1,\ldots, x_n,r)$ so that all the functions $g_{ij}(p,r)$, $\widetilde{g}_{ij}(p,r)$ and ${v}_{ij}(p,z)$ are real analytic. Moreover, since 
\[
{\rm det}\Big(\widetilde{g}_{ij}(p,0)+ v_{ij}(p,0) \Big)={\rm det}(\widetilde{g}_{ij}(p,0))>0,
\]
we have
\[
{\rm det}\Big(\widetilde{g}_{ij}(p,q)+ v_{ij}(p,z) \Big)>0
\]
whenever $|q|, |z|<\delta(p)$ for some sufficiently small $\delta(p)>0$. Thus, if this is the case, $E(p,q,z)$ is real analytic on $(q,z)$, as required.

Since $L_\infty$ is compact, we can find a positive constant $\delta>0$ independent of $p$ such that $E(p,q,z)$ is real analytic on $(q,z)$ whenever $|q|, |z|<\delta$. This confirms the assumption (ii).
\end{itemize}
\end{proof}

Since the $f$-minimal submanifold $L_\infty$ corresponds to $u=0$, we have $\mathcal{M}(0)=0$, where $\mathcal{M}$ is the Euler-Lagrange functional for $\mathcal{E}$. Therefore, we can apply Theorem \ref{thm:LS}, and there exists an open neighbourhood $\mathcal{U}\subset \mathcal{U}_{R}$ of the origin in $C^{2,\alpha}$ satisfying  the \L{}ojasiewicz--Simon inequality \eqref{eq:LS}.

The Euler-Lagrange functional $\mathcal{M}$ for $\mathcal{E}: \mathcal{U}_R\to \R$ is explicitly written as follows. By the first variational formula \eqref{eq:fvv},  we have
\begin{align*}
\frac{d}{ds}\mathcal{E}(u+sv)\Big{|}_{s=0}&=-\int_{L_\infty}g(K_u, V)\,e^{nf\circ G_u}\, dv_{u}\\
&=-\int_{L_\infty}g(K_u, V)\,e^{nf\circ G_u}J_u\, dv_{\infty}
\end{align*}
where $K_u$ is the generalised mean curvature vector of $L_u$. Moreover, we see
\begin{align*}
V_p=\frac{d}{ds}\Big{|}_{s=0}G_{u+sv}(p)=\frac{d}{ds}\Big{|}_{s=0} {\rm exp}_p^{\widetilde{f}}\Big( (u(p)+sv(p))\widetilde{N}_p \Big)=v(p)\cdot d({\rm exp}_p^{\widetilde{f}})_{u(p)\widetilde{N}_p}(\widetilde{N}_p).
\end{align*}
Therefore, the Euler-Lagrange functional is explicitly given by
\begin{align}\label{eq:M}
&\mathcal{M}(u)=g(K_u, W_u)\cdot J_u\cdot e^{nf\circ G_u},\\
{\rm where}\quad &W_u(p):=d({\rm exp}_p^{\widetilde{f}})_{u(p)\widetilde{N}_p}(\widetilde{N}_p).  \label{def:W}
\end{align}

Now, we take a convergent subsequence $\{L_{t_i}\}_{i=1}^\infty$ as in Theorem  \ref{thm:main1}. 
We may assume that $L_{t_i}$ is a graph of some function $u_i\in C^\infty(L_\infty)$. Since $u_i\to 0$, there exists  $i_0$ such that  $u_{t_{i}}\in \mathcal{U}\subset \mathcal{U}_R$ for any $i\geq i_0$. We first aim to prove the following fact.
%\item $\in \mathcal{U}_\varepsilon$ with $\|u_{t_{i_0}}\|<\varepsilon/2$ and $\mathcal{E}(u_{t_{i_0}})-\mathcal{E}(0)<\varepsilon/2$.

\begin{proposition}\label{prop:keyLS}
By taking a sufficiently large $i_0$, $L_t=F_t(L)$ remains in the neighbourhood $\mathcal{U}$ for all $t\geq t_{i_0}$. More precisely, there exists a smooth function $u_t$ such that $u_t\in \mathcal{U}$ and $L_t=L_{u_t}$. 
\end{proposition}

To prove this, we set 
\[
T_0:={\rm sup}\{T\mid \textup{ $\exists u_t\in C^\infty(L_\infty)$ s.t. $u_t\in \mathcal{U}$ and $L_t=L_{u_t}$ for any $t\in [t_{i_0}, T]$}\}.
\]
Our aim is to show $T_0=\infty$. Note that we may assume that $u_t$  smoothly depends on $t$.  On the interval $[t_{i_0}, T_0]$, we prove some lemmas. 

\begin{lemma}
There exists a positive constant $C_1$ depending only on $\mathcal{U}_R$ and $M$ such that 
\begin{align}\label{eq:c1}
{\|\p_tu_t\|_{L^2}\leq C_1\|\mathcal{M}(u_t)\|_{L^2}}.
\end{align}
\end{lemma}
\begin{proof}
Let $\pi: L_{u_t}\to L_\infty$ be the projection. Since $F_t$ is an embedding and $L_{u_t}=F_t(L)$, it follows that $\psi_t:=\pi\circ F_t: L\to L_\infty$ is a diffeomorphism. Moreover, we have $G_{u_t}=F_t\circ \psi_t^{-1}$, and hence, 
\[
\p_tG_{u_t}=\p_t(F_t\circ \psi_t^{-1})=K_t+dF_t(\p_t\psi_t^{-1}).
\]
Note that all the vectors $\p_tG_{u_t}$, $K_t$ and $dF_t(\p_t\psi_t^{-1})$ are tangent to $\mu_{\rm can}^{-1}(0)$. Thus, we take the orthogonal projection $\perp_t: T_p\mu_{\rm can}^{-1}(0)\to T_p^\perp F_t(L)$ onto the normal subspace of $F_t(L)$ in $\mu_{\rm can}^{-1}(0)$, and then we have
$
(\p_tG_{u_t})^{\perp_t}=K_t.
$
On the other hand, by the definition of $G_{u_t}$, we have
$
\p_tG_{u_t}=\p_tu_t\cdot W_{u_t},
$
where $W_u$ is defined by \eqref{def:W}. Therefore, we obtain the following relation 
\[
K_t=\p_tu_t\cdot W_{u_t}^{\perp_t}.
\]
Inserting this to \eqref{eq:M}, we have
\begin{align}\label{eq:M2}
\mathcal{M}(u_t)=\p_tu_t\cdot |W_{u_t}^{\perp_t}|^2\cdot J_{u_t}\cdot e^{nf\circ G_{u_t}}.
\end{align}
Note that at the origin $u=0$, we have
\[
|W_{0}^\perp|=|\widetilde{N}^\perp|=|\widetilde{N}|=e^{-\widetilde f},\quad J_0=1
\]
since $\widetilde{N}$ is a unit vector with respect to $g_{\widetilde{f}}=e^{2\widetilde{f}}g$. 
Thus, by taking  $r$ sufficiently small, we may assume that there exist positive constants $c_1,c_2,c_3,c_4$ depending only on $\mathcal{U}_R$ such that 
\begin{align}\label{eq:ur}
c_1\leq |W_u^\perp|\leq c_2,\quad c_3\leq J_u\leq c_4,\quad \forall u\in \mathcal{U}_R.
\end{align}
Therefore, \eqref{eq:M2} shows that there exists a positive constant $C_1'$ depends only on $\mathcal{U}_R$ and $M$ such that 
\[
|\mathcal{M}(u_t)|^2\geq C_1'|\p_t u_t|^2 
\]
and this implies the desired inequality \eqref{eq:c1}.
\end{proof}

For any $t\in [t_{i_0}, T_0]$, we set 
\[
e(t):=\mathcal{E}(u_t)-\mathcal{E}(0)={\rm Vol}_f(L_{u_t})-{\rm Vol}_f(L_\infty)
\]
Note that, without loss of generality, we may assume $e(t)>0$  since $F_t(L)=L_{u_t}$ and $F_t$ is volume decreasing.
\begin{lemma}
There exists a positive constant $C_2$ depending only on $\mathcal{U}_R$ and $M$ such that 
\begin{align}\label{eq:c2}
{ \|\mathcal{M}(u_t)\|_{L^2}^2\leq-C_2e'(t)}
\end{align}
\end{lemma}
\begin{proof}
Note that we have
\[
e'(t)=\frac{d}{dt} \mathcal{E}(u_t)=\frac{d}{dt} {\rm Vol}_f(L_{u_t})=\frac{d}{dt} {\rm Vol}_f(F_t(L))=-\int_L|K_t|^2\,dv_{f}(t)
\]
Since there exists an isometry $\psi_t: (L,F_t^*g)\to (L_\infty, G_{u_t}^*g)$ (see the proof of previous lemma) such that $F_t=G_{u_t}\circ \psi_t$, it follows that 
\[
-e'(t)=\int_{L_\infty}|K_{u_t}|^2\ e^{nf\circ G_{u_t}}\,dv_{u_t}=\int_{L_\infty}|K_{u_t}|^2\ e^{nf\circ G_{u_t}}J_{u_t}\,dv_{\infty},
\]
where we used the relation $K_t=K_{u_t}\circ \psi_t$.
Therefore, using \eqref{eq:M} and \eqref{eq:ur}, we see
\begin{align*}
\|\mathcal{M}(u_t)\|_{L^2}^2
&\leq \int_{L_\infty} |K_{u_t}|^2|W_{u_t}|^2|J_{u_t}|^2 e^{2nf\circ G_{u_t}}\,dv_{\infty}\\
&\leq C_2\int_{L_\infty} |K_{u_t}|^2 J_{u_t}e^{2nf\circ G_{u_t}}\,dv_{\infty}=-C_2e'(t)
\end{align*}
for some positive constant $C_2$ depending only on $\mathcal{U}_R$ and $M$. This proves the lemma.
 \end{proof}

\begin{lemma}\label{lem:ub}
For any $k\geq 0$, there exists a positive constant $C_k$ independent of $t$ such that $\|u_t\|_{C^k}\leq C_k$ for any $t\in [t_{i_0},T_0]$.
\end{lemma}

\begin{proof}
By Proposition \ref{prop:bddB}, we have that 
$
|\nabla^kB_t|
$
is uniformly bounded for every $k\geq0$, where $B_t$ denotes the second fundamental form of the immersion $F_t: L\to M$, and $\nabla$ is the Levi-Civita connection of $F_t^*g$. It follows that the corresponding covariant derivatives of the second fundamental form of the hypersurface
$
\widetilde{F}_t: L\to\mu_{\rm can}^{-1}(0)
$
are also uniformly bounded. Since
$
F_t=G_{u_t}\circ\psi_t,
$
the same is true for the second fundamental form $B_{u_t}$ of the graph immersion $G_{u_t}$. Moreover, since $\mu_{\rm can}^{-1}(0)$ is compact, the metrics $g$ and $g_{\widetilde{f}}$ are uniformly equivalent, and hence so are the corresponding covariant derivatives. Therefore, the derivative of the second fundamental form of $G_{u_t}$ w.r.t.\  $g_{\widetilde{f}}$ is also uniformly bounded for every $k\geq0$. Thus, Lemma \ref{lem:ba} in Appendix \ref{app:norgra} then yields
$
\|u_t\|_{C^k}\leq C_k,
$
for every $k\geq2$.
\end{proof}

The following lemma may be standard, in fact, holds in greater generality.

\begin{lemma}\label{lem:l2}
Let $\Sigma$ be a smooth compact Riemannian manifold and $\alpha\in (0,1)$. For $k\geq 4$, we set 
\[
B_R:=\{u\in C^k(\Sigma)\mid \|u\|_{C^k}\leq R\}.
\]
Then, for any $\varepsilon>0$, there exists $\delta=\delta(\epsilon, R)$ such that if $u\in B_R$ and $\|u\|_{L^2}<\delta$, then we have $\|u\|_{C^{2,\alpha}}<\varepsilon$.
\end{lemma}

\begin{proof}
Suppose that the contrary were true. Then there exist
$\varepsilon_0>0$ and a sequence $\{u_j\}_{j=1}^{\infty}\subset B_R$ such that
$\|u_j\|_{L^2}\to 0$ as $j\to \infty$ and   $\|u_j\|_{C^{2,\alpha}}\geq \varepsilon_0$ for every $j$.

Since
$
\|u_j\|_{C^k}\leq R
$
for  $k\geq 4$, by using the Arzel\`a--Ascoli theorem, there exists a subsequence $\{u_{j_k}\}_{k=1}^\infty$ such that $u_{j_k}\in C^3(\Sigma)$ and 
$u_{j_k}\to u_\infty$ in $C^3(\Sigma)$. 
In particular, we have 
$
u_{j_k}\to u_\infty
$
in $L^2(\Sigma)$ and hence, $u_\infty=0$ since $\|u_j\|_{L^2}\to 0$ by assumption.

On the other hand, we have a continuous inclusion
$
C^3(\Sigma)\rightarrow C^{2,\alpha}(\Sigma)
$
since $L_\infty$ is compact and $0<\alpha<1$.
In particular, $u_{j_k}\to 0$ in $C^{2,\alpha}(\Sigma)$. However, this contradicts the assumption
$
\|u_j\|_{C^{2,\alpha}}\geq \varepsilon_0.
$
This proves the lemma.
\end{proof}

Now, we give a proof of Proposition \ref{prop:keyLS}.

\begin{proof}[Proof of Proposition \ref{prop:keyLS}]
By the \L{}ojasiewicz--Simon inequality \eqref{eq:LS}, we have 
\[
\frac{1}{\|\mathcal{M}(u_t)\|_{L^2}}\leq \frac{C_0}{e(t)^\beta}
\]
Combining this with \eqref{eq:c2}, we obtain
\[
\|\mathcal{M}(u_t)\|_{L^2}\leq -C_0C_2\frac{e'(t)}{e(t)^\beta}.
\]
Therefore, by \eqref{eq:c1}, we have
\begin{align*}
\|\p_tu_t\|_{L^2}\leq -C\frac{e'(t)}{e(t)^\beta},
\end{align*}
where we put $C:=C_0C_1C_2$.
Integrating this, we see
\begin{align}\label{eq:int}
\int_{t_{i_0}}^t \|\p_tu_t\|_{L^2}\,dt\leq -C\int_{t_{i_0}}^t \frac{e'(t)}{e(t)^\beta}\,dt 
=\frac{-C}{1-\beta}(e(t)^{1-\beta}-e(t_{i_0})^{1-\beta})\leq \frac{C}{1-\beta}e(t_{i_0})^{1-\beta}
\end{align}
since $\beta\in (\frac{1}{2},1)$ and $e(t)>0$. On the other hand, we have
\begin{align}\label{eq:Cau}
\|u(t)-u(t_{i_0})\|_{L^2}=\Big{\|}\int_{t_{i_0}}^t \p_tu_t\, dt\Big{\|}_{L^2}\leq \int_{t_{i_0}}^t \|\p_tu_t\|_{L^2}\,dt
\end{align}
by Minkowski integral inequality, and consequently, we obtain
\begin{align*}
\|u(t)-u(t_{i_0})\|_{L^2}\leq \frac{C}{1-\beta}e(t_{i_0})^{1-\beta}
\end{align*}
for any $t\in [t_{i_0}, T_0]$. In particular, we obtain
\[
\|u(t)\|_{L^2}\leq \|u(t_{i_0})\|_{L^2}+ \frac{C}{1-\beta}e(t_{i_0})^{1-\beta}
\]
Since $u_{t_i}\to 0$, $e(t_i)\to 0$ and the fact that both $\beta$ and $C=C_0C_1C_2$ are independent of $i_0$, for any small $\delta>0$, by taking a sufficiently large $i_0$, we may assume that 
\[
\|u(t_{i_0})\|_{L^2}<\frac{\delta}{2}\quad {\rm and}\quad \frac{C}{1-\beta}e(t_{i_0})^{1-\beta}<\frac{\delta}{2}
\]
so that we have $\|u(t)\|_{L^2}<\delta$ for all $t\in [t_{i_0},T_0]$.

As mentioned above,  by lemma \ref{lem:ub},  $u_t$ is bounded in $C^k$ for any $k\geq 2$. Then by Lemma \ref{lem:l2}, we see that for any $\varepsilon>0$, there exists $\delta>0$ such that if $\|u(t)\|_{L^2}<\delta$, then we have $\|u(t)\|_{C^{2,\alpha}}<\varepsilon$. As a consequence, for any $\varepsilon>0$, by taking a sufficiently large $i_0$, we have $\|u(t)\|_{C^{2,\alpha}}<\varepsilon$ for any $t\in  [t_{i_0}, T_0]$.

Now, we take a sufficiently small $\varepsilon>0$ so that $\mathcal{U}_\varepsilon\subset \mathcal{U}$, and suppose that $T_0<\infty$. Since we have $u_t\in \mathcal{U}_\varepsilon\subset \mathcal{U}$ for any $t\in  [t_{i_0}, T_0]$,  this implies that there exists $t>T_0$ such that $F_t(L)=L_{u_t}$ and $u_t\in \mathcal{U}$. This contradicts the definition of $T_0$. This proves the proposition.
\end{proof}

Theorem \ref{thm:main2} follows from Proposition \ref{prop:keyLS}:

\begin{proof}[Proof of Theorem \ref{thm:main2}]
By Proposition \ref{prop:keyLS}, there exists $u_t\in C^\infty(L_\infty)$ such that $L_t=F_t(L)=L_{u_t}$ for any sufficiently large $t$.  Moreover, by \eqref{eq:int}, we see
\[
\int_{t_{i_0}}^\infty \|\p_tu_t\|_{L^2}\,dt\leq \frac{C}{1-\beta}e(t_{i_0})^{1-\beta}<\infty.
\]
This implies that 
\[
\lim_{s\to \infty}\int_{s}^\infty \|\p_tu_t\|_{L^2}\,dt\to 0\quad (s\to \infty).
\]
Therefore, \eqref{eq:Cau} shows that we have 
\[
\|u(t)-u(s)\|_{L^2}\to 0,\quad (s,t\to \infty)
\]
Namely, $u(t)$ is a Cauchy sequence in the Hilbert space $L^2(L_\infty)$. Therefore, there exists a unique limit $u_\infty$. On the other hand, we have $u_{t_i}\to 0$, and hence $u_\infty=0$. Thus, $u_t$ converges to 0 in $L^2$. By Lemma \ref{lem:l2}, it follows that $u_t\to 0$ in $C^{2,\alpha}$ and moreover, we see $u_t\to 0$ in $C^\infty$ by using a standard argument since $\|u_t\|_{C^k}$ is bounded for any $k\geq 2$ (Lemma \ref{lem:ub}). This completes the proof.
\end{proof}

\appendix

\section{Orbifolds and Group Actions}\label{app-orbifold}

In this section, we summarize some basic definitions and results relating to orbifolds that are
used in this paper. For further details, we refer the reader to \cite[\S 1]{ALR}, \cite[\S 2]{KL} and \cite[Ch.13]{Thurston}.

\subsection{Basic notions of orbifolds}

Let $X$ be a topological space. An ($n$-dimensional) {\it orbifold chart} on $X$ is a triple $(\widetilde{U}, G, \varphi)$ consisting of a connected open subset  $\widetilde{U}\subset \R^{n}$, a finite group $G$ and a $G$-invariant continuous  map $\varphi: \widetilde{U}\to X$  such that $G$ acts on $\widetilde{U}$ smoothly and effectively (i.e.\ there is an injective homomorphism $G\to {\rm Diff}(\widetilde{U})$) and the $G$-invariant map $\varphi$ induces a homeomorphism $\overline{\varphi}: \widetilde{U}/G\to U$, where we set $U:=\varphi(\widetilde{U})$. 
A map $\lambda: \widetilde V \to \widetilde U$ between two orbifold charts $(\widetilde V,H,\psi)$ and $(\widetilde U, G, \varphi)$ is an \textit{embedding} if $\lambda$ is a smooth topological embedding such that 
$\varphi\circ \lambda=\psi$. Such an embedding induces an injective homomorphism $\rho_{\lambda}: H\to G$ so that $\lambda(h \widetilde{x})=\rho_{\lambda}(h)\lambda(\widetilde{x})$ for any $h\in H$ and $\widetilde{x}\in \widetilde{V}$.

An {\it orbifold atlas} of $X$ is a covering of $X$ by $n$-dimensional orbifold charts $\{(\widetilde{U}_{\alpha},G_{\alpha},\varphi_{\alpha})\}_{\alpha}$ satisfying  the following condition: If $x\in U_{\alpha_{1}}\cap U_{\alpha_{2}}\neq \emptyset$,   there 
are an orbifold chart $(\widetilde{V}, H, \psi)$ around $x$ and embeddings ${\lambda}_{i}: (\widetilde{V}, H, \psi)\to (\widetilde{U}_{\alpha_{i}},G_{\alpha_{i}},\varphi_{\alpha_{i}})$  for $i=1,2$. 

% Here, ${\lambda}:  (\widetilde{V}, H, \psi)\to (\mathcal{O},G,\varphi)$ is an {\it embedding} if $\lambda: \widetilde{V}\to \mathcal{O}$ is a smooth embedding such that $\varphi\circ \lambda=\psi$. 
% Note that the embedding $\lambda$ induces an injective homomorphism $\rho_{\lambda}: H\to G$ so that $\lambda(h \widetilde{x})=\rho_{\lambda}(h)\lambda(\widetilde{x})$ for any $h\in H$ and $\widetilde{x}\in \widetilde{V}$.
For the embeddings $\lambda_{i}: \widetilde{V}\to \widetilde{U}_{\alpha_{i}}$ ($i=1,2$), we set $\lambda_{12}:=\lambda_{2}\circ \lambda_{1}^{-1}: \lambda_{1}(\widetilde{V})\to\lambda_{2}(\widetilde{V})$, where we regard each $\lambda_{i}: \widetilde{V}\to \lambda_{i}(\widetilde{V})$ as a diffeomorphism. We call the map $\lambda_{12}$ a {\it transition map}. 
An atlas $\mathcal{A}$ is a \textit{refinement} of an atlas $\mathcal{B}$ when every chart of $\mathcal{A}$ admits an embedding into a chart of $\mathcal{B}$. Two atlases are \textit{equivalent} if they have a common refinement.

An {\it $n$-dimensional (smooth) orbifold} $\mathcal{O}$ is a Hausdorff and paracompact topological space $X$ equipped with an equivalence class $[\mathcal{A}]$ of $n$-dimensional orbifold atlases.
We denote the underlying topological space $X$ by $|\mathcal{O}|$. 
% The notion of an {\it orbifold with boundary} can be defined in a similar way (see \cite{KL}). However, throughout this paper we assume that the orbifold $\mathcal{O}$ has no boundary. 

An orbifold $\mathcal{O} = (X,[\mathcal{A}])$
is called {\it connected} (resp. {\it compact}) if $X=|\mathcal{O}|$ is connected (resp. compact). We say an atlas $\{(\mathcal{O}_{\alpha},G_{\alpha},\varphi_{\alpha})\}_{\alpha\in A}$ is {\it oriented} if each chart $\mathcal{O}_{\alpha}$ is oriented so that the action of $G_{\alpha}$ on $\mathcal{O}_{\alpha}$ preserves the orientation and every 
embedding between charts also preserves the orientation. An orbifold $\mathcal{O}$ is {\it oriented} if there is an oriented atlas $\mathcal{B} \in [\mathcal{A}]$.

For each $x\in |\mathcal{O}|$ and a chart $(\widetilde{U}, G, \varphi)$ around $x$, the {\it local group} $\Gamma_{x}$ is  the conjugacy class of the stabiliser subgroup $G_{\widetilde{x}}:=\{g\in G\mid g \widetilde{x}=\widetilde{x}\}$, where $\widetilde{x}$ is an element in $\widetilde{U}$ so that $\varphi(\widetilde{x})=x$. This definition of $\Gamma_{x}$ depends only on $x$, namely, it  does not depend on the choice of the orbifold chart around $x$ or the choice of $\widetilde{x}$.
The {\it regular part} $\mathcal{O}^{{\rm reg}}$ of $\mathcal{O}$ is defined by $\mathcal{O}^{{\rm reg}}:=\{x\in |\mathcal{O}|\mid \Gamma_{x}=\{e\}\}$. It turns out that $\mathcal{O}^{{\rm reg}}$ is an open dense subset of $\mathcal{O}$ and 
 has the structure of a smooth manifold. Moreover, $\mathcal{O}^{{\rm reg}}$ is connected whenever $\mathcal{O}$ is connected. 
The \textit{singular part} of $\mathcal{O}$ is then defined by
$\mathcal{O}^{{\rm sing}}:=|\mathcal{O}|\setminus \mathcal{O}^{{\rm reg}}$. A point $x\in |\mathcal{O}|$ will be called a {\it regular point} (resp. {\it singular point}) if $x\in \mathcal{O}^{{\rm reg}}$ (resp. $x\in \mathcal{O}^{{\rm sing}}$).

A continuous map $f: |\mathcal{O}_{1}|\to |\mathcal{O}_{2}|$ between two orbifolds is said to be {\it smooth} if for any $x\in |\mathcal{O}_{1}|$, there exist orbifold charts $(\widetilde{U},G,\varphi)$ around $x$ and $(\widetilde{V}, H, \psi)$ around $f(x)$  such that $f(U)\subset V$,  together with a homomorphism $\rho: G\to H$ and a $\rho$-equivariant smooth map $\widetilde{f}:  \widetilde{U}\to \widetilde{V}$ with $\psi\circ \widetilde {f}=f\circ \varphi$. 
We will denote such a map by $f: \mathcal{O}_1 \to \mathcal{O}_2$.
If a smooth map $f: \mathcal{O}_{1}\to \mathcal{O}_{2}$ is bijective and the inverse map $f^{-1}: \mathcal{O}_{2}\to \mathcal{O}_{1}$ is also a smooth map, then $f$ is said to be a {\it diffeomorphism}.

For an orbifold chart $(\widetilde{U},G,\varphi)$, the finite group $G$ naturally acts on the tangent bundle $T\widetilde{U}$ over $\widetilde{U}$ by the map $(\widetilde{x}, v)\mapsto (g\widetilde{x}, dg_{\widetilde{x}}(v))$, where $v\in T_{\widetilde{x}}\widetilde{U}$. 
We therefore obtain a natural projection $p: T\widetilde{U}/G\to U$,  where
the fibre $p^{-1}(x)$ is homeomorphic to $T_{\widetilde{x}}\widetilde{U}/G_{\widetilde{x}}$ (see \cite[\S 1.3]{ALR}) for each $x\in U$ (Note that the fibre is no longer a vector space). For a given orbifold atlas $\{(\widetilde{U}_{\alpha},G_{\alpha},\varphi_{\alpha})\}_{\alpha\in A}$ of $\mathcal{O}$, we define the {\it orbifold tangent bundle} $T\mathcal{O}$  by 
\[
T\mathcal{O}:=\Big(\bigsqcup_{\alpha\in A} T\widetilde{U}_{\alpha}/G_{\alpha}\Big)\Big{/}\sim,
\]
where the equivalence relation
$\sim$ is defined as follows: $[(\widetilde{x},v)]\in T\widetilde{U}_{\alpha_{1}}/G_{\alpha_{1}}$ is equivalent to $[(\widetilde{y},w)]\in T\widetilde{U}_{\alpha_{2}}/G_{\alpha_{2}}$ if $\varphi_{\alpha_{1}}(\widetilde{x})=\varphi_{\alpha_{2}}(\widetilde{y})=x\in U_{\alpha_{1}}\cap U_{\alpha_{2}}$ and for the orbifold atlas $(\widetilde{V}, H, \psi)$ around $x$ with embeddings ${\lambda}_{i}: \widetilde{V}\to \widetilde{U}_{\alpha_{i}}$, it holds that $(d\lambda_{12})_{\widetilde{x}}(v)=w$, where $\lambda_{12}$ is the transition map. Note that we have a natural projection $\pi: T\mathcal{O}\to \mathcal{O}$ and each fibre $\pi^{-1}(x)$ is homeomorphic to $T_{\widetilde{x}}\widetilde{U}_{\alpha}/G_{\alpha}$ where $\varphi_{\alpha}(\widetilde{x})=x$.
 The orbifold tangent bundle  $T\mathcal{O}$ becomes a $2n$-dimensional orbifold with an orbifold atlas $\{(T\widetilde{U}_{\alpha},G_{\alpha},\Phi_{\alpha})\}_{\alpha\in A}$, where $T\widetilde{U}_{\alpha}\simeq \widetilde{U}_{\alpha}\times \R^{n}$ and $\Phi_{\alpha}: T\widetilde{U}_{\alpha}\to T\widetilde{U}_{\alpha}/G_{\alpha}$ is the natural projection.
 Analogously to the orbifold tangent bundle,
we can construct the orbifold cotangent bundle $T^{*}\mathcal{O}$ and orbifold tensor bundles $\bigotimes^{k} T\mathcal{O}$ and $\bigotimes^{k} T^{*}\mathcal{O}$.
For example, the equivalence relation
for $T^{*}\mathcal{O}$ should be defined as follows; $[(\widetilde{x},v)]\in T^{*}\widetilde{U}_{\alpha_{1}}/G_{\alpha_{1}}$ is equivalent to 
 $[(\widetilde{y},w)]\in T^{*}\widetilde{U}_{\alpha_{2}}/G_{\alpha_{2}}$ if and only if $\varphi_{\alpha_{1}}(\widetilde{x})=\varphi_{\alpha_{2}}(\widetilde{y})=x\in U_{\alpha_{1}}\cap U_{\alpha_{2}}$ and $(\lambda_{12})_{\widetilde{y}}^{*}w=v$.
 
A {\it $1$-form} on $\mathcal{O}$ is  a smooth section of $T^{*}\mathcal{O}$, that is, a smooth map $s: \mathcal{O}\to T^{*}\mathcal{O}$ 
such that $\pi\circ s=id_{\mathcal{O}}$, where $\pi: T^{*}\mathcal{O}\to \mathcal{O}$ is the natural projection. By definition of $T^{*}\mathcal{O}$, the section $s: \mathcal{O}\to T^{*}\mathcal{O}$ yields a $G_{\alpha}$-equivariant smooth section $s_{\alpha}: \widetilde{U}_{\alpha}\to T^{*}\widetilde{U}_{\alpha}$ for each $\alpha\in A$, which satisfies the following commutative diagram:
\[
\xymatrix@C=25pt@R=25pt{
\widetilde{U}_{\alpha}\  \ar[r]^{s_{\alpha}}   \ar[d] &\ T^{*}\widetilde{U}_{\alpha} \ar[d]\\
\widetilde{U}_{\alpha}/G_{\alpha}\simeq U_{\alpha}\  \ar[r]^-{s} &\ T^{*}\mathcal{O}
}
\]
Thus, we obtain a family of smooth $1$-forms $\{s_{\alpha}\}_{\alpha\in A}$. Notice that if $U_{\alpha_{1}}\cap U_{\alpha_{2}}\neq \emptyset$,  the transition map $\lambda_{12}:\lambda_{1}(\widetilde{V})\to \lambda_{2}(\widetilde{V})$ satisfies 
\begin{align}\label{eq:treq}
\lambda_{12}^{*}s_{\alpha_{2}}|_{\lambda_{2}(\widetilde{V})}=s_{\alpha_{1}}|_{\lambda_{1}(\widetilde{V})}.
\end{align}
Conversely, if a family of smooth equivariant $1$-forms $\{s_{\alpha}\}_{\alpha\in A}$ defined on each chart satisfies the relation \eqref{eq:treq} for any transition map, then we obtain a smooth section $s:\mathcal{O}\to T^{*}\mathcal{O}$ by gluing the $1$-forms.
A $1$-form $s$ on $\mathcal{O}$ is therefore equivalent to
a family of equivariant $1$-forms $\{s_{\alpha}\}_{\alpha\in A}$ satisfying the relation \eqref{eq:treq} for any transition map.

Analogously, a {\it $(0,k)$-tensor} $\tau$ on $\mathcal{O}$ is defined by a smooth section $\tau: \mathcal{O}\to \bigotimes^{k}T^{*}\mathcal{O}$. If $\tau_{x}$ is symmetric (resp. anti-symmetric) for each $x\in \mathcal{O}$, then the $(0,k)$-tensor $\tau$ is called a {\it symmetric tensor} (resp. a {\it $k$-form}). For a $k$-form $\eta$, we define the exterior derivative $d\eta$ as a $(k+1)$-form obtained by gluing of the family of $(k+1)$-forms $\{d\eta_{\alpha}\}_{\alpha\in A}$, where $\{\eta_{\alpha}\}_{\alpha\in A}$ is the family of $k$-forms corresponding to $\eta$.
A $(k+1)$-form $\omega$ is said to be {\it closed} (resp. {\it exact}) if $d\omega_{\alpha}=0$ for any $\alpha\in A$ (resp. if there exists a $k$-form $\eta$ on $\mathcal{O}$ such that $\omega_{\alpha}=d\eta_{\alpha}$ for any $\alpha \in A$).

A {\it Riemannian metric} on $\mathcal{O}$ is a symmetric $(0,2)$-tensor $g$ so that $g_{x}$ is non-degenerate for each $x\in \mathcal{O}$. A {\it symplectic form} on $\mathcal{O}$ is a closed $2$-form $\omega$ so that $\omega_{x}$ is non-degenerate for each $x\in \mathcal{O}$. 
By the prior discussion, a Riemannian metric $g$ (resp. a symplectic form $\omega$) 
may be regarded as a family of $G_{\alpha}$-invariant Riemannian metrics $\{g_{\alpha}\}_{\alpha\in A}$ (resp. $G_{\alpha}$-invariant symplectic forms $\{\omega_{\alpha}\}_{\alpha\in A}$) defined on each chart, which 
satisfy the compatibility condition with respect to the chart embeddings (i.e.\ they are preserved under pullback by transition maps).
An orbifold $\mathcal{O}$ equipped with a Riemannian metric $g$ (resp. symplectic form $\omega$) is called a {\it Riemannian orbifold} (resp. {\it symplectic orbifold}).

For a smooth function $f: \mathcal{O}\to \R$ on a Riemannian orbifold $(\mathcal{O},g)$, we define a {\it weighted Riemannian metric} by $g_{f}:=e^{2f}g$. We say $f$ a {\it weight function} and  the triple $(\mathcal{O},g, f)$ will be called a {\it weighted Riemannian orbifold}.

Let $(\mathcal{O},g, f)$ be a weighted Riemannian orbifold. For a piecewise smooth curve $\gamma: [a,b]\to \mathcal{O}$, we define the (weighted) length $L_f(\gamma)$ of $\gamma$ as follows: We divide the interval into $a=a_0<a_1<\cdots <a_N=b$ so that the image of each restricted map $\gamma_i:=\gamma|_{[a_i,a_{i+1}]}$ is contained in an orbifold chart $(\widetilde{U}_i,G_i,\varphi_i)$. We take a lift $\widetilde{\gamma}_i: [a_i,a_{i+1}]\to \widetilde{U}_i$ of $\gamma_i$,  and then, we define $L_f(\gamma):=\sum_{i=1}^N L_f(\widetilde{\gamma}_i)$, where  $L_f(\widetilde{\gamma}_i)$ is the (weighted) length with respect to $\varphi_i^*g_f$. This definition is independent of the choice of subdivisions of $I$ and orbifold charts.

A {\it (weighted) distance function} $\widehat{d}: \mathcal{O}\times \mathcal{O}\to \R$ on a weighted Riemannian orbifold $(\mathcal{O},g, f)$ is defined by 
\[
\widehat{d}(x,y):={\rm inf}\{L_f(\gamma)\mid \textup{$\gamma: [a,b]\to \mathcal{O}$ is a piecewise smooth map with $\gamma(a)=x,\ \gamma(b)=y$}\}.
\]
We remark that, although the definition of $\widehat{d}$ is similar to that of smooth Riemannian manifolds, some properties of distance function cannot be extended to orbifolds. See Remark \ref{rem:derd} for example.

\subsection{Orbifolds obtained by a locally free action}\label{subs:lco}
A typical example of an orbifold is one obtained by
the quotient of a locally free proper effective $G$-action on a smooth manifold. In this subsection, we describe the orbifold structure on the quotient space relevant to our situation.

First, we recall some basic notions and facts on group action. Suppose that a compact connected Lie group $G$ acts on a smooth manifold $N$.
An action of $G$ on $N$ is said to be \textit{locally free}
if the stabiliser subgroup $G_{p}$ is discrete for any $p\in N$, and \textit{free} if $G_{p}$ is trivial for every $p\in N$.
If $\cap_{p\in N}G_{p}=\{e\}$, the action is said to be {\it effective}.

We denote the $G$-orbit through $p\in N$ by $G\cdot p$. Two orbits $G\cdot p$ and $G\cdot q$ are 
\textit{equivalent} if $G_{p}$ is conjugate to $G_{q}$, where $G_{p}$ is the stabiliser subgroup at $p\in N$, and the equivalence class $[G\cdot p]$ is called the {\it orbit type}. This equivalence relation defines a partial ordering, 
namely $[G\cdot p]\leq [G\cdot q]$ if $G_{q}$ is conjugate to a subgroup of $G_{p}$ in $G$. 
By the principal orbit theorem \cite[Ch.~IV, Thm. 3.1]{Bredon1972},
there exists  an equivalence class maximal with respect to the partial order,
called the {\it principal orbit type}. 
Furthermore, principal orbits have the largest dimension among $G$-orbits, and
any two principal orbits are diffeomorphic to each other. The codimension
of the principal orbit in $N$ is said to be {\it cohomogeneity} of the $G$-action.
A $G$-orbit $G\cdot p$ is called {\it exceptional} if $G\cdot p$ is not principal, but the dimension of $G\cdot p$ is the same as the dimension of the principal orbit.  A $G$-orbit $G\cdot p$ is said to be {\it singular} if it is neither principal nor exceptional (hence, the dimension of any singular orbit is smaller than the dimension of the principal orbits)
We denote the set of principal (resp. exceptional, singular) orbits by $N^{{\rm pri}}$ (resp. $N^{{\rm ex}}, N^{{\rm sing}}$). Also, we set $N^{{\rm reg}}:=N^{{\rm pri}}\cup N^{{\rm ex}}$, and we say an orbit in $N^{{\rm reg}}$  is
a {\it regular orbit}. The principal orbit theorem also states that
$N^{{\rm pri}}$ is an open dense subset in $N$.

Assume now that $N$ is equipped with a complete Riemannian metric $g$, and $G$ is a connected closed subgroup in the isometry group ${\rm Isom}(N,g)$. Then the action of $G$ on $N$ is proper, and the quotient space $N/G$ becomes a paracompact Hausdorff space. We denote the natural projection by $\pi: N\to N/G$. 

Suppose furthermore that $G$ acts on $N$ locally freely and effectively. Then, $N/G$ admits an orbifold structure $\mathcal{O}$.

\begin{proposition}\label{prop-riemorbifoldquotient}
    Let $(N,g)$ be a smooth Riemannian manifold. Consider a compact, connected Lie subgroup $G \leq \emph{Isom}(N,g)$ such that the action of $G$ is locally free and of cohomogeneity $k$.

    Then $N/G$ admits a $k$-dimensional orbifold structure $\mathcal{O} := (N/G, [\mathcal{A}])$ and an orbifold Riemannian metric $\overline g$ for which the quotient is a Riemannian submersion.
\end{proposition}
\begin{proof}
We define an orbifold chart around each point $\overline{p}\in N/G$ as follows. We take an arbitrary point $p\in \pi^{-1}(\overline{p})$ and define the {\it geodesic slice} $\Sigma_{p}$ by $\Sigma_{p}:={\rm exp}_{p}(B^\perp_p(\epsilon))$, where $B_p^\perp(\epsilon):=\{v\in T_{p}^{\perp}(G\cdot p)\mid \|v\|<\epsilon\}$ is the $\epsilon$-ball in the normal space $T^{\perp}_{p}(G\cdot p)$ of the $G$-orbit $G\cdot p$  for some small $\epsilon>0$, and ${\rm exp}_{p}$ is the Riemannian exponential map at $p$. The isotropy subgroup $G_{p}$ (which is a finite subgroup by assumption) acts on $B_p^\perp(\epsilon)$ via the {\it slice representation} $\tau: G_{p}\times T^{\perp}_{p}(G\cdot p)\to T^{\perp}_{p}(G\cdot p)$,  $\tau(h,v):=dh_{p}(v)$.  
By taking the quotient group $\widehat{G}_p:=G_p/K_p$, where $K_p=\{g\in G_p\mid \tau(g)v=v,\ \forall v\in T_p^\perp(G\cdot p)\}$ is the kernel of $\tau$, the group $\widehat{G}_p$ acts on $ T^{\perp}_{p}(G\cdot p)$ effectively via the induced representation. Moreover, by the slice theorem, 
$(B_p^\perp(\epsilon), \widehat{G}_{p}, \pi\circ {\rm exp}_{p}|_{B_p^\perp(\epsilon)})$
becomes an orbifold chart around $\overline{p}\in N/G$.  

Note that $G\cdot p$ is a principal orbit if and only if the slice representation is trivial (see \cite[Theorem 2.1.3]{BCO}), and hence, if $p$ is an principal element, then $(B_{p}^\perp(\epsilon), \pi\circ {\rm exp}_{p}|_{B_p^\perp(\epsilon)})$ is exactly a chart as a manifold. In particular, the open dense subset $N^{{\rm pri}}/G$ coincides with the regular part of $\mathcal{O}$.

Finally, to construct the orbifold metric, note that the $G$-invariant metric $g$ on $N$ naturally lifts to
a smooth metric $\widetilde{g}$ on $B_p^\perp(\epsilon)$
by using the restricted map  ${\rm exp}_p|_{B_p^\perp (\epsilon)}: B_p^\perp (\epsilon)\to N$
and defining $\widetilde{g}:={\rm exp}_p^*g$. Since  $G_p\subset G\subset {\rm Isom}(N,g)$,  we have ${\rm exp}_p\circ \tau(h)=h\circ {\rm exp}_p$ for any $h\in G_p$ and this implies that $\widetilde{g}$ is a $G_p$-invariant Riemannian metric on $B_p(\epsilon)$; by definition it is clearly a Riemannian submersion. This collection of $G_p$-invariant metrics yields an orbifold Riemannian metric $\overline{g}$ on $\mathcal{O}$.
\end{proof}

\begin{remark}\label{rem:trivialprincipalcase}
In our applications in this work, the principal orbits are assumed to have trivial stabiliser, and the non-principal orbits then correspond precisely to cone points of the two-dimensional orbifold quotient. In this case, the open dense subset $N^{{\rm pri}}/G$ coincides with the regular part $\mathcal{O}^{{\rm reg}}$ of $\mathcal{O}$.
\end{remark}

 Finally, let $f: N\to \R$ be a $G$-invariant smooth function. Then, $f$ yields a well-defined continuous function $\overline{f}: N/G\to \R$ defined by the relation $f=\overline{f}\circ \pi$. Moreover, it turns out that $\overline{f}:\mathcal{O}\to \R$ is a  smooth function in the sense of orbifold smooth map.  Indeed, the local lift $\widetilde{f}: B_p(\epsilon)\to \R$ of $\overline{f}$ in each orbifold chart
$B_p(\epsilon)$ is given by $\widetilde{f}=f\circ {\rm exp}_{p}$, and this is a $G_{p}$-invariant smooth function on $B_{p}(\epsilon)$. 

In particular, if $(N,g)$ is a Riemannian manifold and
there is a $G$-invariant function $f: N\to \R$, then we can define  a weighted Riemannian structure $(\overline{g}, {\overline{f}})$ on the orbifold $\mathcal{O}$. In this paper, we shall 
use the following concrete example.

\begin{example}[Hsiang-Lawson metric (cf.\ \cite{HL})]
{\rm
Suppose that $G$ acts locally freely and isometrically on $(N,g)$.  Note that if $G\cdot p$ is an exceptional orbit, by taking an appropriate subgroup $G_{p'}$ which corresponds to the principal orbit type and satisfies $G_{p'}\subset G_{p}$, we obtain an $m$-fold covering $G/G_{p'}\to G\cdot p=G/G_{p}$, where $m:=\# G_{p}/G_{p'}$.  With this in mind, we define the {\it volume function} $V: N\to \R$ by
\begin{align}\label{def:vol}
V(p):=
\begin{cases}
{\rm Vol}(G\cdot p) & \textup{if $G\cdot p$ is principal},\\
m{\rm Vol}(G\cdot p) & \textup{if $G\cdot p$ is exceptional},
\end{cases}
\end{align}
where ${\rm Vol}(G\cdot p)$ is the volume of the orbit $G\cdot p$, and $m := [G_p : G_{p'}]$.
Then it turns out that  $V$ is a $G$-invariant smooth function (cf.\ \cite[Proposition 1]{Pacini}). Thus, $V$ yields a smooth function $\overline{V}: \mathcal{O}\to \R$ on the quotient orbifold. 

We  put
\[
\overline{f}:=\log \overline{V}^{1/k},
\]
where $k$ is the cohomogeneity of the $G$-action on $N$. Then, we define the weighted Riemannian metric on $\mathcal{O}$ by
\[
\overline{g}_{\overline{f}}:=e^{2\overline{f}}\overline{g}=\overline{V}^{2/k}\overline{g}.
\]
We refer to the weighted metric $\overline{g}_{\overline{f}}$ as  the {\it Hsiang-Lawson metric}.  
}
\end{example}

\section{ Proof of Proposition \ref{prop:distcomp}}\label{app:distcomp} 

In this appendix, we give a proof of Proposition  \ref{prop:distcomp}, which we restate here:

\begin{proposition}[Proposition  \ref{prop:distcomp}]\label{prop:distcomp2}
Suppose that  $L_\psi(T_{\max}):=\lim_{t\to T_{\max}}L_\psi(t)>0$. Then, by taking a sufficiently large $\mathcal{K}$, we have ${\rm sup}_{t\in [0,T_{\max})}R(t)<\infty$.
\end{proposition}

Before giving the proof, we prove a lemma.

\begin{lemma}\label{lem:c0}
Let $(\mathcal{O},\widehat{g})$ be a compact Riemannian orbifold with finitely many cone points. Then there exists a positive constant $c_0 $ such that the following property holds:  If $p,q\in \mathcal{O}^{\rm reg}$ and $\widehat{d}(p,q)<c_0$, then  there is an orbifold chart $(\widetilde{U}, G, \varphi)$ such that $p,q\in \varphi(\widetilde{U})$ and $\widetilde{U}$ is strongly convex with respect to the induced metric $\varphi^*\widehat{g}$.
\end{lemma}

\begin{proof}
Let $p_i$ ($i=1,\ldots, m$) be cone points and $(\widetilde{U}_i, G_i, \varphi_i)$ be an orbifold chart around $p_i$. Note that we may assume that $\varphi_i(0)=p_i$. Then, there exists a positive constant $\delta_i > 0$ such that the open geodesic ball $\widetilde{B}_i(\delta_i) \subset \widetilde{U}_i$  centred at the origin $0$ is strongly convex with respect to $\varphi_i^*\widehat{g}$.
We set $B_i:=\varphi_i(\widetilde{B}_i(\delta_i))$. We may assume that  $\widehat{d}(B_i,B_j)>0$ for any $i\neq j$ and we put 
$c_1:=\min_{i\neq j}\widehat{d}(B_i,B_j)>0$.

Let $S := \mathcal{O} \setminus \bigsqcup_{i = 1}^m B_i$. Note that $S\subset \mathcal{O}^{\rm reg}$. Moreover, the closed subset  $S$ is a compact set since $X=|\mathcal{O}|$ is compact.
Let $r(p)$ be the convexity radius at $p\in \mathcal{O}^{\rm reg}$, where we regard $\mathcal{O}^{\rm reg}$ is a smooth Riemannian manifold. It is known that the convexity radius $r(p)$ is a positive continuous function on a smooth Riemannian manifold (see \cite[Ch.IV, Theorem 5.3]{Sakai}), and hence, using the compactness of $S\subset \mathcal{O}^{\rm reg}$, there exists a positive constant $c_2> 0$ such that the open geodesic ball $B_p(c_2)\subset  \mathcal{O}^{\rm reg}$ around $p$ is strongly convex for every $p\in S$. Note that $B_p(c_2)$ can be regarded as a chart of $\mathcal{O}^{\rm reg}$.

Now, we define $c_0 := \min \{ c_1,c_2 \}$, and  we take any $p, q \in \mathcal{O}^{\mathrm{reg}}$ with $\widehat{d}(p, q) < c_0$.  We divide the case into the following possibilities:
\begin{itemize}
\item At least one of $p$ and $q$ belongs to $S$: We may assume $p\in S$. Then $B_p(c_0)\subset B_p(c_2)$ yields an orbifold chart $(\widetilde{U}_p, \{e\}, \varphi_p)$ around $p$ such that $\varphi_p: \widetilde{U}_p\to B_p(c_0)$ is a homeomorphism, and this is the desired chart. 
\item $p,q\in B_i$ for some $i$: In this case, $(\widetilde{B}_i(\delta_i), G_i, \varphi_i|_{\widetilde{B}_i(\delta_i)})$ is the desired orbifold chart. 
\item $p\in B_i$ and $q\in B_j$ for $i\neq j$: This case is excluded since $\widehat{d}(p, q) < c_0<c_1$.
\end{itemize}
Thus, we obtain the lemma.
\end{proof}

We now prove Proposition \ref{prop:distcomp2}. We use the notation of subsection \ref{ssec:dist}, in particular recall that
\begin{equation*}
R(t) := \sup_{x,y\in S^{1};\, x \neq y} R(x,y,t) := \sup_{x,y\in S^{1};\, x \neq y} \left[ \frac{L_{\psi}(t) e ^{-\mathcal{K} t} }{\pi {d}_{\psi}(x,y,t)} \sin \left( \frac{\pi l_{\psi}(x,y,t)}{L_{\psi}(t)} \right) \right].
\end{equation*}

We aim to show that ${\rm sup}_{t\in [0,T_{\max})}R(t)<\infty$. We shall show this by contradiction. Thus, we suppose that ${\rm sup}_{t\in [0,T_{\max})}R(t)=\infty$.

\textbf{Step 1.}  First, note that for each $t\in [0,T_{\max})$, we have $R(t)<\infty$. 
Indeed, it is known that  for an embedded curve $c: S^1\to M$ into a smooth Riemannian manifold $M$, we have that 
\[\lim_{(x,y)\to (x_0,x_0)} \frac{l_c(x,y)}{d_M(c(x),c(y))}=1\] for any $x_0\in S^1$, where $l_c$ is the intrinsic distance between two points on $S^1$ with respect to the induced metric and $d_M$ is the distance function on $M$. In particular, we have
\begin{equation*}
 \lim_{(x,y)\to (x_0,x_0)}R(x,y,t)=e^{-\mathcal{K}t}<\infty.
\end{equation*}
This implies that the function $R(x,y,t)$ continuously extends to $S^1\times S^1$ and for $ t < T_{\rm max}$,
\[R(t)={\rm max}_{(x,y)\in S^1\times S^1}R(x,y,t)<\infty.\]

\textbf{Step 2.} We now define a comparison function $Z_N$. Recall that $\widehat{K}_{\mathcal{O}}=\sup_{\mathcal{O}}|\widehat{K}|$, where $\widehat{K}$ is the Gaussian curvature of $(\mathcal{O},\widehat{g})$. We define  constants $d_0$, $N$ so that
\begin{equation*}
d_0 := \min \Big\{\frac{1}{2 \sqrt{\widehat{K}}_{\mathcal{O}}}, c_0 \Big{\}}, \quad N > \max \Big\{ \frac{L_{\psi}(0)}{\pi d_0}, R(0), 1 \Big\},
\end{equation*}
where $c_0$ is the positive constant given in Lemma \ref{lem:c0}. We then define $Z_N: S^1\times S^1\times [0,T_{\max})\to \R$:
\begin{equation*}
Z_N (x,y,t) := N d_{\psi}(x, y, t) - \frac{L_{\psi}(t)}{\pi} \sin \left( \frac{\pi l_{\psi}(x,y,t)}{L_{\psi}(t)} \right) e^{- \mathcal{K} t}.
\end{equation*}
Note that if $x\neq y$, then $d_\psi(x,y,t)\neq0$ since $\gamma_t$ is embedding, and if this is the case,  we have $Z_N(x,y,t)=d_{\psi}(x,y,t)\{N-R(x,y,t)\}$.

Since we assume that  ${\rm sup}_{t\in [0,T_{\max})}R(t)=\infty$, there exists a first time $\overline{t}\in [0,T_{\max})$ such that $R(\bar{t})=N$. If this is the case, it must hold that 
\begin{itemize}
\item[(i)] $Z_N(x,y,t) > 0$ for all $x, y \in {S}^1$ with $x\neq y$ and for all $0 < t < \bar{t}$;
\item[(ii)] $Z_N(x,y, \bar{t}) \geq 0$ for all $x, y \in {S}^1$;
\item[(iii)] $Z_N(\bar{x}, \bar{y}, \bar{t}) = 0$ for some $\bar{x}, \bar{y} \in {S}^1$ with $\bar{x}\neq \bar{y}$;
\item[(iv)] $d_\psi(\bar{x},\bar{y},\bar{t})<d_0$.
\end{itemize}
Here, in the item (iii), the fact $\bar{x} \neq \bar{y}$ holds because
\begin{equation*}
\lim_{y \to x} \frac{L_{\psi}(t)}{\pi d_{\psi}(x, y, t)} \sin \left( \frac{\pi l_{\psi}(x,y,t)}{L_{\psi}(t)} \right) e^{- \mathcal{K} t} = e^{- \mathcal{K} t} \leq 1,
\end{equation*} 
and $N > 1$. The item (iv) is a consequence of the fact that  $R(\bar{x}, \bar{y}, \bar {t})=N$ which follows from (iii); indeed if $d_{\psi}(x,y,t) \geq  d_0$, then by the definition of $N$ we have
\begin{equation*}
R(x,y,t) \leq \frac{L_{\psi}(t) e^{- \mathcal{K} t}}{\pi d_0} \leq \frac{L_{\psi}(0)}{\pi d_0}<N.
\end{equation*}
For the remainder of the proof, we fix $\overline x, \overline y$ so that iii), iv) above hold.

\textbf{Step 3.}  Our next aim is to modify $Z_N$ to a function $\widetilde Z_N$ which is smooth near $(\overline x, \overline y, \overline t)$, so we may apply the differential operators $\mathcal{L}_{\pm}$ of \eqref{def:Lpm} to $\widetilde Z_N$. First, we note that $\sin(\pi l_\psi(x,y,t)/L_\psi(t)$ is always smooth.

\begin{lemma}
The function  $\sin \left( {\pi l_{\psi}(x,y,t)}/{L_{\psi}(t)} \right)$ is smooth around $(\bar{x}, \bar{y}, \bar{t})$.
\end{lemma}
\begin{proof}
We first note that the intrinsic distance $l_{\psi}(x,y,t)$ is given by
\begin{equation*}
l_{\psi}(x,y,t) = \min \left\{ \int_x^y d \widehat{s}_t, L_{\psi}(t) - \int_x^y d \widehat{s}_t  \right\},
\end{equation*}
where we define an appropriate orientation on $S^1$, and $d\widehat{s}_t$ denotes the volume measure on $S^1$ with respect to $\gamma_t^*\widehat{g}$. In any case, we have
\begin{equation}\label{eq:sin0}
\sin \left( \frac{\pi l_{\psi}(x,y,t)}{L_{\psi}(t)} \right) = \sin \left( \frac{\pi \int_x^y d \widehat{s}_t}{L_{\psi}(t)} \right)
\end{equation}
since
\begin{equation*}
\sin \left( \frac{ \pi(L_{\psi}(t) - \int_x^y d \widehat{s}_t) }{L_{\psi}(t)} \right) = \sin \left( \pi - \frac{\pi \int_x^y d \widehat{s}_t}{L_{\psi}(t)} \right) = \sin \left( \frac{\pi \int_x^y d \widehat{s}_t}{L_{\psi}(t)} \right).
\end{equation*}
This implies smoothness of the function around $(\bar{x}, \bar{y}, \bar{t})$. 
\end{proof}

We now consider the function $d_\psi(x,y,t)$.  Since $\gamma_{\bar{t}}(\bar{x}), \gamma_{\bar{t}}(\bar{y})\in \mathcal{O}^{\rm reg}$ and $\widehat{d}(\gamma_{\bar{t}}(\bar{x}), \gamma_{\bar{t}}(\bar{y}))=d_\psi(\bar{x}, \bar{y}, \bar{t})<d_0\leq c_0$, Lemma \ref{lem:c0} shows that there is an orbifold chart $(\widetilde{U}, G,\varphi)$ such that $\gamma_{\bar{t}}(\bar{x}), \gamma_{\bar{t}}(\bar{y})\in \varphi(\widetilde{U})$ and 
$\widetilde{U}$ is convex with respect to $\varphi^*\widehat{g}$. We denote the induced metric $\varphi^*\widehat{g}$ on $\widetilde{U}$ by the same $\widehat{g}$. We note that if $\widetilde U$ is an orbifold chart for a cone point, then $d_\psi$ on $\varphi(\widetilde U)$ may not be a differentiable function (See Remark \ref{rem:derd}). We will therefore instead work on the chart $\widetilde U$.

For sufficiently small $\varepsilon>0$,  we consider the restricted maps $\gamma_i:=\gamma:  I_i\times [\bar{t}-\varepsilon,\bar{t}+\epsilon]\to \mathcal{O}^{\rm reg}$, where $I_i$ ($i=1,2$) are some closed intervals (independent of $t$) such that  $\bar{x}$ (resp. $\bar{y}$) is contained in the interior of $I_1$ (resp. $I_2$) and $\gamma_{t}(I_i)\subset \varphi(\widetilde{U})$ for any $t\in [\bar{t}-\varepsilon,\bar{t}+\epsilon]$. Since $\bar{x}\neq \bar{y}$, we may assume that $I_1\cap I_2=\emptyset$. Note that this also implies that $\gamma_{t}(I_1)\cap \gamma_{t}(I_2)=\emptyset$ because $\gamma_{t}$ is an embedding.

Next, we fix a lift $\Gamma^1: I_1\times [\bar{t}-\varepsilon,\bar{t}+\epsilon]\to \widetilde{U}$ of the restricted flow $\gamma_1$, and choose a lift $\Gamma^2: I_2 \times [\bar{t}-\varepsilon,\bar{t}+\epsilon]\to \widetilde{U}$ of the restricted flow  $\gamma_{2}$ so that 
\begin{align}\label{eq:dp1}
d_{\psi}(\bar{x}, \bar{y}, \bar{t})  =\widehat{d}(\Gamma_{\bar{t}}^1(\bar{x}), \Gamma_{\bar{t}}^2(\bar{y})),
\end{align}
where $\widehat{d}$ is the distance function on $(\widetilde{U}, \widehat{g})$. Note that there are $k \geq 2$ possible choices of lifts of $\gamma_{2}$, and  the inequality $d_{\psi}(\bar{x}, \bar{y}, \bar{t}) <\widehat{d}(\Gamma_{\bar{t}}^1(\bar{x}), \Gamma_{\bar{t}}^2(\bar{y}))$ may hold for some other choices of lifts.

On the other hand, for any $(x,y,t)\in I_1\times I_2\times[\bar{t}-\varepsilon,\bar{t}+\epsilon]$, it holds that 
\begin{align}\label{eq:dp2}
d_{\psi}(x, y, t)  \leq \widehat{d}(\Gamma_{t}^1(x), \Gamma_{t}^2(y)).
\end{align}
Moreover, both lifts $\Gamma_t^1$ and $\Gamma_t^2$ are $\widetilde{\psi}$-CSF on $\widetilde{U}$, where $\widetilde{\psi}=\varphi^*\psi$. We thus define the lifted distance function $\widetilde{d_\psi}: I_1\times I_2\times [\bar{t}-\varepsilon,\bar{t}+\epsilon]\to \R$ by 
\[
\widetilde{d_\psi}(x,y,t):=\widehat{d}(\Gamma_{t}^1(x), \Gamma_{t}^2(y)).
\]
Note that $\widetilde{d_\psi}$ is a differentiable (smooth) function by Lemma \ref{lem:derd}.  We may then define $\widetilde{Z}_N: I_1\times I_2\in [\bar{t}-\varepsilon,\bar{t}+\epsilon]\to \R$ by
\begin{equation*}
\widetilde{Z}_N (x, y, t) 
:= N \widetilde{d_\psi}(x, y, t) - \frac{L_{\psi}(t)}{\pi} \sin \left( \frac{\pi l_{\psi}(x,y, t)}{L_{\psi}(t)} \right) e^{- \mathcal{K} t }.
\end{equation*}
By the above discussion, $\widetilde{Z}_N$ is differentiable. Moreover, by \eqref{eq:dp1} and \eqref{eq:dp2}, we have  $Z_N (\bar{x}, \bar{y}, \bar{t}) = \widetilde{Z}_N (\bar{x}, \bar{y}, \bar{t})$ and $Z_N \leq \widetilde{Z}_N$.  In particular, it is easy to see that the properties (i)--(iii) still hold for the replaced function $\widetilde{Z}_N$.

\textbf{Step 4.} We now work with the smooth function $\widetilde Z_N$, and show that our assumptions imply that $\mathcal{L}_{\pm}\widetilde Z \leq 0$ at $(\overline x, \overline y, \overline t)$. In the following, as long as we fix the time parameter $t=\bar{t}$, we parametrise $\gamma(\cdot, \bar{t})$ by the arc length parameter $\widehat{s}$ with respect to the metric $\gamma_{\bar{t}}^* \widehat{g}$. Moreover, we fix the orientation of $S^1$ so that it satisfies
\begin{equation}\label{eq:orient}
0<l_{\psi}(\bar{x}, \bar{y}, \bar{t}) = \widehat{s}(\bar{y}) - \widehat{s}(\bar{x}) \leq \frac{L_{\psi}(\bar{t})}{2},
\end{equation}
where $\widehat{s}(x)$ denotes the parameter corresponding to $x\in S^1$.
Using this parameter, we may regard the function $l_\psi(\cdot,\cdot,\bar{t}): S^1\times S^1\to \R$ as a function of two variables $\widehat{s}_1$ and $\widehat{s}_2$, and  \eqref{eq:sin0} shows that 
  \begin{align}\label{eq:sin}
\sin \left( \frac{\pi l_{\psi}(\widehat{s}_1,\widehat{s}_2,\bar{t})}{L_{\psi}(\bar t)}\right) =\sin \left(  \frac{\pi (\widehat{s}_2 - \widehat{s}_1)}{L_{\psi}(\bar{t})} \right).
\end{align}
for any $(\widehat{s}_1,\widehat{s}_2)$ sufficiently close to $(\widehat{s}(\bar{x}),\widehat{s}(\bar{y}))$. 

Now, since $\widetilde{Z}_N$ attains a new minimum at $(\bar{x}, \bar{y}, \bar{t})$, at this point we have
\begin{align*}
&\frac{\partial \widetilde{Z}_N}{\partial t} \leq 0, \quad \frac{\partial \widetilde{Z}_N}{\partial \widehat{s}_1}  (\bar{x}, \bar{y}, \bar{t}) = \frac{\partial \widetilde{Z}_N}{\partial \widehat{s}_2}  (\bar{x}, \bar{y}, \bar{t}) = 0, \quad \frac{\partial^2 \widetilde{Z}_N}{\partial \widehat{s}_1^2} \frac{\partial^2 \widetilde{Z}_N}{\partial \widehat{s}_2^2} - \left( \frac{\partial^2 \widetilde{Z}_N}{\partial \widehat{s}_1 \partial \widehat{s}_2} \right)^2 \geq 0.
\end{align*}
We set
\begin{equation*}
e_1 := e^{\widetilde{\psi}(p_1)} = e^{\widetilde{\psi}( \Gamma^1_{\bar{t}}(\bar{x}) )}, \quad e_2 := e^{\widetilde{\psi}(p_2)} = e^{\widetilde{\psi}( \Gamma^2_{\bar{t}}(\bar{y}) )}.
\end{equation*}
Then, we have 
\begin{equation*}
e_1^2 \frac{\partial^2 \widetilde{Z}_N}{\partial \widehat{s}_1^2} + e_2^2 \frac{\partial^2 \widetilde{Z}_N}{\partial \widehat{s}_2^2} \geq 2 e_1 e_2 \sqrt{ \abs{ \frac{\partial^2 \widetilde{Z}_N}{\partial \widehat{s}_1^2} \frac{\partial^2 \widetilde{Z}_N}{\partial \widehat{s}_2^2}  } } \geq 2 e_1 e_2 \abs{ \frac{\partial^2 \widetilde{Z}_N}{\partial \widehat{s}_1 \partial \widehat{s}_2} }.
\end{equation*}
Thus, using the operators $\mathcal{L}_{\pm}$ given in \eqref{def:Lpm}, we see
\begin{equation}\label{eq:pmz}
\mathcal{L}_{\pm}( \widetilde{Z}_N)\biggm\vert_{(\bar{x},\bar{y},\bar{t})} 
= \frac{\partial \widetilde{Z}_N}{\partial t} - \left( e_1^2 \frac{\partial^2 \widetilde{Z}_N}{\partial \widehat{s}_1^2} + e_2^2 \frac{\partial^2  \widetilde{Z}_N }{\partial \widehat{s}_2^2} \pm 2 e_1 e_2 \frac{\partial^2  \widetilde{Z}_N }{\partial \widehat{s}_1 \partial \widehat{s}_2} \right) \leq 0,
\end{equation}
where the upper (resp. lower) signs correspond to $\mathcal{L}_+$ (resp. $\mathcal{L}_-$).

\textbf{Step 5.} Finally, we compute $\mathcal{L}_{\pm}\widetilde{Z}_N$ at $(\overline x, \overline y, \overline t)$ directly and show that this leads to a contradiction with \eqref{eq:pmz}.  We first consider the derivatives of the second term in $\widetilde{Z}_N$, namely, we set
\[
S_\psi(x,y,t):=\frac{L_{\psi}(t)}{\pi} \sin \left( \frac{\pi l_{\psi}(x,y, t)}{L_{\psi}(t)} \right) e^{- \mathcal{K} t } \, = \, \frac{L_{\psi}(t)}{\pi}\sin \left(  \frac{\pi (\widehat{s}_2(y) - \widehat{s}_1(x))}{L_{\psi}(\bar{t})} \right)e^{-\mathcal{K}t}.
\]

\begin{lemma}\label{lem:LS}
At $(\bar{x}, \bar{y}, \bar{t})$, we have
\begin{align*}
\mathcal{L}_{\pm}(S_\psi)\leq -\Big{(}\mathcal{K} -4D^2\frac{\pi^2}{L_\psi^2}\Big{)}\cdot \frac{L_\psi}{\pi}\sin\left( \frac {\pi l_\psi}{L_{\psi}}\right) e^{- \mathcal{K} \bar{t} },
\end{align*} 
where $D:={\rm sup}_{\widetilde{U}}e^{\widetilde{\psi}}$.
\end{lemma}
\begin{proof}
Recall that we have the expression \eqref{eq:sin0}. Using this and the facts
\begin{align*}
\frac{\p}{\p t}\int_x^yd\widehat{s}_t
= - \int_x^y \kappa_{\psi}^2 d \widehat{s}_{\bar{t}},\quad 
\frac{\partial L_{\psi}}{\partial t} (\bar{t}) 
= - \int_{S^1} \kappa_{\psi}^2 d \widehat{s}_{\bar{t}},
\end{align*}
which follow from \eqref{eq:fv}, a direct computation shows that 
\begin{align*}
\frac{\partial S_\psi}{\partial t} 
&=- e^{- \mathcal{K} \bar{t}}  \left\{  \frac{ 1 }{\pi}  \sin \left(\frac{\pi l_{\psi}}{L_{\psi}}\right) - \frac{l_{\psi}}{ L_{\psi}} \cos \left(\frac{\pi l_{\psi}}{L_{\psi}}\right)  \right\} \int_{\mathbb{S}^1} \kappa_{\psi}^2 d \widehat{s}_{\bar{t}}  \\
&- e^{- \mathcal{K} \bar{t}} \cos \left(\frac{\pi l_{\psi}}{L_{\psi}}\right) \int_x^y \kappa_{\psi}^2 d \widehat{s}_{\bar{t}}-\mathcal{K} \cdot \frac{L_{\psi}}{\pi} \sin \left(  \frac{\pi l_{\psi}}{L_{\psi}}  \right)e^{- \mathcal{K} \bar{t}}\\
&\leq -\mathcal{K} \cdot \frac{L_{\psi}}{\pi} \sin \left(  \frac{\pi l_{\psi}}{L_{\psi}}  \right)e^{- \mathcal{K} \bar{t}}
\end{align*}
since we have
\begin{equation*}
\frac{1 }{\pi}  \sin \left(\frac{\pi l_{\psi}}{L_{\psi}} \right) - \frac{ l_{\psi} }{ L_{\psi} } \cos \left(\frac{\pi l_{\psi}}{L_{\psi}} \right)
\geq 0,\quad  \cos \left(\frac{\pi l_{\psi}}{L_{\psi}}\right)\geq 0
\end{equation*}
owing to the assumption that  $0<l_\psi/L_\psi\leq 1/2$ at $(\bar{x}, \bar{y}, \bar{t})$ (see \eqref{eq:orient}). 

Moreover, by using  the expression \eqref{eq:sin}, we have
\begin{align}\label{eq:dS}
\frac{\partial  S_\psi}{\partial \widehat{s}_1}  
=  -\cos \left( \frac{\pi (\widehat{s}_2-\widehat{s}_1)}{L_{\psi}}\right) e^{- \mathcal{K} \bar{t} } ,
\qquad 
\frac{\partial S_\psi}{\partial \widehat{s}_2} 
= 
\cos \left( \frac{\pi(\widehat{s}_2-\widehat{s}_1)}{L_{\psi}}\right) e^{- \mathcal{K} \bar{t} }.
\end{align}
and hence, we see
\begin{align*}
e_1^2 \frac{\partial^2 S_\psi}{\partial \widehat{s}_1^2} + e_2^2 \frac{\partial^2  S_\psi }{\partial \widehat{s}_2^2} \pm 2 e_1 e_2 \frac{\partial^2 S_\psi}{\partial \widehat{s}_1 \partial \widehat{s}_2}
\,&=\,-(e_1\mp e_2)^2\frac{\pi}{L_\psi}\sin\left( \frac {\pi l_\psi}{L_{\psi}}\right) e^{- \mathcal{K} \bar{t} }
\,\geq\, - 4D^2\frac{\pi}{L_\psi}\sin\left( \frac {\pi l_\psi}{L_{\psi}}\right) e^{- \mathcal{K} \bar{t} }
\end{align*}
at $(\bar{x}, \bar{y}, \bar{t})$, where we again used $0 < l_\psi/L_\psi\leq 1/2$. Thus, we obtain the lemma. 
\end{proof}

For the derivatives of the first term $N\widetilde{d_\psi}$  in $\widetilde{Z}_N$, we can apply Lemmas \ref{lem:J} and \ref{lem:formulas} since we choose the constant $d_0$ so that $d_\psi(\bar{x},\bar{y},\bar{t})<d_0<1/2{\widehat{K}_{\mathcal{O}}^{1/2}}$.
 Moreover, by a similar argument given in the proof of Proposition \ref{prop:key}, we obtain the following estimate:
By Lemma \ref{lem:formulas}, we see that 
\[
\frac{\partial  \widetilde{Z}_N}{\partial \widehat{s}_1}=-N\widehat{g}(\widehat{T}_1,\tau_1) -\frac{\partial  S_\psi}{\partial \widehat{s}_1}  ,\quad\frac{\partial  \widetilde{Z}_N}{\partial \widehat{s}_2}=N\widehat{g}(\widehat{T}_2,\tau_2) -\frac{\partial  S_\psi}{\partial \widehat{s}_2}. 
\]
Since $\partial  \widetilde{Z}_N/\partial \widehat{s}_1=\partial  \widetilde{Z}_N/\partial \widehat{s}_2=0$ at $(\bar{x},\bar{y},\bar{t})$ and we have \eqref{eq:dS}, we see that 
\[
\widehat{g}(\widehat{T}_1, \tau_1) = \widehat{g}(\widehat{T}_2, \tau_2)=\frac{1}{N}\cos \left( \frac{\pi l_\psi}{L_{\psi}}\right) e^{- \mathcal{K} \bar{t}} \geq 0.
\]
This also implies that we have either 
\[\widehat{g}(\widehat{T}_1, \nu_1) =  \widehat{g}(\widehat{T}_2, \nu_2)\quad  \textup{or}\quad \widehat{g}(\widehat{T}_1, \nu_1) = -\widehat{g}(\widehat{T}_2, \nu_2)\]
 (see subsection \ref{ssec:dist} for the definition of $\{\tau_i,\nu_i\}$). We first suppose the case when $\widehat{g}(\widehat{T}_1, \nu_1) =  \widehat{g}(\widehat{T}_2, \nu_2)$. Then, by a similar  argument given in the proof of Proposition \ref{prop:key} (see \eqref{eq:Lp}), we have 

\begin{align}\label{eq:Lp2}
\mathcal{L}_{+}(N\widetilde{d_\psi}) \bigg\vert_{ (\bar{x}, \bar{y}, \bar{t})} 
&\geq \left(-2D^2 \widehat{K}_{\mathcal{O}}- 2 D^2 E^2 \right) N\widetilde{d_\psi},\\
\implies \mathcal{L}_{+} \widetilde{Z}_N \bigg\vert_{ (\bar{x}, \bar{y}, \bar{t})} 
&\geq \left( \mathcal{K} - 2D^2 \widehat{K}_{\mathcal{O}} - 2 D^2 E^2- 4D^2 \frac{\pi^2}{L_\psi^2} \right) \frac{L_{\psi}}{\pi} \sin \left(  \frac{\pi l_{\psi}}{L_{\psi}}   \right) e^{- \mathcal{K} \bar{t}}.
\end{align}
where $D:={\rm sup}_{\widetilde{U}}e^{\widetilde{\psi}}$ and $E:={\rm sup}_{\widetilde{U}}|\widehat{\nabla}{\widetilde{\psi}}|_{\widetilde{\psi}}$.

Moreover, since we assume that $L_\psi(T_{\max})=\lim_{t\to T_{\max}}L_\psi(t)>0$, there exists a constant $\delta_0$ such that $L_\psi(\bar{t})\geq \delta_0$.  Therefore, if we  take the constant $\mathcal{K}$ sufficiently large so that 
\begin{equation*}
\mathcal{K} 
> 2 D^2 \widehat{K}_{\mathcal{O}} + 2 D^2 E^2 + 4 D^2 \frac{\pi^2}{\delta_0^2},
\end{equation*}
then we see $\mathcal{L}_+\widetilde{Z}_N>0$ at $(\bar{x}, \bar{y}, \bar{t})$.
This, however, contradicts \eqref{eq:pmz}. The case that $\widehat{g}(\widehat{T}_1, \tau_1) =- \widehat{g}(\widehat{T}_2, \tau_2)$ also leads to a contradiction by using $\mathcal{L}_{-} \widetilde{Z}_N$. This completes the proof.

\section{Estimates for a normal graph}\label{app:norgra}

In this appendix, we give  a proof of the following lemma, which is used in the proof of Theorem \ref{thm:main2}.

\begin{lemma}\label{lem:ba}
Let $(P,g)$ be a smooth compact Riemannian manifold and $\Sigma$ a compact hypersurface in $P$ with globally defined unit normal vector field $N$. For a smooth function $u\in C^\infty(\Sigma)$, we consider the normal graph over $\Sigma$
\[
G_u: \Sigma\to P,\quad G_u(p):={\rm exp}_p(u(p)N_p). 
\]
Let $\mathcal{U}_R=\{u\in C^{2,\alpha}(\Sigma)\mid \|u\|_{C^{2,\alpha}}<R\}$ be a small neighbourhood of the origin in $C^{2,\alpha}$ such that $G_u$ is an embedding for any $u\in \mathcal{U}_R$. 
We denote the second fundamental form of the normal graph $\Sigma_u=G_u(\Sigma)$ in $P$  by $B_u$.  

Then, for any $k\geq 2$, there exists a polynomial $P_k$ depending only on $\Sigma$, $\mathcal{U}_R$ and $(P,g)$ such that 
\[
\|u\|_{C^k(\Sigma)}\leq P_k\left(
1,
\|B_u\|,
\|\nabla^uB_u\|,
\ldots,
\|(\nabla^u)^{k-2}B_u\|
\right),\quad {\rm for}\quad \forall u\in \mathcal{U}_R,
\]
where ${\nabla}^u$ is the normal connection on $\Sigma_u$ with respect to the induced metric $G_u^*g$, and $\|(\nabla^u)^mB_u\|$ denotes the $C^0$-norm with respect to $G_u^*g$.
\end{lemma}

\begin{proof}
The claim is obvious for $k=2$ since $u\in \mathcal{U}_R$. We prove the claim for $k=3$. 

Let $g_0$ be the induced metric on $\Sigma$ and denote by $\nabla$ the Levi-Civita connection of $g_0$. We write
$
g_u:=G_u^*g
$
and denote by $\nabla^u$ the Levi-Civita connection of $g_u$ on
$\Sigma$. 

We may assume that any graph $\Sigma_u$ for $u\in \mathcal{U}_R$ is contained in a fixed tubular neighbourhood $\mathcal{N}$ of $\Sigma$. 
Let
\[
\Phi:\Sigma\times(-\varepsilon,\varepsilon)\to \mathcal{N},
\quad
\Phi(p,r)=\exp_p(rN_p)
\]
be the corresponding diffeomorphism.

We derive the expression of the second fundamental form $B_u$ for $\Sigma_u$. Let $(x_1,\ldots, x_n)$ be a local coordinate of $\Sigma$ and we take the corresponding Fermi coordinate $(x_1,\ldots, x_n,r)$ on $\mathcal{N}$. We put $\p_i(p,r):=\p/\p x_i|_{(p,r)}$ and $\p_r(p,r):=\p/\p r|_{(p,r)}$. Note that $\langle \p_i,\p_r\rangle=0$ and $\langle \p_r,\p_r\rangle=1$. We put
\[
g_{ij}(p,r):=\langle\p_i,\p_j\rangle_{(p,r)}
\]
and denote by $(g^{ij}(p,r))$ the inverse matrix of $(g_{ij}(p,r))$.

We see that the tangent space of $\Sigma_u$ is spanned by
\[
e_i(p,u(p)):=(G_u)_*\p_i(p,0)
=
\p_i+(\p_iu)\p_r,\quad i=1,\ldots, n,
\]
and a unit normal vector field $N_u$ of  the graph $\Sigma_u$ is given by
\[
N_u
=
\frac{1}{n_u}
\left(
\p_r
-
g^{ij}(p,u(p))(\p_j u)\p_i
\right),
\quad
n_u
=
\sqrt{
1+
g^{ij}(p,u(p))(\p_i u)(\p_j u)
}.
\]
Note that $N_u$ (and $n_u$) smoothly depends on $p$, $u$ and $\nabla u$.

In order to compute $\overline\nabla_{e_i}e_j$, we extend $u$ as a function on $\mathcal{N}$ satisfying $u(p,r)=u(p)$. Then $\p_r u=0$. Since
$
\overline\nabla_{\partial_r}\partial_r=0
$
and
$
\overline\nabla_{\partial_r}\partial_i
=
\overline\nabla_{\partial_i}\partial_r,
$
where $\onab$ is the Levi-Civita connection of $(P,g)$, 
we obtain
\[
\overline\nabla_{e_i}e_j
=
\overline\nabla_{\partial_i}\partial_j
+
(\p_i u)\overline\nabla_{\partial_j}\partial_r
+
(\p_j u)\overline\nabla_{\partial_i}\partial_r
+
(\p_i\p_j u)\partial_r.
\]
Consequently,
\begin{align*}
\langle B_u(e_i,e_j), N_u\rangle
&=
\frac{1}{n_u}
(\p_i\p_j u)
+
\langle
\overline\nabla_{\partial_i}\partial_j
+
(\p_i u)\overline\nabla_{\partial_j}\partial_r
+
(\p_j u)\overline\nabla_{\partial_i}\partial_r,
\,
N_u
\rangle
\end{align*}
and this shows that we have
\begin{align}\label{eq:nab2}
(\nabla^2 u)(\p_i,\p_j)=\langle B_u(e_i,e_j), n_uN_u\rangle+O_{ij}(u,\nabla u, \overline\Gamma),
\end{align}
where $O_{ij}$ is a smooth function of $u$, $\nabla u$ and the Christoffel symbols of $\overline{\nabla}$.  

We consider the derivative of  \eqref{eq:nab2} in the direction of $\p_k$. A direct computation shows that 
\begin{align}\label{eq:nab3}
\p_k(\nabla^2 u)(\p_i,\p_j)
={}&
\left\langle
(\nabla^uB_u)(e_k,e_i,e_j),n_uN_u
\right\rangle \\
&+
(\Gamma^u)_{ki}^l
\left\langle
B_u(e_l,e_j),n_uN_u
\right\rangle \nonumber\\
&+
(\Gamma^u)_{kj}^l
\left\langle
B_u(e_i,e_l),n_uN_u
\right\rangle \nonumber\\
&+
\langle B_u(e_i,e_j),\overline{\nabla}_{e_k}(n_uN_u)\rangle
+e_k(O_{ij}(u,\nabla u, \overline{\Gamma})), \nonumber
\end{align}
where $\nabla_{e_k}^ue_i=(\Gamma^u)_{ki}^le_l$.  Since the induced metric $g_u$ is given by
\[
(g_u)_{ij}
=
g(e_i,e_j)
=
g_{ij}+(\p_i u)(\p_j u),
\]
 the Christoffel symbols $(\Gamma^u)_{ki}^l$ of $\nabla^u$
depend smoothly on $p, u, \nabla u$ and $\nabla^2u$. Moreover, we see that the last two terms of the RHS of \eqref{eq:nab3} depend on $p, u, \nabla u$ and $\nabla^2u$, $\overline{\Gamma}$ and the derivatives of $\overline{\Gamma}$. Therefore, we may write 
\begin{align}\label{eq:nab32}
\p_k(\nabla^2 u)(\p_i,\p_j)=
\left\langle(\nabla^uB_u)(e_k,e_i,e_j),n_uN_u\right\rangle
+O_{ijk}(B_u,u,\nabla u, \nabla^2 u, \overline{\Gamma}, \p\overline{\Gamma})
\end{align}
for any $i,j,k=1,\ldots, n$.

On the other hand, we have
\begin{align}\label{eq:nab33}
(\nabla^3u)(\p_k,\p_i,\p_j)
=
\p_k\bigl((\nabla^2u)(\p_i,\p_j)\bigr)
-
(\nabla^2u)(\nabla_{\p_k}\p_i,\p_j)
-
(\nabla^2u)(\p_i,\nabla_{\p_k}\p_j).
\end{align}
Since $\nabla$ is a fixed connection on the compact manifold $\Sigma$, \eqref{eq:nab2}--\eqref{eq:nab33} imply that there exists a polynomial $P_3$ depending only on $\Sigma$, $\mathcal{U}_R$ and $(P,g)$ such that
\[
\|u\|_{C^3(\Sigma)}
\leq
P_3\left(
1,
\|B_u\|,
\|\nabla^uB_u\|
\right),
\quad \forall u\in \mathcal{U}_R.
\]
This proves the claim in the case when $k=3$. 

We obtain the desired claim for every $k\geq 2$ by an inductive argument, and thus, we omit it.
\end{proof}

\bibliographystyle{abbrv}

\bibliography{main}

\end{document}